\documentclass[11pt]{article}

\usepackage{amsfonts,amsthm,amssymb,amsopn,amsmath,mathrsfs,dsfont,mathtools}
\usepackage{hyperref}
\usepackage{graphicx}
\usepackage[T1]{fontenc}
\usepackage{verbatim}
\usepackage{authblk}
\usepackage{color}
\usepackage{enumitem}
\usepackage{stmaryrd}
\usepackage{cancel}

\newcounter{subsection1}[section]
\newtheorem{lemma}[subsection1]{Lemma}
\newtheorem{prop}[subsection1]{Proposition}
\newtheorem{theorem}[subsection1]{Theorem}
\newtheorem{remark}[subsection1]{Remark}

\newcounter{mysubequations}

\renewcommand{\themysubequations}{(\roman{mysubequations})}

\newcommand{\mysubnumber}{\refstepcounter{mysubequations}\themysubequations}

\newcommand{\E}[1]{{\mathbb E}\left[#1\right]}

\newcommand{\p}[1]{{\mathbb P}\left(#1\right)}
\newcommand{\psub}[2]{{\mathbb P}_{#1}\left(#2\right)}
\newcommand{\Esub}[2]{{\mathbb E}_{#1}\left[#2\right]}

\long\def\new#1{{#1}}
\long\def\newnew#1{{#1}}

\newcounter{assumptions}
\numberwithin{equation}{section} 
\numberwithin{subsection1}{section}

\begin{document}

\def\P{\mathbb{P}}
\def\R{\mathbb{R}} %
\def\C{\mathbb{C}} %
\def\N{\mathbb{N}}
\def\Q{\mathbb{Q}}
\def\Z{\mathbb{Z}}
\def\1{\mathds{1}}
\def\to{\rightarrow}
\def\F{\mathscr F}
\newcommand{\Vect}{\operatorname{Vect}}
\newcommand{\dom}{\operatorname{dom}}
\newcommand{\supp}{\operatorname{supp}}
\newcommand{\thr}{\operatorname{thr}}

\def\aminus{a_{-}}
\def\aplus{a_{+}}

\def\ball{\mathcal{B}}

\def\kmax{\overline k}
\def\ep{\varepsilon}

\allowdisplaybreaks

\author{Pascal Maillard\thanks{Institut de Mathématiques de Toulouse, CNRS, UMR5219, Université de Toulouse, 118 route de Narbonne, F-31062 Toulouse cedex 09, France. E-mail: \texttt{Pascal.Maillard at math dot univ-toulouse dot fr}. Partially supported by ANR grant ANR-20-CE92-0010-01.}, Sarah Penington\thanks{Department of Mathematical Sciences, University of Bath, Claverton Down, Bath, UK. E-mail: \texttt{S.Penington at bath dot ac dot uk}. Supported by a Royal Society University Research Fellowship.}}

\title{Branching random walk with non-local competition}

\date{March 21, 2024}
\maketitle

\vspace{-0.35cm}

\begin{abstract}
    We study the Bolker-Pacala-Dieckmann-Law (BPDL) model of population dynamics in the regime of large population density. The BPDL model is a particle system in which particles reproduce, move randomly in space, and compete with each other locally. We rigorously prove global survival as well as a shape theorem describing the asymptotic spread of the population, when the population density is sufficiently large. In contrast to most previous studies, we allow the competition kernel to have an arbitrary, even infinite range, whence the term \emph{non-local competition}. This makes the particle system non-monotone and of infinite-range dependence, meaning that the usual comparison arguments break down and have to be replaced by a more hands-on approach. Some ideas in the proof are inspired by works on the non-local Fisher-KPP equation, but the stochasticity of the model creates new difficulties.
\end{abstract}

\medskip 


\noindent \textbf{MSC2020 subject classifications:} 60K35; 60J80; 92D25.

\section{Introduction}

\subsection{Definition of the model} \label{subsec:modeldefn}

In this article, we study the Bolker-Pacala-Dieckmann-Law (BPDL) model, which we also refer to as \emph{branching random walk with non-local competition} (BRWNLC). The BRWNLC can be regarded as a Markov process $(\xi_t)_{t\geq 0}$ taking values in $\N_0^{\Z^d}$, with the interpretation that $\xi_t(x)$ is the number of particles at site $x$ at time $t$, for $t\ge0$ and $x\in\Z^d$. 

\new{Take $\gamma >0$.
Let $p:\Z^d\to [0,1]$ satisfy the following assumptions:
\begin{equation} \label{eq:p_assumptions}
    \begin{aligned}\setcounter{mysubequations}{0}
      \mysubnumber\quad &p(x)\ge 0 \; \; \forall x\in \Z^d,\\
      \mysubnumber\quad &\sum_{x\in \Z^d}p(x)=1,\\
      \mysubnumber\quad &\exists R_1<\infty \text{ s.t. }p(x)=0\; \forall x\in \Z^d\text{ with }\|x\|\ge R_1,\\
      \mysubnumber\quad &\{x\in\Z^d:p(x)>0\}\text{ is a spanning set of the vector space }\R^d.
    \end{aligned}
\end{equation}
(Here, and throughout the article, we write $\|\cdot \|$ to denote the $\ell^2$ or Euclidean norm.)
In words, $p$ is a finite range jump kernel on $\Z^d$ whose support is a spanning set of $\R^d$.
Let $\Lambda :\Z^d \to [0,\infty)$ with $\sum_{x\in \Z^d}\Lambda(x)\in (0,\infty)$.

We now define the BRWNLC with jump rate $\gamma$, jump kernel $p$ and competition kernel $\Lambda$.}
Starting from an initial configuration $\xi_0$ consisting of a finite number of particles, the model evolves as follows:
\begin{itemize}
\item Particles \emph{branch} (or \emph{reproduce}) at (constant) rate 1. 

That is, for each $x\in \Z^d$, the transition $\xi_t \to \xi_t +\delta_x$ occurs with rate $\xi_t(x)$ (i.e.~such a transition occurs at each jump time in an inhomogeneous Poisson process with rate $\xi_t(x)$ at time $t$).

\item Particles \emph{jump} at (constant) rate $\gamma > 0$ \new{with jump kernel $p$}.

\new{That is,} for each $x\in \Z^d$ and $y\in \Z^d$, the transition $\xi_t \to \xi_t-\delta_x+\delta_{x+y}$ occurs with rate $\gamma \xi_t(x) p(y)$.

\item Particles \emph{compete} with each other \new{with competition kernel $\Lambda$}.

A particle at $x$ gets killed with rate 
\begin{equation} \label{eq:Ktdefn}
K_t(x) \coloneqq \xi_t* \Lambda (x)=\sum_{y\in \Z^d}\xi_t(x-y)\Lambda (y),
\end{equation}
i.e.~for each $x\in \Z^d$, the transition $\xi_t\to \xi_t-\delta_x$ occurs with rate $K_t(x)\xi_t(x)$.
\end{itemize}
For $\xi\in \N_0 ^{\Z^d}$, write $\mathbb P_\xi$ for the probability measure under which $(\xi_t)_{t\ge 0}$ is a BRWNLC with $\xi_0=\xi$.
A precise construction of the model is given in Section~\ref{sec:construction}. 

We further define
\begin{align}
\label{eq:Ndefn}
N &=N(\Lambda)\coloneqq \frac{1}{\sum_{x\in \Z^d}\Lambda (x)}\\
\label{eq:mudefn}
\text{and} \quad \mu &=\mu(p) \coloneqq\gamma \sum_{y\in \Z^d}yp(y) \in \R^d.
\end{align}
The parameter $N$ should be interpreted as the local population density or carrying capacity (we do not need to assume that $N$ is an integer). We will be interested in the regime where $N$ is large. One might think of a sequence of competition kernels $\Lambda_N = \frac1N \Lambda_1$ for some fixed $\Lambda_1$ with $\sum_{x\in \Z^d}\Lambda_1 (x) = 1$; however, our results hold in greater generality.

\subsection{Main results}

We state two theorems. 
\new{The first theorem concerns the global survival of the BRWNLC with competition kernel $\Lambda$ when $N=N(\Lambda)$ is sufficiently large and $\Lambda$ satisfies a suitable exponential decay:
\begin{theorem}[Global survival]
\label{th:survival}
For $\delta, \gamma, \lambda, R >0$ and $p:\Z^d\to [0,1]$ satisfying~\eqref{eq:p_assumptions}, 
there exist $\kappa=\kappa(\gamma, p)>0$ and $N_0=N_0(\delta,\gamma,p,\lambda,R)>0$ such that the following holds.
Suppose $\Lambda :\Z^d \to [0,\infty)$ with $\sum_{x\in \Z^d}\Lambda(x) \in (0,\infty)$ and
\begin{equation} \label{eq:Lambda_assumptions}
    \begin{aligned}\setcounter{mysubequations}{0}
    \mysubnumber\quad &\sum_{x\in \Z^d}\Lambda (x) \le \frac 1 {N_0},\\
      \mysubnumber\quad &\frac{\Lambda(0)}{\sum_{x\in \Z^d}\Lambda (x)} \ge \lambda,\\
      \mysubnumber\quad &\frac{\Lambda(x)}{\sum_{y\in \Z^d}\Lambda (y)} \le e^{-\kappa \|x\|} \quad \forall x\in \Z^d \text{ with }\|x\|\ge R.\\      
    \end{aligned}
\end{equation}
Let $(\xi_t)_{t\ge 0}$ be a BRWNLC with jump rate $\gamma$, jump kernel $p$ and competition kernel $\Lambda$, as defined in Section~\ref{subsec:modeldefn}.
Then for any $\xi\in \N_0^{\Z^d}$ consisting of a finite number of particles with $\xi\not\equiv 0$,
$$
\psub{\xi}{\sum_{x\in \Z^d} \xi_t(x)>0 \; \forall t\ge 0}> 1-\delta.
$$
\end{theorem}
}

\new{We remark that when $\delta >0$ is chosen to be small, $N_0$ is very large and so the condition (i) in the conditions on $\Lambda$ in~\eqref{eq:Lambda_assumptions} requires that $\sum_{x\in \Z^d}\Lambda (x)$ is very small, or in other words $N$, as defined in~\eqref{eq:Ndefn}, is very large.
In particular, this condition means that the killing rate (given in~\eqref{eq:Ktdefn}) is very small when there are not many particles in the system, which in turn means that the probability of survival can be close to 1 even when the system begins with a small number of particles. 
Condition (ii) on $\Lambda$ ensures that the normalised competition kernel $\Lambda(\cdot)/\sum_{x\in \Z^d}\Lambda (x)$ is strictly positive at $0$, or in other words, there is some on-site competition in the particle system.
Condition (iii) requires exponential decay for the normalised competition kernel.
}

The next theorem describes the asymptotic spread of the BRWNLC \new{when started with a single particle at the origin}.
\new{Take $\gamma>0$ and $p:\Z^d\to [0,1]$ satisfying~\eqref{eq:p_assumptions}.}
Let $(X_t)_{t\ge 0}$ denote a continuous-time random walk started at $0$ with jump rate $\gamma$ and jump kernel $p$.
The Cramér transform of $X_1$ is expressed for $u\in\R^d$ as
\begin{equation}
\label{eq:Cramer}
\E{e^{\langle u, X_1 \rangle}} = \exp\bigg(\gamma\sum_{x\in \Z^d}p(x) \left(e^{ \langle u, x \rangle }-1\right)\bigg) < \infty,
\end{equation}
since the jump kernel $p$ is of finite range by assumption. The rate function is expressed for $v\in\R^d$ as
\begin{equation}  \label{eq:Idefn}
I(v)\new{=I(v;\gamma,p)}=\sup_{u \in \R^d} \left( \langle v, u \rangle - \log \E{e^{\langle u, X_1 \rangle}} \right) =\sup_{u \in \R^d} \bigg( \langle v, u \rangle - \gamma \sum_{x\in \Z^d}p(x) \left(e^{ \langle u, x \rangle }-1\right) \bigg).
\end{equation}
By the finiteness of the Cramér transform, $I$ is a good convex rate function, i.e.~it is convex and all sub-level sets $\{I \le a\}$, $a\in\R$, are compact \cite[Lemma~2.2.31]{DemboZeitouni}. %
In particular, the set
\begin{equation} \label{eq:I1defn}
\mathcal I_1 \new{=\mathcal I_1(\gamma,p)}:= \{x\in\R^d: I(x\new{;\gamma,p})\le 1\}
\end{equation}
is compact and convex. See Figure~\ref{fig:shape} for examples.

\begin{figure}
\begin{center}
\includegraphics[scale=1]{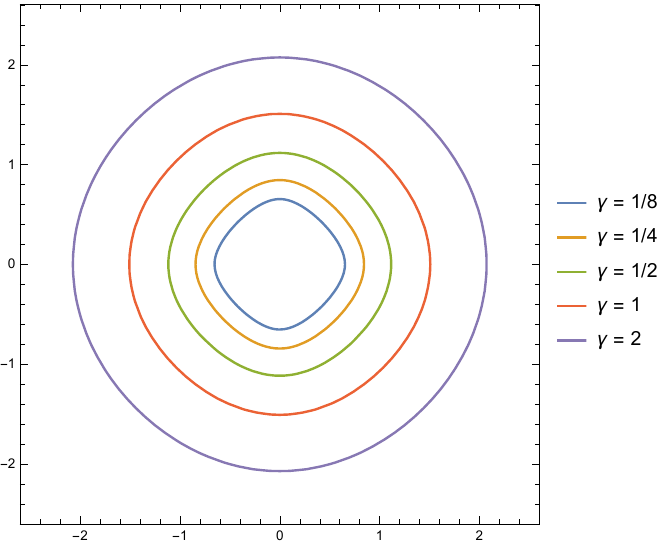}

\includegraphics[scale=1]{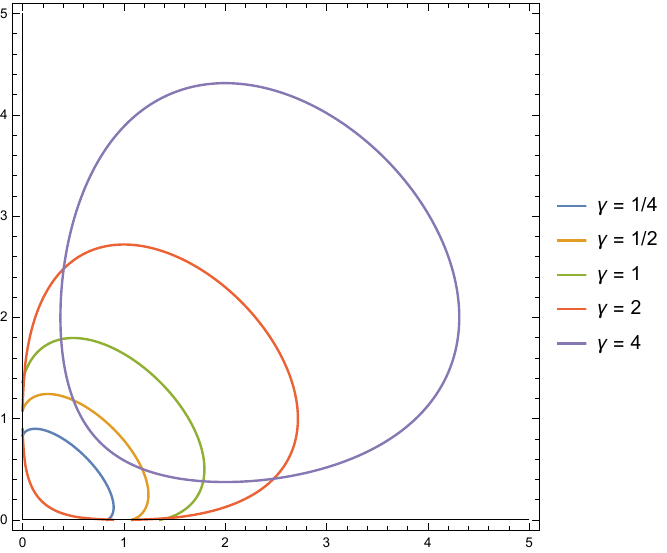}
\end{center}
\caption{\label{fig:shape} Two examples of plots of the boundary of the set $\mathcal I_1$, for two different jump kernels $p$ and various values of $\gamma$. Upper figure: $p((0,1)) = p((0,-1)) = p((1,0)) = p((-1,0)) = 1/4$. $I((x_1,x_2)) = \gamma(g(x_1/\gamma)+g(x_2/\gamma))$, with $g(x) = 1 -\sqrt{1+x^2}+x \sinh ^{-1}(x)$. Lower figure: $p((0,1)) = p((1,0)) = p((0,0))=p((1,1))=1/4$. $I$ is of the same form, with $g(x) = x\log(2x)-x+1/2$. In this case, the sets $\mathcal I_1$ are contained in the first quadrant and their boundaries contain segments of the axes, when $\gamma<2$.}
\end{figure}

Recall the definition of the Hausdorff distance between two sets $X,Y\subseteq \R^d$:
\begin{equation} \label{eq:Hausdorff}
d_H(X,Y) =  \inf\{\epsilon>0: X\subseteq Y_\epsilon\text{ and }Y\subseteq X_\epsilon\},
\end{equation}
where for $X\subseteq \R^d$ and  $\epsilon >0$,
\begin{equation}
X_\epsilon = \bigcup_{x\in X} \{z\in \R^d : \|x-z\|<\epsilon \}. \label{eq:fatten}
\end{equation}

\new{
\begin{theorem}[Shape theorem] \label{th:shape}
For $\delta, \gamma, \lambda, R >0$ and $p:\Z^d\to [0,1]$ satisfying~\eqref{eq:p_assumptions}, 
there exist $\kappa=\kappa(\delta,\gamma, p)>0$ and $N_0=N_0(\delta,\gamma,p,\lambda,R)>0$ such that the following holds.
Suppose $\Lambda :\Z^d \to [0,\infty)$ with $\sum_{x\in \Z^d}\Lambda(x) \in (0,\infty)$ and
\begin{equation} \label{eq:Lambda_assumptions0}
    \begin{aligned}\setcounter{mysubequations}{0}
    \mysubnumber\quad &\sum_{x\in \Z^d}\Lambda (x) \le \frac 1 {N_0},\\
      \mysubnumber\quad &\frac{\Lambda(0)}{\sum_{x\in \Z^d}\Lambda (x)} \ge \lambda,\\
      \mysubnumber\quad &\frac{\Lambda(x)}{\sum_{y\in \Z^d}\Lambda (y)} \le e^{-\kappa \|x\|} \quad \forall x\in \Z^d \text{ with }\|x\|\ge R.\\      
    \end{aligned}
\end{equation}
Let $(\xi_t)_{t\ge 0}$ be a BRWNLC with jump rate $\gamma$, jump kernel $p$ and competition kernel $\Lambda$, as defined in Section~\ref{subsec:modeldefn}.
Then
\[
\psub{\delta_0}{\limsup_{t\to\infty} d_H\left(\frac 1 t \{\xi_t > 0\}, \mathcal I_1(\gamma,p) \right) < \delta} > 1-\delta.
\]
\end{theorem}
}

We remark that an analogue of Theorem~\ref{th:shape} has been proven for branching random walk (BRW) \emph{without} competition by Biggins~\cite{Biggins1978}. Theorem~\ref{th:shape} thus shows that in the limit of large population density, the spreading speed of the BRWNLC is asymptotically the same as in BRW without competition.

\new{
We also remark that as in Theorem~\ref{th:survival}, when $\delta>0$ is chosen to be small, the condition (i) on $\Lambda$ in~\eqref{eq:Lambda_assumptions0} ensures that $\sum_{x\in \Z^d}\Lambda(x)$ is very small, and so the killing rate given in~\eqref{eq:Ktdefn} is very small until there are a large number of particles in the system. A consequence of this is that the lower bound on the probability in the theorem can be close to 1, even though the system is started with a single particle.
}

\subsection{Discussion and comparison with the literature}

\paragraph{Previous works on the BPDL model.}
The BPDL model, studied initially by Bolker and Pacala \cite{Bolker1997} as well as Law and Dieckmann \cite{Law1999}, is a popular individual-based model in population dynamics. It has been studied under various guises in the mathematics, ecology and physics literature. Questions of interest concern global and local survival, asymptotic spread, equilibrium states, and the description of ancestral lineages. The methods used to study the model include the following:
\begin{enumerate}
    \item moment equations and approximation by scaling limits, see e.g.~\cite{Bolker1997,Law1999,Fournier2004,Ovaskainen2006,Finkelshtein2009,Savov2015};
    \item duality, see e.g.~\cite{Hutzenthaler2007,Birkner2016,Birkner2019};
    \item comparison with particle systems and percolation models, see e.g.~\cite{Etheridge2004,Bertacchi2007,Birkner2007,Blath2007}.
\end{enumerate}
The first method is a powerful tool allowing in particular the derivation of precise numerical estimates of various quantities of interest \cite{Cornell2019} but, to our knowledge, has not yet been used to rigorously study the asymptotic behaviour of the process as time goes to infinity. The second method, duality, is a powerful tool, in particular for studying ancestral lineages and the equilibrium distribution, but it is restricted to certain special cases and seems not to be applicable for our model.

The third method is well suited for treating questions concerning survival, asymptotic spread and ergodicity. The main technique is to compare the model with simpler models, such as oriented percolation, by means of a renormalization procedure. This method was introduced by Bramson and Durrett for analysing the contact process \cite{Bramson1988} and has been applied in many contexts since. However, such a comparison (typically done through coupling arguments) is more straightforward if the particle system satisfies a property called \emph{monotonicity}. This property states in particular that two copies $(\xi_t)_{t\ge0}$ and $(\xi'_t)_{t\ge0}$ of the system, starting from two initial configurations $\xi_0$ and $\xi_0'$ with $\xi_0(x) \ge \xi_0'(x)$ for all $x\in\Z^d$, can be coupled in such a way that $\xi_t(x) \ge \xi_t'(x)$ for all $t\ge0$ and $x\in\Z^d$. Unfortunately, the BPDL model is monotone if and only if the competition is on-site only, i.e.~if $\Lambda = \Lambda(0)\delta_0$. 

In order to get around this problem, several authors have introduced additional assumptions, in particular on the jump kernel $p$ and the competition kernel $\Lambda$. For example, Etheridge \cite{Etheridge2004} assumes (for a variant of the model) a condition analogous to $p \ge c\Lambda$ for some $c>0$, which allows her to obtain a certain monotonicity for a truncated version of the model. Birkner and Depperschmidt \cite{Birkner2007} assume that the model evolves in discrete time and that both $p$ and $\Lambda$ have finite range (and that $\Lambda$ is a small perturbation of $\Lambda(0)\delta_0$). While discrete time induces additional complications due to large jumps in the numbers of particles and the chaotic behaviour of the logistic map, it ensures that the dependence between the particles is of finite range only, which allows them to apply a comparison with so-called $k$-dependent oriented percolation. 
(We will discuss these models and other examples of non-monotone particle systems in more detail later in this section.)

In the current article, we allow for $\Lambda$ to be of arbitrary, even infinite range, and we work in continuous time. This makes the particle system non-monotone and of infinite-range dependence. Instead of comparing the process to an oriented percolation, we implement a contour argument tailored to our process. Much work is devoted to dealing with the infinite-range dependence \emph{in space}. In order not to be burdened with the dependence \emph{in time}, we have introduced assumption~(ii) in~\eqref{eq:Lambda_assumptions}. \new{This assumption says that on-site competition is comparable to total competition with other sites, and is a weaker version of Birkner and Depperschmidt's assumption that $\Lambda$ is a small perturbation of $\Lambda(0)\delta_0$. Our proof uses this assumption to} allow us to effectively treat the system as being of finite-range dependence in time. For more details about the proof, see Section~\ref{sec:proof_outline}.

\paragraph{\new{The large population limit.}}
A natural approach to studying the BRWNLC \new{would be to consider the rescaled empirical measure of the particles in the large population limit, i.e.~to consider the limit of $(\xi_t/N)_{t\ge0}$ as $N\to\infty$} and consider the BRWNLC as a perturbation of its limit. This is indeed the underlying idea of the first method mentioned above. The \new{large population limit} of a related model has been rigorously shown to be a certain evolution equation with a quadratic non-linearity due to the competition term \cite{Fournier2004}. Maybe surprisingly, little is known about the long-time behaviour of solutions to this evolution equation. One might expect it to behave similarly to \new{an analogous partial differential equation}, known as the non-local Fisher-KPP equation. While the classical Fisher-KPP equation has been extensively studied since the 30's \cite{Fisher1937,Kolmogorov1937}, with many celebrated results such as Bramson's logarithmic correction to the front position \cite{Bramson1983}, its non-local counterpart has spurred interest only in recent years, see e.g.~\cite{Hamel2014,Penington2018,Bouin2020}. The non-local Fisher-KPP equation displays intriguing behavior, such as the existence of non-constant steady states if the competition kernel is `sufficiently non-local' \cite{Hamel2014}. More importantly for our purposes, the study of the non-local Fisher-KPP equation greatly suffers from the lack of a parabolic maximum principle, which is the basic technical tool for the study of semi-linear parabolic partial differential equations such as the Fisher-KPP equation. Indeed, the parabolic maximum principle is crucially used in order to compare the solution to the equation to simpler functions, chosen to be super- or subsolutions to the equation. It is the analytic analogue of the probabilistic concept of monotonicity mentioned above -- the lack of a parabolic maximum principle for the non-local Fisher-KPP equation is therefore a heritage of the non-monotonicity of the BRWNLC.

To circumvent this problem, the authors of \cite{Hamel2014,Penington2018,Bouin2020} rely on other techniques, such as:
\begin{itemize}
    \item Focusing on the regions where the solution is small, and comparing it to the solution of the linearised equation.
    \item Bootstrapping: starting from `crude' global bounds, and using these bounds to obtain bounds on the regularity of the solution (for example through a certain Harnack-type inequality \cite{Bouin2020}), and obtaining improved lower bounds from crude upper bounds and vice versa (for example using a Feynman-Kac formula \cite{Penington2018}). 
\end{itemize}
Our proof is partly influenced by these ideas, but the stochasticity of the model adds additional difficulties. For example, the lack of deterministic global bounds requires us to handle situations where the particle density is much larger than usual, which a priori might lead to extinction in neighbouring (or more distant) regions. See Section~\ref{sec:proof_outline} for more details.

\paragraph{Non-monotone particle systems.}
Here we give an overview of previous work on non-monotone interacting particle systems, beyond the work already discussed earlier in this section on the BPDL model and closely related models.

Several works have studied branching and annihilating processes on $\Z^d$; in these processes, if at any time there is more than one particle at a single site $x\in \Z^d$, then all the particles at $x$ immediately `annihilate' (all the particles at $x$ disappear). One of the earliest appearances of such models seems to be in~\cite{Griffeath1979}.
Bramson and Gray~\cite{BramsonGray85} considered a branching annihilating random walk in which particles give birth to a new particle at a neighbouring site at rate 1, and jump to a neighbouring site at rate $\delta$; using a hands-on contour method approach, they proved survival with positive probability for $\delta$ sufficiently small.
Bramson, Ding and Durrett~\cite{BramsonDingDurrett91} studied a closely related annihilating branching process in which particles produce offspring particles at neighbouring sites at rate 1, and die at rate $\delta$; they proved survival with positive probability for $\delta$ sufficiently small (and also convergence to a stationary distribution).

The method used in~\cite{BramsonDingDurrett91} is a comparison with oriented percolation using a renormalization procedure as first used by Bramson and Durrett in~\cite{Bramson1988}; the power of this technique is demonstrated for several different examples in~\cite{Durrett91} and~\cite{Durrett10lectures}.
The key to using this technique to prove survival is to choose suitable space-time boxes and a `good' event for each box depending only on the behaviour of particles inside (or not too far away from) the box (or, more precisely, depending only on some graphical representation of the particle system inside or near the box) such that each good event has probability at least $p^*$ for some sufficiently large constant $p^*$, and such that the existence of an infinite path of boxes whose good events occur implies survival for the particle system.
In particular, on the good event for a box, it must be impossible for `far away' particles to kill all of the particles inside the box.
In the case of branching and annihilating processes on $\Z^d$ with finite-range jumps and finite-range interactions between particles, because there is at most one particle at each site, there is a maximum rate at which particles inside a box interact with particles outside, making it feasible to implement this method if suitable results can be shown for the `local' behaviour inside a box.

The Bramson-Durrett method was subsequently applied to several other non-monotone particle systems; see for example work of Bramson and Neuhauser~\cite{BramsonNeuhauser94} on cellular automata, work of Neuhauser and Pacala~\cite{NeuhauserPacala99} on a particle system based on a Lotka-Volterra model with two competing species, work of Etheridge~\cite{Etheridge2004} and Birkner and Depperschmidt~\cite{Birkner2007} on locally regulated population models (as already mentioned earlier in this section), work of Blath, Etheridge and Meredith~\cite{Blath2007} on two models for a pair of competing populations (one of which is dual to a branching annihilating random walk),
and work of Sturm and Swart~\cite{SturmSwart08} on variants of the voter model (also dual to branching annihilating random walks).
In~\cite{BramsonNeuhauser94,NeuhauserPacala99,SturmSwart08}, as in~\cite{BramsonGray85,BramsonDingDurrett91}, the number of particles at a single site is bounded and the interactions have finite range, enabling comparison with oriented percolation.
In~\cite{Etheridge2004,Blath2007}, for a given site $x$, by ignoring times at which the population size at $x$ is higher than a suitable threshold, the net effect of the population at other sites on the population size at $x$ is positive, giving a form of monotonicity which allows comparison with oriented percolation. (As mentioned earlier in this section, this requires an assumption of the form $p>c\Lambda$ for some constant $c>0$.)
In~\cite{Birkner2007}, the model is in discrete time, interactions have finite range, and the on-site competition is substantially stronger than total competition with other sites, meaning that a sufficiently large population at a single site is guaranteed to die out due to on-site competition before it can kill many particles at other sites.

Sudbury and Neuhauser proved results about branching and annihilating systems without using comparison with oriented percolation but using more hands-on approaches and martingale techniques, as well as self-duality and the so-called Holley-Liggett method~\cite{Sudbury90,NeuhauserSudbury93,Sudbury97,Sudbury99,Sudbury00}; again the proofs rely on finite-range interactions and a bounded number of particles at each site.

Blath and Kurt~\cite{BlathKurt11} determined survival and extinction regimes for a non-monotone particle system called a caring double-branching annihilating random walk; their survival result is the simpler part of the proof and is established using a comparison with a supercritical branching process. Athreya and Swart~\cite{AthreyaSwart12} use duality to show that a system of branching and annihilating particles can be seen as a thinning of a particle system without annihilation (but this can only allow initial laws that are Poisson with random intensity).
See also~\cite{Swart17} for a review of the interacting particle systems literature (in particular, Section~5 therein contains examples of non-monotone particle systems).

In our model, in contrast to all the work discussed in this section, we have to deal with the possibility that for any given site, an unbounded number of far away particles could kill all the particles at the site, meaning that a new approach is needed.

The approach developed here involves bounding the probability that the particles originating from a given site might cause the killing rate at sites at a distance $r$ to be too large during a fixed time interval, for every $r$, see Section~\ref{sec:proof_outline}. This is reminiscent of \emph{multi-scale analysis of particle systems}, see~\cite{StaufferLN}. In that context, multi-scale analysis is usually stated using a tessellation of space into boxes at various scales. In our approach, we do not use the language of boxes, but we do bound the probability that sites interact over arbitrarily large distances. It would be interesting to explore in more detail the relation between the two approaches. 

\paragraph{Other branching particle systems with competition.}
We finish this section with a (biased) review of some other branching particle systems with local or non-local competition. Closely related to the BPDL model is the \emph{branching Brownian motion with decay of mass}, where the competition between particles leads to a decay of their mass rather than a reduction in their numbers \cite{ABPe2017}. While this system admits the non-local Fisher-KPP equation as its \new{large population limit} \cite{Addario-Berry2019}, it differs from the BPDL model in that the population density (in terms of particle numbers) is not prescribed by a parameter $N$, but grows with time. 

A popular model of branching random walk with competition is the so-called $N$-BRW, which is a one-dimensional branching random walk in which after each branching step only the $N$ particles at the maximal positions are kept. This model was introduced by Brunet and Derrida \cite{Brunet1997} as a toy model to study finite-$N$ corrections to the speed of travelling wave fronts, and has seen significant interest since in both the physics and the mathematics literature, see e.g. \cite{Brunet2006,Berard2010,nbbm}. The \new{large population limit} of a closely related model, the $N$-BBM, has recently been shown to be a certain free boundary problem \cite{DeMasi2017,Berestycki2018}. The genealogy of the $N$-BBM (and $N$-BRW with light-tailed jump distribution) is conjectured to converge to the Bolthausen-Sznitman coalescent on the time-scale $(\log N)^3$, a fact that has been proven for certain related models \cite{Brunet2007,Berestycki2010}.

One can define a one-dimensional branching Brownian motion with local competition, based on the intersection local time of two Brownian paths. This system is dual to a certain stochastic partial differential equation, the Fisher-KPP equation with Wright-Fisher noise \cite{Shiga1988}. This duality is used in \cite{BarnesMytnikSunComingDown} to study the `coming down from infinity' property of the process.

Maillard, Raoul and Tourniaire \cite{popheterogeneous} study a BRW with local competition in an environment which is space- and time-heterogeneous over a macroscopic scale $1/\epsilon$ \new{in both space and time}. In contrast to the homogeneous case, the authors find that the spreading speed, even in the limit of large population size $N\to\infty$, may differ from its hydrodynamic limit \new{obtained after rescaling time and space}, i.e.~the limits $N\to\infty$ and $\epsilon\to 0$ do not commute in general.

Finally, branching particle systems with long-range dispersal and competition have been considered as well in the literature, see e.g.~\cite{Berard2013,Bezborodov2018,Penington2021}.

\subsection{Proof outline}
\label{sec:proof_outline}

\new{In this section, we provide an overview of the main ideas in the proof, which we hope some readers might find helpful. However, nothing here is required for the proofs in the rest of the article.}

\paragraph{Renormalization grid.}
In order to prove our results, we will take a large constant $T$ and define a renormalization grid of edges of the form $e=((x,kT),(x+y,(k+1)T))$ for some $x,y\in \Z^d$ and $k\in \N_0$, with $x$ in a suitable one-dimensional lattice and $y\in \{y_-,y_+\}$ where $\|y_{\pm}\|=\mathcal O(T)$.
\new{See Figure~\ref{fig:heuristic_grid} for a schematic representation of the grid.
Recall the definition of $N$ in~\eqref{eq:Ndefn}.}
For each edge \new{$e=((x,kT),(x+y,(k+1)T))$}, we will say that the edge is `closed' if an event $\mathcal C_e$ occurs \new{and `open' otherwise}; this event will be defined in such a way that on the event $(\mathcal C_e)^c$, if there are at least $J=\lfloor N^{1/3}\rfloor$ particles at some site near $x$ at time $kT$ then there will be at least $J$ particles at some site near $x+y$ at time $(k+1)T$.
(We could have chosen any $J$ such that $1\ll J \ll N^{1/2},$ but for definiteness, we \new{always take} $J=\lfloor N^{1/3}\rfloor$.)
\new{The idea here is that if the edge is open, having a large number of particles near the start point $(x,kT)$ of the edge guarantees a large number of particles near the end point $(x+y,(k+1)T)$.
Therefore if there exists an infinite path of open edges with a large number of particles near the start point of the first edge in the path, then the particle system must survive.}
See Section~\ref{sec:proof_thm} for precise definitions \new{of the renormalization grid and the events $\mathcal C_e$}.

\paragraph{\new{Main intermediate result and contour argument.}}
The main intermediate step in the proof of Theorems~\ref{th:survival} and~\ref{th:shape} will be to show that for any \newnew{$b\in(0,1)$}, if $N$ is large enough then for any $n\in \N$ and for a collection of edges $e_1,\ldots , e_n$ \new{in our renormalization grid} and a suitable initial \new{particle} configuration $\xi$,
\begin{equation} \label{eq:closededgesk}
\psub{\xi}{\bigcap_{i=1}^n \mathcal C_{e_i}}\le b^n ;
\end{equation}
see Proposition~\ref{prop:bound_closed_edges} below.

\new{In order to apply~\eqref{eq:closededgesk}, we will show that if $N$ is large, if the number of particles in the system ever hits 1 (say at a time $\tau_1$), then with high probability, at a time $\tau_2$ with $\tau_2-\tau_1=\mathcal O(\log J)$, there is a site containing at least $J$ particles.
Indeed, this will be straightforward to prove because for $N$ large, when all sites have less than $J$ particles, the killing rate for each particle is small (recalling the definition of the killing rate in~\eqref{eq:Ktdefn}).
Then by applying the strong Markov property at time $\tau_2$, in order to prove our global survival result it suffices to show that with high probability, for any suitable initial particle configuration in which some site $z$ has at least $J$ particles, there exists an infinite path of open edges such that the first edge in the path has start point $(z,0)$. 
}

Using~\eqref{eq:closededgesk}, we will be able to establish \new{this} using \new{standard} contour arguments.
\new{The idea of the contour argument is that if there is no infinite path of open edges from $(z,0)$, then the `open cluster' (set of vertices in the renormalization grid that can be reached by a path of open edges from $(z,0)$) must be surrounded by a cycle of edges in the dual lattice such that at least half of the dual edges in the cycle cross a closed edge in the renormalization grid. The probability that such a cycle of dual edges exists can be shown to be small using~\eqref{eq:closededgesk} and a union bound.

This argument establishes the global survival result; by choosing suitable renormalization grids and using the same contour argument to imply that there is an infinite path of open edges with high probability, we can also prove the shape theorem result.
The proofs of Theorems~\ref{th:survival} and~\ref{th:shape} using Proposition~\ref{prop:bound_closed_edges} are in Section~\ref{sec:proof_thm}.

In the remainder of this section, we describe the proof of~\eqref{eq:closededgesk}.
}

\begin{figure}
\centering
\def\svgwidth{0.9\columnwidth}
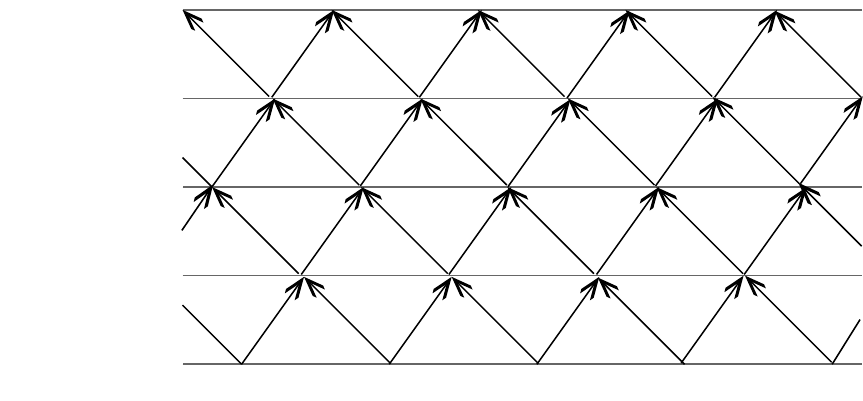
\caption{\label{fig:heuristic_grid} \new{Schematic representation of the renormalization grid described at the start of Section~\ref{sec:proof_outline}. The arrows represent the edges $e$ of the grid, each with a start point $(x,kT)$ and end point $(x+y,(k+1)T)$ in space-time.}}
\end{figure}

\paragraph{\new{Proof of main intermediate result:} dependence in time.}
We first deal with the dependence \emph{in time} of the events $\mathcal C_{e_i}$. Suppose \new{$N$ is large and} we have a subset of edges $e_{i_1},\ldots,e_{i_\ell}$ of the form $((x_j,kT),(x_j+y_j,(k+1)T))$ for some $x_j,y_j\in\Z^d$ and some $k\ge 1$, i.e.~the time-coordinates are the same for all edges in the subset. 
\new{Write $(\F_t)_{t\ge 0}$ for the natural filtration of the process $(\xi_t)_{t \ge 0}$.
We will be able to prove that for \textit{any} particle configuration at time $(k-1)T$, we have
\begin{equation} \label{eq:heur_timedep}
\mathbb P_{\xi}\left( \left. \bigcap_{j=1}^{\ell} \mathcal C_{e_{i_j}} \right| \F_{(k-1)T} \right) \le b^\ell.
\end{equation}}
In other words, we will be able to bound the probability in \eqref{eq:closededgesk} \newnew{as if the following held for the events $\mathcal C_{e_1},\ldots,\mathcal C_{e_n}$: if $\mathcal E_1$ and $\mathcal E_2$ are two subsets of the edges $e_1,\ldots,e_n$ and if there exists $k \ge 2$ such that each $e\in \mathcal E_1$ is of the form $e = ((x,\ell T),(y,(\ell+1)T))$ for some $\ell \le k-2$ and each $e\in \mathcal E_2$ is of the same form, but with $\ell \ge k$, then the two families of events $\{\mathcal C_e:e\in\mathcal E_1\}$ and $\{\mathcal C_e:e\in\mathcal E_2\}$ are mutually independent.}

\new{The proof of~\eqref{eq:heur_timedep} can also be used to show that for a subset of edges $e_{i_1},\ldots,e_{i_\ell}$ of the form $((x_j,0),(x_j+y_j,T))$ for some $x_j,y_j\in\Z^d$, for suitable $\xi$ we have
\begin{equation} \label{eq:heur_timedep0}
\mathbb P_{\xi}\left( \bigcap_{j=1}^{\ell} \mathcal C_{e_{i_j}}  \right) \le b^\ell.
\end{equation}
Using~\eqref{eq:heur_timedep0} together with~\eqref{eq:heur_timedep}, by conditioning on $\F_{(k^*-1)T}$, then $\F_{(k^*-3)T}, \ldots$ successively for a suitably chosen $k^*$, we can prove~\eqref{eq:closededgesk}; see the proof of Proposition~\ref{prop:bound_closed_edges} in Section~\ref{sec:propproof} for details.
}

\new{
From now on, we describe the proof of~\eqref{eq:heur_timedep}.

\paragraph{Stochastic upper bound on particle configuration.}
We will show that for large $N$, the particle configuration at time $kT$, conditional on time $(k-1)T$, has a stochastic upper bound, uniformly over all configurations at time $(k-1)T$.
This stochastic upper bound can be written as follows:
\begin{equation} \label{eq:heur_stoch_upper}
\xi_{kT}=\zeta + \sum_{x\in \Z^d}\zeta^{(x)},
\end{equation}
where $\zeta, \zeta^{(x)}\in \N_0^{\Z^d}$ for $x\in \Z^d$ are particle configurations such that for a large constant $K_0$, we have $\zeta(y)\le K_0 N$ $\forall y\in \Z^d$. Moreover, we can take a collection of i.i.d.~random variables $(Z^{(x)})_{x\in \Z^d}$ such that $Z^{(x)}$ takes value 0 with high probability, and for some small constant $a>0$, the tail is bounded by $\p{Z^{(x)}\ge r}\le e^{-ar\log (r+1)}$ for large $r$, and such that for each $x$, the configuration $\zeta^{(x)}$ is supported on the open ball with centre $x$ and radius $Z^{(x)}$, and $\zeta^{(x)}$ is uniformly bounded by $(Z^{(x)})^A K_0 N$ for a constant $A$.
This result is stated in Proposition~\ref{prop:upperbound}.}
One may call this a `coming down from infinity' property of the particle system.\footnote{The term `coming down from infinity' is often used in coalescent theory and in the study of one-dimensional diffusions and birth-and-death processes, and has a precise meaning in these settings, see e.g.~\cite{Pitman1999,Cattiaux2009,bansaye2016}. For a particle system, one could define it to mean that the system can be defined for an infinite initial configuration as a unique limit when starting from an increasing sequence of finite initial configurations approaching the initial configuration in a certain sense. See Theorem~2 in Hutzenthaler and Wakolbinger~\cite{Hutzenthaler2007}, where this is shown for a certain system of interacting Feller diffusions with logistic growth.
The BRWNLC is in general not monotone, and we do not prove the existence of a unique limit, but we do prove a form of local boundedness uniformly in the initial configuration, which is why we still use the term in a loose sense.}

\new{Before we describe the proof of the stochastic upper bound, we note the following consequence of the on-site competition assumption (condition~(ii) in~\eqref{eq:Lambda_assumptions}).
If for some $r>0$ there are $r N$ particles at a site $x$, then since $\Lambda(0)\ge \lambda N^{-1}$ by~\eqref{eq:Ndefn} and~(ii) in~\eqref{eq:Lambda_assumptions}, the total rate at which particles at $x$ are killed by particles at the same site $x$ is $\Lambda(0)(rN)^2\ge \lambda r^2 N$.
Since new particles are born at $x$ at rate $rN$, this means that if $r$ is large, the number of particles staying at $x$ (and their descendants at $x$) should be decreasing at rate roughly $(\lambda r^2 -r)N\approx \lambda r^2 N$. Hence for a constant $c\in (0,1)$, the number of these particles should decrease to $c rN$ in time of order $r^{-1}$ with high probability.

We use the following strategy to prove the stochastic upper bound~\eqref{eq:heur_stoch_upper}.
We take a small constant $c_0>0$ and split the time interval $[kT-c_0,kT]$ into subintervals $[s_{\overline{m}}, s_{\overline{m}-1}]$, $[s_{\overline{m}-1}, s_{\overline{m}-2}], \ldots , [s_1,s_0]$, where the interval lengths double successively, i.e.~$s_{m-1}-s_{m}=2(s_m-s_{m+1})$ for each $m$.
The value $\overline{m}$ is chosen depending on the (finite) particle configuration at time $kT-c_0$ in such a way that the maximum number of particles at a single site at time $kT-c_0$ is at most $2^{\overline{m}+1}K_1 N$, for a large constant $K_1$.

On the time interval $[kT-c_0,kT]$, particles are coloured either blue or red; at time $s_{\overline m}=kT-c_0$, all particles are blue.
Since $s_{\overline{m}-1}-s_{\overline{m}}$ is of order $2^{-\overline{m}}$, if the constant $K_1$ is large enough then the above heuristics suggest that for each site $x$, the number of particles at $x$ at time $s_{\overline{m}}$ (and their descendants) that stay at $x$ until time $s_{\overline{m}-1}$ should be at most $c_1 2^{\overline{m}}K_1 N$ with high probability, where $c_1$ is a small constant.
At time $s_{\overline{m}-1}$, for each site $x$, if there are more than $c_1 2^{\overline{m}}K_1 N$ particles that have stayed at $x$ since time $s_{\overline{m}}$, these particles are coloured red; also if too many particles at $x$ at time $s_{\overline{m}}$ (or their descendants) have spread too far to other sites by time $s_{\overline{m}-1}$, these particles are coloured red.
The colouring rule is chosen in such a way that at time $s_{\overline{m}-1}$, the number of blue particles at each site is at most $2^{\overline{m}}K_1 N$.
We then repeat this procedure on successive time subintervals, in such a way that for each $m$, the number of blue particles at each site $x$ at time $s_m$ is bounded by $2^{m+1}K_1 N$, and if there are more than $c_1 2^{m}K_1 N$ particles that stay at a site $x$ during time $[s_m,s_{m-1}]$, or if too many particles spread too far from $x$, then these `badly behaved' particles are coloured red.
On each time subinterval $[s_m,s_{m-1}]$, we will be able to show that for a given site $x$, with high probability no particles from $x$ at time $s_m$ are turned red at time $s_{m-1}$.

At time $s_0=kT$, we have at most $2K_1N$ blue particles at each site, and (possibly) some red particles.
This gives us the stochastic upper bound in~\eqref{eq:heur_stoch_upper}: $\zeta$ roughly corresponds to the configuration of blue particles, and for each $x$, $\zeta^{(x)}$ roughly corresponds to the configuration of red particles whose ancestor turned red at the end of some subinterval $[s_m,s_{m-1}]$ and was at $x$ at the start of this subinterval (we think of such particles as red particles `from $x$').
Our bounds on the probability of red particles appearing allow us to prove the conditions on $Z^{(x)}$ and the upper bound on $\zeta^{(x)}$ stated after~\eqref{eq:heur_stoch_upper}. In particular, the tail bound $\p{Z^{(x)}\ge r}\le e^{-ar\log (r+1)}$ comes from the tail of the Poisson distribution: for red particles `from $x$' to reach distance $r$ away from $x$, a continuous-time random walk with finite range jump kernel must move more than distance $r$, and so a Poisson random variable must have value at least $c' r$ for some constant $c'>0$.
}

\new{This proof strategy is carried out in Section~\ref{sec:propupper}. We now describe how~\eqref{eq:heur_stoch_upper} can be used to establish~\eqref{eq:heur_timedep}.}

\paragraph{\new{Sketch proof of~\eqref{eq:heur_timedep}:} dependence in space.}
\new{From now on, we consider edges $e_1,\ldots ,e_n$ in the renormalization grid such that}
the start time for edge $e_i$ is the same $kT$ for each $i$.
\new{To establish~\eqref{eq:heur_timedep}, we want to show that for large $N$,
for any particle configuration at time $(k-1)T$, 
\begin{equation} \label{eq:heur_timedep2}
\mathbb P_{\xi}\left( \left. \bigcap_{i=1}^{n} \mathcal C_{e_{i}} \right| \F_{(k-1)T} \right) \le b^n.
\end{equation}
We can assume that none of the edges $e_i$ are too close together in space (because by removing at most a constant proportion of the edges $e_i$, we can ensure that this condition is satisfied, and so if~\eqref{eq:heur_timedep2} holds when this condition is satisfied, then the general case follows with $b$ replaced by $b^{1/A}$ for some constant $A$).}

\new{We now explain heuristically why, for an edge $e=((x_0,kT),(x_0+y_0,(k+1)T))$, the event $\mathcal C_e$ is unlikely to occur; these heuristics will then be used in our sketch proof of~\eqref{eq:heur_timedep2}.}
Recall that a particle at site $x$ at time $t$ is killed at rate $K_t(x)$.
Let $\epsilon>0$ be a small constant, and let $\mathcal T_e$ denote a tube from $(x_0,kT)$ to $(x_0+y_0,(k+1)T)$ with radius $\epsilon T$, i.e.
$$
\mathcal T_e := \{(z,kT+s): z\in \Z^d, s\in [0,T], \|z- (x_0+y_0 s/T)\|<\epsilon T\}.
$$
We now consider three possible cases for the values of $K_t(x)$ in the tube $\mathcal T_e$.
\new{See Figure~\ref{fig:heuristic_cases} for an illustration.}
Let $c>0$ be a small constant and let $C>0$ be a large constant.

\textbf{Case 1:} $K_t(x)\le c$ for all $(x,t)\in \mathcal T_e$. In this case, the killing rate inside the tube is very small, \new{and particles are branching at rate 1,} so for suitable $y_0$, we can use large deviation results to show that if there are $J$ particles at a site near $x_0$ at time $kT$, they are likely to have more than $J$ descendants that stay inside the tube $\mathcal T_e$, are not killed by competition, and are at a site near $x_0+y_0$ at time $(k+1)T$. \new{Therefore $\mathcal C_e$ is unlikely to occur.}

\textbf{Case 2:} $K_{t^*}(x^*)>c$ for some $(x^*,t^*)\in \mathcal T_e$,
but $K_t(x)\le C$ \new{and $\sum_{\|y\|>\epsilon T}\Lambda(y)\xi_t(x-y)\le c/2$} for all $(x,t)$ close to $\mathcal T_e$.
In this case, at time $t^*$ there will be some site $y^*$ \new{within distance $\epsilon T$ of $x^*$ at which there are at least $N^{1/2}$ particles. Indeed, if we had $\xi_{t^*}(y)< N^{1/2}$ for all $y$ within distance $\epsilon T$ of $x^*$, then by~\eqref{eq:Ktdefn} and~\eqref{eq:Ndefn},
\[
K_{t^*}(x^*) =\sum_{y\in \Z^d}\xi_{t^*}(x^*-y)\Lambda (y)\le N^{1/2}\sum_{y\in \Z^d}\Lambda (y)+\sum_{\|y\|>\epsilon T}\Lambda(y)\xi_{t^*}(x^*-y)\le N^{-1/2}+\tfrac 12 c<c
\]
for large $N$, which is a contradiction.}
Since $N^{1/2} \gg J$, and since the killing rate of particles near $\mathcal T_e$ is at most $C$, we can show that these $N^{1/2}$ particles are likely to have at least $J$ surviving descendants at a site near $x_0+y_0$ at time $(k+1)T$, \new{meaning that $\mathcal C_e$ is unlikely to occur.}

\textbf{Case 3:} $K_{t'}(x')>C$ \new{or $\sum_{\|y\|>\epsilon T}\Lambda(y)\xi_{t'}(x'-y)>c/2$} for some $(x',t')$ near $\mathcal T_e$. 
\new{Using the stochastic upper bound in~\eqref{eq:heur_stoch_upper}, and using the exponential decay of $\Lambda$ (i.e.~condition~(iii) in~\eqref{eq:Lambda_assumptions}), we can show that if $C$ and $T$ are large enough then Case~3 is unlikely to occur.
}

\begin{figure}
\centering
\def\svgwidth{0.9\columnwidth}
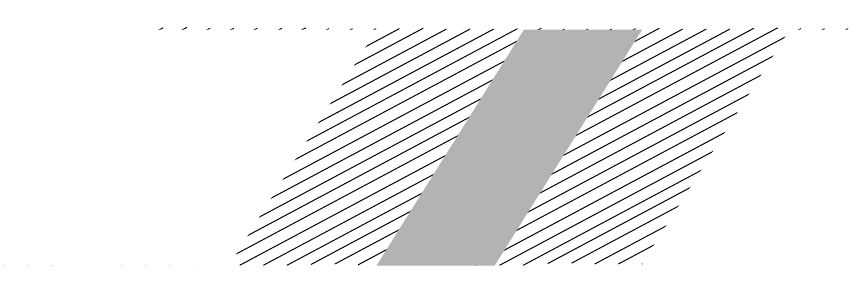
\caption{\label{fig:heuristic_cases} \new{Illustration of Cases 1-3 introduced in the sketch proof of~\eqref{eq:heur_timedep}. The shaded space-time region is the tube $\mathcal T_e$.
In case 1, the killing rate $K_t(x)$ is small (at most $c$) inside the shaded region $\mathcal T_e$.
In case 2, the killing rate is greater than $c$ at some point $(x^*,t^*)$ in the shaded region $\mathcal T_e$, but the killing rate is not too large (at most $C$) throughout the hatched and shaded space-time regions, and for each point $(x,t)$ in the hatched and shaded space-time regions, the contribution to the killing rate from particles distance $\epsilon T$ away from $x$ is small (at most $c/2$).
In case 3, there is a point in the hatched or shaded region where either the killing rate is large (greater than $C$) or the contribution to the killing rate from particles distance $\epsilon T$ away is greater than $c/2$.}}
\end{figure}

The heuristics for the three cases above suggest that the event $\mathcal C_e$ is unlikely to occur.
\new{In order to prove~\eqref{eq:heur_timedep2} using these heuristics, 
we will define `bad events' such that if $\mathcal C_e$ occurs then one of a collection of bad events must occur.}

\new{A useful tool will be a construction of} the process $(\xi_t)_{t\ge 0}$ using independent families of decorated BRW trees; each particle in the BRW trees will be assigned an independent Exp(1) random variable that will determine the time at which it may be killed by competition (as a function of the killing rate on its trajectory).
\new{We will set out this construction in Section~\ref{sec:construction}.}

\paragraph{\new{Bad events. }}
\new{We now describe our construction of the process on the time interval $[kT,(k+1)T]$ that we use to prove~\eqref{eq:heur_timedep2}. This construction will be set out in Section~\ref{sec:propclosedinit}.
We also describe the definition of `bad events' for an edge $e=((x_0,kT),(x_0+y_0,(k+1)T))$.
The bad events are all defined in such a way that they are independent of $\F_{(k-1)T}$.}

\new{Recall our stochastic upper bound in~\eqref{eq:heur_stoch_upper}.}
In our construction, at time $kT$, \new{for each site $x$, we colour particles from the configuration $\zeta$ at the} site \textit{blue} \new{(there are at most $K_0 N$ of these)}, and \new{any} remaining particles at the site \new{are} coloured \textit{red} \new{(with an exception to ensure that if $\xi_{kT}(x)\ge J$ then there are at least $J$ blue particles at $x$)}.
\new{For $y\in \Z^d$, red particles that are from the configuration $\zeta^{(y)}$ are given label $y$.}
 The descendants of each particle \new{are then} constructed using independent decorated BRW trees.
 
Using the heuristic in Case 1 above, we define a `bad event' $B_e$ depending only on the \new{decorated} BRW trees \new{that encode descendants of} blue particles near $x_0$ \new{at time $kT$}. \new{The event is defined in such a way} that on the event $(B_e)^c$, if Case 1 happens \new{and if there are at least $J$ blue particles near $x_0$ at time $kT$}, then there are at least $J$ blue particles near $x_0+y_0$ at time $(k+1)T$, \new{i.e.~}$\mathcal C_e$ cannot occur; moreover, the bad event $B_e$ has low probability if $N$ is large.

We will also define \new{a} stopping \new{time} $\tau(z)$ for each site $z$, given by the first time after time $kT$ at which there are at least $N^{1/2}$ red and blue particles at $z$. \new{In our construction, }at a time $\tau(z)$, \new{if $\tau(z)<(k+1)T$,} we turn $\lfloor N^{1/2}\rfloor $ red and blue particles at $z$ into \textit{yellow} particles, and their descendants will be constructed using another independent family of  \new{decorated} BRW trees.
Using the heuristic in Case~2 above, we define a `bad event' $Y_e$ depending only on the  \new{decorated} BRW trees for yellow particles that appear near the tube $\mathcal T_e$. \new{The event is defined in such a way} that on the event $(Y_e)^c$, if Case 2 happens then there are at least $J$ yellow particles near $x_0+y_0$ at time $(k+1)T$, which means that $\mathcal C_e$ cannot occur; moreover, the bad event $Y_e$ has low probability if $N$ is large.

Finally, we need \new{to define a bad event that occurs if} Case 3 occurs.
We will define bad events $P_{z,r}$ for \new{each} site $z$ and `radius' $r\in \N_0$, which heuristically say that too many descendants of \new{blue particles at $z$ at time $kT$, red particles with label $z$, or yellow particles that appear at $z$} spread to distance $r$ from $z$ at some time in $[kT,(k+1)T]$.
The events will be defined in such a way that for $z'\in \Z^d$, 
\new{if $P_{z, \lfloor \|z-z'\|\rfloor}$ does not occur for every $z\in \Z^d$, i.e.~}on the event $\cap_{z\in \Z^d}(P_{z, \lfloor \|z-z'\|\rfloor})^c$, we have $K_t(z')\le C$ 
\new{and $\sum_{\|y\|>\epsilon T}\Lambda(y)\xi_{t}(z'-y)\le c/2$} for all $t\in [kT,(k+1)T]$.
Therefore if Case~3 occurs then $\cup_{z\in \Z^d}P_{z, \lfloor \|z-z'\|\rfloor}$ must occur for some $z'\in \Z^d$ close to $\mathcal T_e$.
The events $P_{z,r}$ will be independent for different values of $z$
\new{(roughly speaking, this is because the event $P_{z,r}$ only depends on the descendants of particles at $z$, and these are constructed from independent decorated BRW trees for different values of $z$)}.
We will \new{also} be able to show that \new{if $N$ is large}, $P_{z,r}$ has low probability for each $r$, and moreover, for a small constant $a>0$, for large $r$,
\begin{equation} \label{eq:Pheur}
\p{P_{z,r}}\le e^{-a r \log (r+1)}.
\end{equation}
(The proof of~\eqref{eq:Pheur}
\new{comes from the tail bound on $Z^{(x)}$ in the stochastic upper bound in~\eqref{eq:heur_stoch_upper}, and the fact that for a particle to move distance $r$ from $z$, a continuous-time random walk with finite range jump kernel must move distance $r$; in particular,~\eqref{eq:Pheur} relies}
on the assumption that the jump kernel $p$ has finite range.)
Take a small constant $a'\in (0,a/(d+1))$; we can take $T$ in the definition of our renormalization grid to be large enough that for any large $r$, for any fixed $z_0\in \Z^d$, there are at most $a' r$ edges in our grid of the form $((x,kT),(x+y,(k+1)T))$ with $\|x-z_0\|\le r$. (Recall \new{Figure~\ref{fig:heuristic_grid} and} the description of our renormalization grid at the start of this section.)

Write $L_e:=B_e\cup Y_e$; this is a `local' bad event for the edge $e$.
Our definitions of the events $B_e$ and $Y_e$ will ensure that for our collection of edges $e_1,\ldots , e_n$ \new{in~\eqref{eq:heur_timedep2}}, the events $(L_{e_i})_{i=1}^n$ are independent (\new{roughly speaking, this works because the events $B_e$ and $Y_e$ only depend on the decorated BRW trees controlling particles near the edge $e$, and we use} our assumption that the edges are not too close together \new{in space}), and if $N$ is large enough, 
\begin{equation} \label{eq:problocal**}
\p{L_{e_i}}<b/2 \quad \forall i\in \{1,\ldots , n\}.
\end{equation}
\new{In the previous three paragraphs, we described how for an edge $e$, if Case~1 happens then $\mathcal C_e \cap (B_e)^c$ cannot occur, if Case~2 happens then $\mathcal C_e \cap (Y_e)^c$ cannot occur, and if Case~3 happens then 
$\cup_{z\in \Z^d}P_{z, \lfloor \|z-z'\|\rfloor}$ occurs for some $z'\in \Z^d$ close to $\mathcal T_e$.}
\new{Therefore} we can show that if $N$ is large enough, for some large constant $A>0$, for each edge $e_i=((x_i,kT),(x_i+y_i,(k+1)T))$,
\begin{equation} \label{eq:Ceiheuristic*}
\mathcal C_{e_i} \subseteq L_{e_i}\cup \bigcup_{\{y\in \Z^d:\|y-x_i\|\le A\}}\bigcup_{z\in \Z^d}P_{z,\lfloor \|z-y\|\rfloor}.
\end{equation}
\new{The definitions of our bad events will ensure that} the events $P_{z,r}$ are independent for different values of $z$, \new{and }the events $P_{z,r}$ and $L_{e_i}$ are independent if $\|x_i-z\|\ge r+A$ (\new{roughly speaking, this works because the event $P_{z,r}$ only depends on the behaviour of particles inside the ball with centre $z$ and radius $r$, and the event $L_{e_i}$ only depends on the behaviour of particles near the edge $e_i$}).
\new{Moreover, using~\eqref{eq:Pheur} and the comment above~\eqref{eq:Pheur}, and since we chose $a'<a/(d+1)$, we can show}
\begin{equation} \label{eq:Pzrheurbd1}
\p{P_{z,r}}\le \varphi(r) \quad \forall z\in \Z^d \text{ and }r\in \N_0,
\end{equation}
 where $\varphi:[0,\infty)\to [0,1]$ satisfies 
\begin{equation} \label{eq:Pzrheurbd2}
\sum_{\{z\in \Z^d,r\in \N_0:\|z\|\le r+A+1\}}\varphi (r)^{\frac 1{a'(r+A+1)}}<b/2.
\end{equation}
\new{Precise versions of the} claims \new{in this paragraph} are proved in Lemmas~\ref{lem:bad_blue} and~\ref{lem:Ce_inclusion} in Section~\ref{sec:propclosedinit}.

\paragraph{\new{Proof of~\eqref{eq:heur_timedep2} using bad events. }}
We now give a sketch proof that \new{the claims in~\eqref{eq:problocal**}-\eqref{eq:Pzrheurbd2} are} enough to deduce~\eqref{eq:heur_timedep2}; a rigorous version of this part of the proof can be found in Section~\ref{sec:propclosedinit} (in the proof of Proposition~\ref{prop:closedinitcond}).

First note that by~\eqref{eq:Ceiheuristic*}, if the event $\cap_{i=1}^n \mathcal C_{e_i}$ occurs then there must be some subset $S\subseteq\{1,\ldots,n\}$ of the edges such that $L_{e_i}$ occurs for each $i\in S$, and for each $i\in \{1,\ldots ,n\}\setminus S$, an event $P_{z_i,r_i}$ occurs for some $z_i\in \Z^d$ and $r_i\in \N_0$, with $r_i+1\ge \|z_i-x_i\|-A$.
\new{(Indeed, if $P_{z_i,\lfloor \|z_i-y\|\rfloor }$ occurs for some $y$ with $\|y-x_i\|\le A$ then $\|z_i-y\|\ge \|z_i-x_i\|-A$ by the triangle inequality.)}
Therefore, by removing from $S$ any $i$ such that $\|x_i-z_j\|\le r_j+A+1$ for some $j\in \{1,\ldots , n\}\setminus S$, and removing from $\{(z_j,r_j)\}_{j\in \{1,\ldots ,n\}\setminus S}$ any $(z_j,r_j)$ such that $z_j=z_i$ and $r_j<r_i$ for some $i\in \{1,\ldots ,n\}\setminus S$, and then \new{finally} relabelling the $(z_i,r_i)$, if $\cap_{i=1}^n \mathcal C_{e_i}$ occurs then there must exist $S\subseteq \{1,\ldots ,n\}$, $\ell \le n$ and $(z_1,r_1),\ldots ,(z_\ell,r_\ell)\in \Z^d\times \N_0$, with $z_1,\ldots , z_\ell$ distinct, such that $L_{e_i}$ occurs for each $i\in S$, and $P_{z_j,r_j}$ occurs for each $j\le \ell$, and for each $i\in \{1,\ldots ,n\}\setminus S$, there exists $j(i)\le \ell$ such that $\|x_i-z_{j(i)}\|\le r_{j(i)}+A+1$, and for each $i\in S$, $\|x_i-z_j\|>r_j+A+1$ $\forall j\le \ell$.
Hence by a union bound, and then in the second line using independence, and finally in the third line using~\eqref{eq:Pzrheurbd1},
\begin{align*}
\mathbb P_{\xi}\left( \left. \bigcap_{i=1}^{n} \mathcal C_{e_{i}} \right| \F_{(k-1)T} \right)
&\le \sum_{\substack{S\subseteq \{1,\ldots ,n\}, \, \ell \le n, \, \{(z_1,r_1),\ldots , (z_\ell ,r_\ell)\}
\text{ with }z_j \text{ distinct s.t. }\\
\forall i\in \{1,\ldots ,n\}\setminus S, \, \exists j(i)\le \ell \, \text{ s.t. }\|x_i-z_{j(i)}\|\le r_{j(i)}+A+1,\\
\{j(i)\}_{i\in \{1,\ldots , n\}\setminus S}=\{1,\ldots , \ell\}\text{ and }\forall i\in S, \, \|x_i-z_j\|>r_j +A+1 \,  \forall j\le \ell
}
}
\p{\bigcap _{i\in S}L_{e_i}\cap \bigcap_{j=1}^\ell P_{z_j,r_j}}\\
&= \sum_{\substack{S\subseteq \{1,\ldots ,n\}, \, \ell \le n, \, \{(z_1,r_1),\ldots , (z_\ell ,r_\ell)\}
\text{ with }z_j \text{ distinct s.t. }\\
\forall i\in \{1,\ldots ,n\}\setminus S, \, \exists j(i)\le \ell \, \text{ s.t. }\|x_i-z_{j(i)}\|\le r_{j(i)}+A+1,\\
\{j(i)\}_{i\in \{1,\ldots , n\}\setminus S}=\{1,\ldots , \ell\}\text{ and }\forall i\in S, \, \|x_i-z_j\|>r_j +A+1 \,  \forall j\le \ell
}
}
\prod _{i\in S}\p{L_{e_i}} \prod_{j=1}^\ell \p{P_{z_j,r_j}}\\
&\le \sum_{\substack{S\subseteq \{1,\ldots ,n\}, \, \ell \le n, \, \{(z_1,r_1),\ldots , (z_\ell ,r_\ell)\}
\text{ with }z_j \text{ distinct s.t. }\\
\forall i\in \{1,\ldots ,n\}\setminus S, \, \exists j(i)\le \ell \, \text{ s.t. }\|x_i-z_{j(i)}\|\le r_{j(i)}+A+1\\
\text{and }\{j(i)\}_{i\in \{1,\ldots , n\}\setminus S}=\{1,\ldots , \ell\}
}
}
\prod _{i\in S}\p{L_{e_i}} \prod_{j=1}^\ell \varphi( r_j).
\end{align*}
We now claim that 
\begin{align} \label{eq:CprobheurStar}
&\sum_{\substack{S\subseteq \{1,\ldots ,n\}, \, \ell \le n, \, \{(z_1,r_1),\ldots , (z_\ell ,r_\ell)\}
\text{ with }z_j \text{ distinct s.t. }\\
\forall i\in \{1,\ldots ,n\}\setminus S, \, \exists j(i)\le \ell \, \text{ s.t. }\|x_i-z_{j(i)}\|\le r_{j(i)}+A+1\\
\text{and }\{j(i)\}_{i\in \{1,\ldots , n\}\setminus S}=\{1,\ldots , \ell\}
}
}
\prod _{i\in S}\p{L_{e_i}} \prod_{j=1}^\ell \varphi( r_j) \notag \\
&\hspace{3cm} \le \prod_{i=1}^n \bigg( \p{L_{e_i}}+\sum_{\{z\in \Z^d, r\in \N_0:\|z-x_i\|\le r+A+1\}}\varphi(r)^{\frac 1 {a'(r+A+1)}}\bigg).
\end{align}
\new{Indeed, 
we can rewrite the right hand side of~\eqref{eq:CprobheurStar} as
\begin{align} \label{eq:CprobheurStar2}
&\prod_{i=1}^n \bigg( \p{L_{e_i}}+\sum_{\{z\in \Z^d, r\in \N_0:\|z-x_i\|\le r+A+1\}}\varphi(r)^{\frac 1 {a'(r+A+1)}}\bigg) \notag\\
&= 
\sum_{\substack{S\subseteq \{1,\ldots ,n\},  \, (z_i,r_i)_{i\in \{1,\ldots ,n\}\setminus S}
\text{ s.t. }\\
\|x_i-z_{i}\|\le r_{i}+A+1 \, \forall i\in \{1,\ldots , n\}\setminus S
}
}
\prod _{i\in S}\p{L_{e_i}} \prod_{i\in \{1,\ldots , n\}\setminus S} \varphi(r_i)^{\frac 1 {a'(r_i+A+1)}}.
\end{align}
Then}
for each $S\subseteq \{1,\ldots , n\}$, $\ell \le n$ and $\{(z_1,r_1),\ldots ,(z_\ell, r_\ell)\}$ corresponding to a term in the sum on the left hand side of~\eqref{eq:CprobheurStar}, 
\new{there is a term in the sum on the right hand side of~\eqref{eq:CprobheurStar2} corresponding to $S$, $(z_{j(i)},r_{j(i)})_{i\in \{1,\ldots , n\}\setminus S}$, and which is given by
}
\begin{align*}
\prod_{i\in S}\p{L_{e_i}}\prod_{i\in \{1,\ldots ,n\}\setminus S}\varphi(r_{j(i)})^{\frac 1 {a'(r_{j(i)}+A+1)}}
&=\prod_{i\in S}\p{L_{e_i}}\prod_{j=1}^\ell \prod_{\{i\in \{1,\ldots ,n\}\setminus S:j(i)=j\}}\varphi(r_{j})^{\frac 1 {a'(r_{j}+A+1)}}\\
&\ge \prod_{i\in S}\p{L_{e_i}}\prod_{j=1}^\ell \varphi(r_{j}),
\end{align*}
where the last inequality follows because \new{for each $j$,} if $j(i)=j$ then $\|x_i-z_j\|\le r_j+A+1$, and \new{after~\eqref{eq:Pheur}} we chose our renormalization grid in such a way that \new{for any $z$,} the total number of edges $e=((x,kT),(x+y,(k+1)T))$ \new{in the grid with $\|x-z\|\le r$} is at most $a'r$ for large $r$, and so \new{in particular}
$$
\# \{i\in \{1,\ldots ,n\}\setminus S:j(i)=j\}\le a' (r_j+A+1).
$$
This proves the claim~\eqref{eq:CprobheurStar}; by~\eqref{eq:problocal**} and~\eqref{eq:Pzrheurbd2},~\eqref{eq:heur_timedep2} then follows from~\eqref{eq:CprobheurStar}.

Note that the bound \eqref{eq:Pheur} is sharp for this argument to work, since otherwise the sum on the right hand side of \eqref{eq:CprobheurStar} diverges. In particular, this means that we are currently not able to weaken the assumption that the jump kernel has finite range.

\subsection{Overview of the article}

The remainder of the article is organised as follows. In Section~\ref{sec:construction}, we construct the BRWNLC from a collection of BRW trees decorated with `resiliences'. 
Section~\ref{sec:proof_thm} contains the contour argument used to prove the main results, relying on Proposition~\ref{prop:bound_closed_edges}. Section~\ref{sec:propproof} contains the proof of Proposition~\ref{prop:bound_closed_edges}. Finally, Section~\ref{sec:badbluepf} contains certain estimates for sums of independent random variables, which are the building blocks of the proof of Proposition~\ref{prop:bound_closed_edges}.

\subsection{Notation}
Throughout the article, we write $\llbracket m \rrbracket := \{1,2,\ldots, \lfloor m \rfloor \}$ for  any $m\ge 1$. Furthermore, we write $\N = \{1,2,\ldots\}$ and $\N_0 = \{0,1,2,\ldots\}$.

As mentioned in Section~\ref{subsec:modeldefn}, we write $\|\cdot\|$ for the $\ell^2$ or Euclidean norm.
For $x\in \R^d$ and $r\ge 0$, we let
\begin{equation} \label{eq:ballrdefn}
\mathcal B_r(x):=\{z\in \Z^d:\|x-z\|<r\}.
\end{equation}
For $x=(x_1,\ldots,x_d)\in \R^d$, we write
\begin{equation} \label{eq:intpartdefn}
\lfloor x \rfloor :=(\lfloor x_1 \rfloor , \ldots , \lfloor x_d \rfloor )\in \Z^d.
\end{equation}

\subsection*{Acknowledgements}

This work was initiated at the Centre de recherches mathématiques (CRM), Montreal, during the Montreal Summer Workshop on Probability and Mathematical Physics organised by Louigi Addario-Berry, Omer Angel and Alex Fribergh in July 2018. We warmly thank the organisers for making the event possible and the CRM for its hospitality.

PM warmly thanks Hugo Duminil-Copin for several stimulating discussions during the initial phase of the project.

PM acknowledges partial support from a Simons CRM fellowship and ANR grant ANR-20-CE92-0010-01.
SP is supported by a Royal Society University Research Fellowship, and acknowledges partial support from a Simons CRM fellowship.

\new{Both authors would like to thank the anonymous referee for their careful reading of the article and helpful comments and suggestions.}

\section{Families of branching random walks with resiliences}
\label{sec:construction}

In this section, we introduce a useful construction of the process $(\xi_t)_{t\ge 0}$ from a collection of decorated branching random walk trees; versions of this construction will be used throughout the proofs.

\subsection{Branching random walks}
\label{sec:brw}
The branching random walks (BRW) we consider here are continuous-time BRW on $\Z^d$, started with a single particle at $0$, where each particle branches into two child particles at rate $1$, located at the position of their parent, and furthermore, each particle jumps at rate $\gamma$ according to the jump kernel $p$. 
(When a particle branches, we say the parent particle `dies' and the two child particles are `born'.)
For such a branching random walk, we introduce notation as follows:
\begin{itemize}
\item Particles are given labels from the set of labels $\mathcal U=\cup_{n=0}^\infty \{1,2\}^n$ according to Ulam-Harris labelling, i.e.~the label 12 corresponds to the second child of the first child of the initial particle (ordering of the children is arbitrary). We write $u\prec v$ for $u,v\in \mathcal U$ if $v$ is a descendant of $u$ (including $u$ itself) in the tree $\mathcal U$.
\item For $t\geq 0$, let $\mathcal N_t\subset \mathcal U$ denote the set of labels of particles alive at time $t$.
\item Let $(X_t(u),u\in \mathcal N_t)$ denote the locations in $\Z^d$ of the particles at time $t$.
\item Let $(\alpha(u),u\in \mathcal U)$ and $(\beta(u),u\in \mathcal U)$ denote the birth and death times of particles respectively.
\item For $u\in \mathcal U$, for $t < \alpha(u)$, we write $X_t(u)$ to denote the location of the ancestor of particle $u$ at time $t$.
\end{itemize}

The existence and formal construction of the process is standard and can easily be obtained by recurrence over the generations, see e.g.~\cite{Jagers1989,Hardy2009}. This way, the trajectory of each particle between branch points is a continuous-time random walk with jump kernel $p$ and jump rate $\gamma$ until a finite time given by the life length of the particle, and the trajectories, independent over all particles, are `glued together' at the branch points. Equivalently, one could construct the BRW following a more modern approach using random trees: Start with a Yule tree, i.e.~a binary tree where each edge has a length given by an Exp(1)-distributed random variable. Consider a Poisson process on the tree with intensity measure equal to $\gamma$ times the length measure of the tree. Now define a random process indexed by the tree that jumps at the times given by the Poisson process according to the jump kernel $p$. See Section~3 in Duquesne and Winkel \cite{Duquesne2007} for a construction of the Yule tree as a random metric space and the notion of Poisson processes on trees.

The \emph{many-to-one} lemma for branching Markov processes allows us to calculate additive functionals of the branching Markov process in terms of a single Markov process in a potential. See for example Section~8 in Hardy and Harris \cite{Hardy2009} for a modern presentation using change of measure techniques. Here we will use the following version, which is a special case of Corollary~8.6 in that article.

\begin{lemma}[Many-to-one lemma \cite{Hardy2009}]
\label{lem:many-to-one}
Let $F$ be a non-negative path functional. Let $(X_t)_{t\ge0}$ be a continuous-time random walk started at $0$ with jump kernel $p$ and jump rate $\gamma$. Then
\[
\E{\sum_{u\in\mathcal N_t} F((X_s(u))_{s\in[0,t]})} = e^t \E{F((X_s)_{s\in[0,t]})}.
\]
\end{lemma}

The following result will also be needed. Note that $|\mathcal N_t|$, the number of particles at time $t$, is a Yule process, i.e.~a continuous-time Galton-Watson process with binary offspring distribution. This implies the following classical result:
\begin{lemma} [\cite{Harris1963}, Section V.8]
\label{lem:yule}
Let $t\ge0$. Then $|\mathcal N_t|\sim \operatorname{Geom}(e^{-t})$, i.e.
\[
\P(|\mathcal N_t| = n) = e^{-t}(1-e^{-t})^{n-1},\quad n\ge1.
\]
\end{lemma}

\subsection{Resilience and BRW with non-local competition}
Let a BRW tree be defined as in the previous section. Let $(\rho(u),u\in \mathcal U)$ be i.i.d.~Exp(1) random variables. We call $\rho(u)$ the \emph{resilience} of the particle $u$.
We denote by 
\begin{equation} \label{eq:calTdefn}
\mathcal T = ((\mathcal N_t)_{t\ge0},(X_t(u),t\ge0,u\in \mathcal N_t),(\alpha(u),u\in \mathcal U),(\beta(u),u\in \mathcal U),(\rho(u),u\in \mathcal U))
\end{equation}
the tuple representing the BRW tree with resiliences. Such a tree will be used to encode the evolution of the descendants of a given particle in the process $(\xi_t)_{t\ge 0}$, represented by the root of the tree. %

For $u\in \mathcal U$, a BRW tree with resiliences $\mathcal T = ((\mathcal N_t)_{t\ge0},(X_t(u),t\ge0,u\in \mathcal N_t),(\alpha(u),u\in \mathcal U),(\beta(u),u\in \mathcal U),(\rho(u),u\in \mathcal U))$, a function $\kappa = (\kappa_t(y))_{t\ge0,y\in\Z^d}$ and a location $x\in \Z^d$, define
\[
\delta(u; \mathcal T,\kappa,x) = \inf\left\{t\in [0,\beta(u)-\alpha(u)]: \int_0^{t} \kappa_{\alpha(u)+s}(x+X_{\alpha(u)+s}(u))\,ds > \rho(u)\right\},
\]
with the convention that $\inf \emptyset =\infty$.
This quantity will be interpreted as the length of time between the birth of the particle $u$ and its death by competition (if this occurs before its `death by branching'), given that $\kappa_t(y)$ is the killing rate experienced by particles at position $y$ at time $t$ and that the root particle of the BRW tree $\mathcal T$ is at position $x$ at time 0.

The process $(\xi_t)_{t\ge 0}$ can now be defined as a deterministic function of a collection of BRW trees with resiliences. Take a collection $\mathcal T^1,\ldots,\mathcal T^n$ of i.i.d.~copies of $\mathcal T$, representing the descendants of $n$ particles at positions $x_1,\ldots,x_n$ at time $0$. For each $i$, let $\mathcal T^i = ((\mathcal N^i_t)_{t\ge0},(X^i_t(u),t\ge0,u\in \mathcal N^i_t),(\alpha^i(u),u\in \mathcal U),(\beta^i(u),u\in \mathcal U),(\rho^i(u),u\in \mathcal U))$. First set
\[
\xi'_t(x) = \sum_{i=1}^n \sum_{u\in\mathcal N^i_t} \1_{x = x_i + X^i_t(u)} \quad \text{for }t\ge 0, \, x\in \Z^d.
\]
Define $K'_t(x) = \xi'_t * \Lambda(x)$. Now let
\[
\sigma_1 = \min\left\{\alpha^i(u) + \delta(u;\mathcal T^i,K',x_i ): i\in \{1,\ldots,n\},\ u\in \mathcal U\right\}
\]
and denote by $(i^*,u^*)$ the minimiser of this quantity. The time $\sigma_1$ is interpreted as the first time a particle is killed by competition. We set 
\[
\xi_t = \xi'_t \quad \text{for }t<\sigma_1.
\]
We then iterate this process, but ignoring the particle $u^*$ in $\mathcal T^{i^*}$ and its descendants from time $\sigma_1$ onwards. Explicitly, write $(i,u) \succ (i^*,u^*)$ if $i = i^*$ and $u \succ u^*$. Define 
\begin{align*}
\xi''_t(x) &= 
\begin{cases}
\xi'_t(x) \quad & \text{for }t<\sigma_1 \\
\sum_{i=1}^n \sum_{\{u\in\mathcal N^i_t: (i,u) \not\succ (i^*,u^*)\}} \1_{x = x_i + X^i_t(u)} \quad  & \text{for }t\geq \sigma_1,
\end{cases}\\
K''_t(x) &=  \xi''_t * \Lambda(x) \qquad \text{for }t\ge 0,\\
\sigma_2 &= \min\left\{\alpha^i(u) + \delta(u;\mathcal T^i,K'',x_i): i\in \{1,\ldots,n\},\ u\in \mathcal U,\ (i,u)  \not\succ (i^*,u^*)\right\},\\
\xi_t &= \xi''_t \quad \text{for }\sigma_1\le t<\sigma_2.
\end{align*}
We continue like this to define $\xi_t$ at all times.
Then $(\xi_t(x),x\in \Z^d)_{t\geq 0}$ is a BRWNLC
with $\xi_0(x)=\sum_{i=1}^n \1_{x_i=x}.$

We write $(\F_{t})_{t\ge 0}$ for the natural filtration of the process $(\xi_t)_{t\geq 0}$.

\section{Proof of Theorems~\ref{th:survival} and \ref{th:shape} (contour arguments)}
\label{sec:proof_thm}

Recall the definition of the rate function $I(v)\new{=I(v;\gamma,p)}$ in~\eqref{eq:Idefn}. Define its \emph{domain} by
\[
\operatorname{dom}(I) = \{x\in\R^d: I(x) < \infty\}.
\]
The following lemma gathers a few properties of $I$ and $\dom(I)$. 
Recall the definition of $\mu$ in~\eqref{eq:mudefn}.
\begin{lemma} \label{lem:basicLambda}
\new{Suppose $p$ satisfies~\eqref{eq:p_assumptions}, and take $\gamma>0$.}
The rate function $I$ satisfies the following properties:
\begin{enumerate}
\item The closure of $\dom(I)$ is the closed cone generated by the convex hull of the support of the jump kernel $p$; in other words, it is the set of linear combinations with non-negative coefficients of points $x\in\Z^d$ such that $p(x) > 0$.
\item The interior of $\dom(I)$, denoted $\dom(I)^\circ$, is non-empty and contains $\mu$. Furthermore, $I$ attains its minimum at $\mu$ where $I(\mu) = 0$.
\item $I$ is continuous on $\dom(I)^\circ$.
\end{enumerate}
\end{lemma}
\begin{proof}
We provide a proof for completeness. 
Recall that $(X_t)_{t\ge 0}$ is a continuous-time random walk started at $0$ with jump rate $\gamma$ and jump kernel $p$.
By Cramér's theorem, $I$ is the good convex rate function in the large deviation principle for $X_n/n$ \cite[Theorem 2.2.30]{DemboZeitouni} and therefore the closure of its domain equals the closed convex hull of the support of $X_1$, see e.g.~\cite{Lanford1973,Petit2018}, see also \cite[Lemma~1]{Biggins1978}. Now, $X_1$ is the sum of a Poisson($\gamma$) number of i.i.d.~random variables with law $p$, and therefore the convex hull of its support is easily shown to be the closed cone generated by the convex hull of the support of $p$, see for example Section~6 in \cite{Biggins1978}. This proves the first point.

For the second point, recall from our assumptions on $p$ \new{in~\eqref{eq:p_assumptions}} that
the set $\mathcal S(p,\gamma):=\{\gamma x: x\in \Z^d, p(x)>0\}$
 is a spanning set of the vector space $\R^d$. Hence, since $\dom(I)$ includes this set as well as the origin $0\in\R^d$, its affine hull is the whole space $\R^d$. It is therefore a convex set of dimension $d$ with non-empty interior $\dom(I)^\circ$ (see e.g.~\cite[Section 2]{Rockafellar1970} for basic properties of convex sets). Now we have that $\mu$, which is a convex combination of points in $\mathcal S(p,\gamma)$, is contained in the relative interior of the convex hull of $\mathcal S(p,\gamma)$ \cite[Theorem~6.9]{Rockafellar1970}, and hence in the relative interior of $\dom(I)$, by the first part of the lemma and the fact that the relative interior of a convex set is contained in the relative interior of the cone it generates \cite[Corollary 6.8.1]{Rockafellar1970}. Since $\dom(I)$ is of dimension $d$, it follows that $\mu\in\dom(I)^\circ$. Furthermore, $I(\mu) \le 0$ by the definition of $I$ in~\eqref{eq:Idefn} and Jensen's inequality, and $I(v) \ge0$ for all $v\in\R^d$, using $u=0$ in~\eqref{eq:Idefn}. This proves the second point.

The third point is a direct consequence of convexity, see \cite[Theorem~10.1]{Rockafellar1970}.
\end{proof}

\paragraph{Renormalization grid.} \new{Take $\gamma>0$ and $p$ satisfying~\eqref{eq:p_assumptions}.} Recall the definition of $\mathcal I_1$ in~\eqref{eq:I1defn}.
For $v\in \R^d\backslash\{0\}$, let
\begin{equation}  \label{eq:a0defn}
a_v = \sup\{a\ge 0:I(\mu+av)\le 1\} = \sup\{a\ge 0:\mu+av\in \mathcal I_1\}.
\end{equation}
Note that $a_v>0$, because by Lemma~\ref{lem:basicLambda} we have that $\mu\in\dom(I)^\circ$, $I(\mu)=0$ and $I$ is continuous on $\dom(I)^\circ$. Furthermore, $a_v<\infty$, because $\mathcal I_1$ is compact. Furthermore, we have $\mu+a_v v\in \mathcal I_1$, since $\mathcal I_1$ is compact, and hence closed. The function $v\mapsto 1/a_v$ is also called the \emph{gauge function} of the convex set $\mathcal I_1-\mu$ \cite[Section~4]{Rockafellar1970}.

\begin{figure}
\centering
\def\svgwidth{1.1\columnwidth}
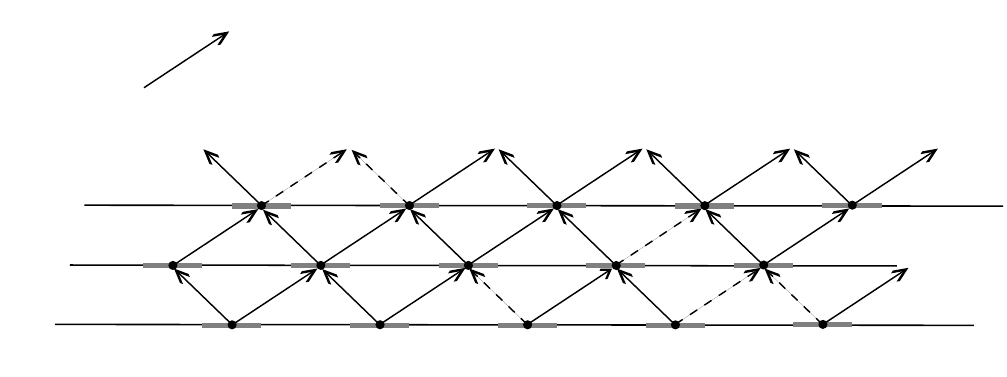
\caption{\label{fig:grid} Schematic representation of the renormalization grid defined in Section~\ref{sec:proof_thm}. The arrows represent the edges $e$ of the grid. Open edges are drawn with solid lines, closed edges with dashed lines \new{(an arbitrary choice has been made for a possible configuration of open and closed edges)}. An edge from $(x,kT)$ to $(y,(k+1)T)$ is closed if there are at least $J$ particles at some site $x'\in\ball_{R'}(x)$ at time $kT$ but \new{either} there are less than $J$ particles at every site $y'\in\ball_{R'}(y)$ at time $(k+1)T$ \new{or there is a time during $[kT,(k+1)T]$ when no particles are in the ball $\ball_{R'}(x)$}.}
\end{figure}

Recall the definition of $\lfloor x \rfloor$ for $x\in \R^d$ from~\eqref{eq:intpartdefn}, and recall the definition of $\mathcal B_r(x)$ from~\eqref{eq:ballrdefn}.
Take $R',$ $T>1$ large positive constants to be fixed later, set $J=\lfloor N^{1/3}\rfloor$ and take $v_0\in \R^d\backslash\{0\}$ and $\aminus,\aplus\ge0$. We introduce a renormalization grid, see Figure~\ref{fig:grid} for a schematic illustration.
Define sets of vertices by letting 
\begin{align*}
&V_{T,\aminus,\aplus ,v_0}(k)
=\{\ell_1 \lfloor (\mu+\aminus v_0) T \rfloor + \ell_2 \lfloor (\mu +\aplus v_0)T \rfloor : \ell_1, \ell_2 \in \Z, \ell_1+ \ell_2=k\}
\quad \text{for }k\in \N_0.
\end{align*}
Then define the sets of directed edges for $k\in \N_0$,
\begin{align} \label{eq:ETkdefn}
&E_{T,\aminus,\aplus ,v_0}(k) \notag \\
&\quad =\left\{
((x,kT),(x+y,(k+1)T)):
x\in V_{T,\aminus,\aplus ,v_0}(k), y\in \{\lfloor (\mu+\aminus v_0) T \rfloor,\lfloor (\mu+ \aplus v_0)T \rfloor\}
\right\}, \notag \\
&E_{T,\aminus ,\aplus ,v_0} = \bigcup_{k\in\N_0} E_{T,\aminus ,\aplus ,v_0}(k).
\end{align}
For an edge $e = ((x,s),(y,t)) \in E_{T,\aminus ,\aplus ,v_0}$, define the event
\begin{align}
\mathcal C_e&=\{\exists x'\in \ball_{R'}(x) \text{ s.t. } \xi_s(x')\geq J\} \nonumber \\
&\qquad \cap \Big[\{\forall y'\in \ball_{R'}(y),\  \xi_t(y')<J\} 
\cup \{\exists u\in [s,t]\text{ s.t. }\xi_u(x'')=0\ \forall x''\in \new{\ball_{R'}(x)} \}\Big]. \label{eq:Cdefn}
\end{align}
(At first reading, the reader can safely ignore the third event in the definition of $\mathcal C_e$ and concentrate on the first two events; the third event will only be relevant at the end of the proof of Theorem~\ref{th:shape}.)
We say that $e$ is \emph{closed} if $\mathcal C_e$ occurs and $e$ is \emph{open} otherwise. We want to show by a contour argument (also known as a Peierls argument~\new{\cite{Peierls1936}}) that there exists an infinite cluster of open edges with high probability when starting from a `good' initial configuration. The key to this argument is the following proposition.

\begin{prop}
\label{prop:bound_closed_edges}
For \new{$\gamma>0$, $p$ satisfying~\eqref{eq:p_assumptions},} $v_0\in \R^d\backslash\{0\}$ and $0\le \aminus <\aplus <a_{v_0}$, there exists $\kappa=\kappa(\gamma,p,v_0,a_-,a_+)>0$ such that the following holds.
For $T>1$ sufficiently large and \new{$\lambda, R>0$},
there exists $R'>1$ such that  for
every $b\in (0,1)$, there exists $N_0>1$ such that the following holds:
\new{Suppose $\Lambda$ satisfies~\eqref{eq:Lambda_assumptions}, take} $n\in\N$ and $e_1,e_2,\ldots ,e_n\in E_{T,\aminus ,\aplus ,v_0}$, and let $\xi$ be a starting configuration consisting of a finite number of particles.
\new{Recall the definition of $N$ in~\eqref{eq:Ndefn}, and let $J=\lfloor N^{1/3}\rfloor$.}
 Assume one of the following conditions:
\begin{enumerate}
\item $e_1,e_2,\ldots ,e_n\in E_{T,\aminus ,\aplus ,v_0}(1)$ and $\xi$ arbitrary, or
\item $e_1,e_2,\ldots ,e_n\in E_{T,\aminus ,\aplus ,v_0}$ arbitrary and  $\xi(x) \le N$ for every $x\in\Z^d$.
\end{enumerate}
\new{Let $(\xi_t)_{t\ge 0}$ be a BRWNLC with jump rate $\gamma$, jump kernel $p$ and competition kernel $\Lambda$, as defined in Section~\ref{subsec:modeldefn}.}
Then \new{recalling the definition of $\mathcal C_e$ in~\eqref{eq:Cdefn}},
$$
\psub{\xi}{\bigcap_{i=1}^n \mathcal C_{e_i}} \le b^n.
$$
\end{prop}

\begin{remark}
\new{The proof of the above proposition under condition~2 uses the result under condition~1.}
\end{remark}
We will prove Proposition~\ref{prop:bound_closed_edges} in Sections~\ref{sec:propproof}-\ref{sec:badbluepf}; in the remainder of this section, we use Proposition~\ref{prop:bound_closed_edges} and the construction in Section~\ref{sec:construction} to prove Theorem~\ref{th:survival} and Theorem~\ref{th:shape}.
\begin{proof}[Proof of Theorem~\ref{th:survival}]
Suppose $N\ge 1$ (and so $J=\lfloor N^{1/3}\rfloor \ge 1$).
Define the stopping times
\begin{align}
\tau_1 &= \inf\{t\ge 0 : \sum_{x\in \Z^d}\xi_t(x)=1\} \notag \\
\text{and }\quad \tau_2 &= \inf\{t\ge 0 : \sup_{x\in \Z^d}\xi_t(x)\ge J\}. \label{eq:tau2defn}
\end{align}
For $\xi\in \N_0^{\Z^d}$ consisting of a finite number of particles with $\xi \not\equiv 0$, by conditioning on $\F_{\tau_1}$ and applying
the strong Markov property,
\begin{align} \label{eq:tau1cond}
\psub{\xi}{\xi_t \equiv 0 \text{ for some }t>0}
&=\Esub{\xi}{\psub{\xi_{\tau_1}}{\xi_t \equiv 0 \text{ for some }t>0}\1_{\tau_1<\infty}} \notag \\
&= \psub{\xi}{\tau_1<\infty}\psub{\delta_0}{\xi_t \equiv 0 \text{ for some }t>0},
\end{align}
where the last line follows by translational symmetry.
Then by conditioning on $\F_{\tau_2}$ and applying the strong Markov property again,
\begin{align} \label{eq:tau2cond}
\psub{\delta_0}{\xi_t \equiv 0 \text{ for some }t>0}
&\le \Esub{\delta_0}{\psub{\xi_{\tau_2}}{\xi_t \equiv 0 \text{ for some }t>0}\1_{\tau_2<\infty}}
+\psub{\delta_0}{\tau_2=\infty}.
\end{align}
\new{The remainder of the proof consists of the following two steps, which we now sketch before giving the complete proof:
\begin{itemize}
\item Step 1: We show that the second term on the right hand side of~\eqref{eq:tau2cond} is small. The idea here is that until there is a site with at least $J$ particles, the killing rate at every site is very small (at most $J/N$), and this will allow us to show that $\tau_2\le \tfrac 32 \log J$ with high probability.
\item Step 2: We show that the first term on the right hand side of~\eqref{eq:tau2cond} is small.
By the definition of the stopping time $\tau_2$, it suffices to show that $(\xi_t)_{t\ge 0}$ is likely to survive if the initial particle configuration is bounded by $J$ and if there is a site $z$ with at least $J$ particles. By translational invariance we can assume $z=0$. Then if there is an infinite path of open edges from $(0,0)$ in our renormalization grid, the particle system survives for all time. We prove that such an infinite path is likely to exist using a contour argument and Proposition~\ref{prop:bound_closed_edges}.
\end{itemize}
}

\paragraph{\new{\textbf{Step 1:}}}
Construct the BRWNLC process $(\xi_t)_{t\ge 0}$ with initial configuration $\xi_0=\delta_0$ as in Section~\ref{sec:construction}, using the BRW tree with resiliences $\mathcal T = ((\mathcal N_t)_{t\ge0},(X_t(u),t\ge0,u\in \mathcal N_t),(\alpha(u),u\in \mathcal U),(\beta(u),u\in \mathcal U),(\rho(u),u\in \mathcal U))$.
\new{Recall from Section~\ref{sec:brw} that for $u\in \mathcal U$, for $t < \alpha(u)$, we write $X_t(u)$ to denote the location of the ancestor of particle $u$ at time $t$.}
Recall that $J=\lfloor N^{1/3}\rfloor$, and define the event
\begin{align}  \label{eq:Adefmainthm}
A &= \{ |\mathcal N_{\frac 32 \log J}|\ge J |\ball_{(\log J)^2}(0)|,
\new{|X_{t}(u)|< (\log J)^2 \, \forall u \in \mathcal N_{\frac 32 \log J},t\le \tfrac 32 \log J} \notag \\
&\hspace{2cm} \text{ and }
\rho(u)> \tfrac J N \min(\tfrac 32 \log J-\alpha (u),\beta(u)-\alpha(u)) \; \forall u \in \cup_{t\le \frac 32 \log J}\mathcal N_{t}\}.
\end{align}
We claim that if $\tau_2>\frac 32 \log J$ then $A^c$ occurs.
Indeed, suppose (aiming for a contradiction) that $A$ occurs and $\tau_2>\frac 32 \log J$.
Then for $x\in \Z^d$ and $t\in [0, \frac 32 \log J]$, we have
$$
K_t(x) = \sum_{y\in \Z^d}\xi_t(x-y)\Lambda(y)\le J N^{-1},
$$
by the definition of $\tau_2$ in~\eqref{eq:tau2defn} and since $\sum_{z\in \Z^d} \Lambda(z)=N^{-1}$.
Therefore, by our construction in Section~\ref{sec:construction}, and since $\rho(u)> \tfrac J N \min(\tfrac 32 \log J-\alpha (u),\beta(u)-\alpha(u))$ $\forall u \in \cup_{t\le \frac 32 \log J}\mathcal N_{t}$, we have
$\xi_t(x)=\sum_{u\in \mathcal N_t}\1_{x=X_t(u)}$ for each $x\in \Z^d$ and $t\in [0, \frac 32 \log J]$.
By the definition of the event $A$, it follows that 
$$
\sum_{x\in \ball_{(\log J)^2}(0)}\xi_{\frac 32 \log J}(x)\ge |\mathcal N_{\frac 32 \log J}|\ge J |\ball_{(\log J)^2}(0)|,
$$
which implies that $\tau_2 \le \frac 32 \log J$ and gives us a contradiction,
proving that the claim holds.

We now establish an upper bound on $\p{A^c}$.
Since $|\mathcal N_t|\sim \text{Geom}(e^{-t})$ for $t\ge 0$ by Lemma~\ref{lem:yule}, we have
\begin{equation} \label{eq:pAc1}
\p{ |\mathcal N_{\frac 32 \log J}|< J |\ball_{(\log J)^2}(0)|}\le J |\ball_{(\log J)^2}(0)|\cdot J^{-3/2}.
\end{equation}
Let $(N^{(\gamma)}_t)_{t\ge 0}$ denote a Poisson process with rate $\gamma$.
By the many-to-one lemma (Lemma~\ref{lem:many-to-one}), and since the jump kernel $p$ is supported on $\ball_{R_1}(0)$,
\begin{align} \label{eq:pAc2}
\p{\exists u \in \mathcal N_{\frac 32 \log J}: \new{\sup_{t\le \frac 32 \log J}|X_{t}(u)|}\ge (\log J)^2}
&\le J^{3/2} \p{N^{(\gamma)}_{\frac 32 \log J}\ge (\log J)^2/ R_1} \notag \\
&\le J^{3/2} e^{-(\log J)^2/R_1}e^{\frac 32 \gamma \log J(e-1)},
\end{align}
where the last line follows by Markov's inequality.

\new{
Let $\F^*$ denote the $\sigma$-algebra generated by $ ((\mathcal N_t)_{t\ge0},(X_t(u),t\ge0,u\in \mathcal N_t),(\alpha(u),u\in \mathcal U),(\beta(u),u\in \mathcal U))$, i.e.~the BRW tree \textit{without} the resiliences.
Then first using that $(\rho(u),u\in \mathcal U)$ are i.i.d.~with $\rho(u)\sim\text{Exp}(1)$ for each $u\in \mathcal U$, and then using our notation from Section~\ref{sec:brw} that
$u\prec v$ for $u,v\in \mathcal U$ if $v$ is a descendant of $u$ (including $u$ itself) in the tree $\mathcal U$, and using that each $u\in \cup_{t\le \frac 32 \log J}\mathcal N_{t}$ has some descendant $v\in \mathcal N_{ \frac 32 \log J}$, we have
\begin{align} \label{eq:rhoJN1}
&\p{ \left. \rho(u)> \tfrac J N \min(\tfrac 32 \log J-\alpha(u),\beta(u)-\alpha(u))\, \forall u \in \cup_{t\le \frac 32 \log J}\mathcal N_{t}\right| \F^*} \notag \\
&= \prod_{u\in \cup_{t\le \frac 32 \log J}\mathcal N_{t}}\exp\left(-\tfrac J N \min(\tfrac 32 \log J-\alpha(u),\beta(u)-\alpha(u))\right) \notag \\
&\ge \exp\left(-\tfrac J N \sum_{v\in \mathcal N_{ \frac 32 \log J}}\sum_{u\prec v} \min(\tfrac 32 \log J-\alpha(u),\beta(u)-\alpha(u))\right).
\end{align}
For each $v\in \mathcal N_{ \frac 32 \log J}$ we have
\begin{equation} \label{eq:rhoJN2}
\sum_{u\prec v} \min(\tfrac 32 \log J-\alpha(u),\beta(u)-\alpha(u))=\tfrac 32 \log J,
\end{equation}
and therefore, first by conditioning on $\F^*$, then using~\eqref{eq:rhoJN1} and~\eqref{eq:rhoJN2}, and then using that $1-e^{-y}\le y$ for $y\ge 0$, 
\begin{align} \label{eq:pAc3}
&\p{\exists u \in \cup_{t\le \frac 32 \log J}\mathcal N_{t}: \rho(u)\le \tfrac J N \min(\tfrac 32 \log J-\alpha(u),\beta(u)-\alpha(u))} \notag \\
&=\E{\p{\left. \exists u \in \cup_{t\le \frac 32 \log J}\mathcal N_{t}: \rho(u)\le \tfrac J N \min(\tfrac 32 \log J-\alpha(u),\beta(u)-\alpha(u))\right| \F^*}} \notag \\
&\le \E{1-\exp\left(-\tfrac 32  \tfrac J N \log J |\mathcal N_{ \frac 32 \log J}| \right) } \notag \\
&\le \E{\tfrac 32  \tfrac J N \log J |\mathcal N_{ \frac 32 \log J}|} \notag \\
&= J^{3/2} \cdot \tfrac 32 \tfrac J N \log J,
\end{align}
where the last line follows since $|\mathcal N_t|\sim \text{Geom}(e^{-t})$ for $t\ge 0$ by Lemma~\ref{lem:yule}.
}

Hence by~\eqref{eq:pAc1},~\eqref{eq:pAc2} and~\eqref{eq:pAc3} and a union bound,
\begin{align} \label{eq:Abound}
\psub{\delta_0}{\tau_2 >\tfrac 32 \log J}
\le \p{A^c}
&\le  |\ball_{(\log J)^2}(0)|J^{-1/2}+J^{3/2} e^{-(\log J)^2/R_1}e^{\frac 32 \gamma \log J(e-1)}
+\tfrac 32 \tfrac {J^{5/2}} N \log J \notag \\
&\le N^{-1/7}
\end{align}
for $N$ sufficiently large, since $J=\lfloor N^{1/3}\rfloor$. 
\new{This completes Step~1.}

\paragraph{\new{\textbf{Step 2:}}}
By~\eqref{eq:tau1cond},~\eqref{eq:tau2cond} and~\eqref{eq:Abound}, the proof now reduces to showing that there exists $\kappa>0$ such that \new{for $\delta,\lambda,R>0$, if $N_0$ is sufficiently large and $\Lambda$ satisfies~\eqref{eq:Lambda_assumptions}}, then for any $\xi'\in \N_0^{\Z^d}$ consisting of a finite number of particles with $\xi'(x)\le J$ $\forall x\in \Z^d$ and $\xi' (x^*)=J$ for some $x^*\in \Z^d$, 
\begin{equation} \label{eq:thmclaim}
\psub{\xi'}{\xi_t\equiv 0 \text{ for some }t>0}\le \delta/2.
\end{equation}
By translational symmetry, we may assume that $x^*=0$.
We will now prove~\eqref{eq:thmclaim} under this assumption using Proposition~\ref{prop:bound_closed_edges} and a contour argument.
This will be a standard application of the contour argument, as can be found in, for example, the proof of Theorem~A.1 in~\cite{Durrett10lectures}, but we include the argument here for completeness.

\begin{figure}
\centering
\def\svgwidth{\columnwidth}
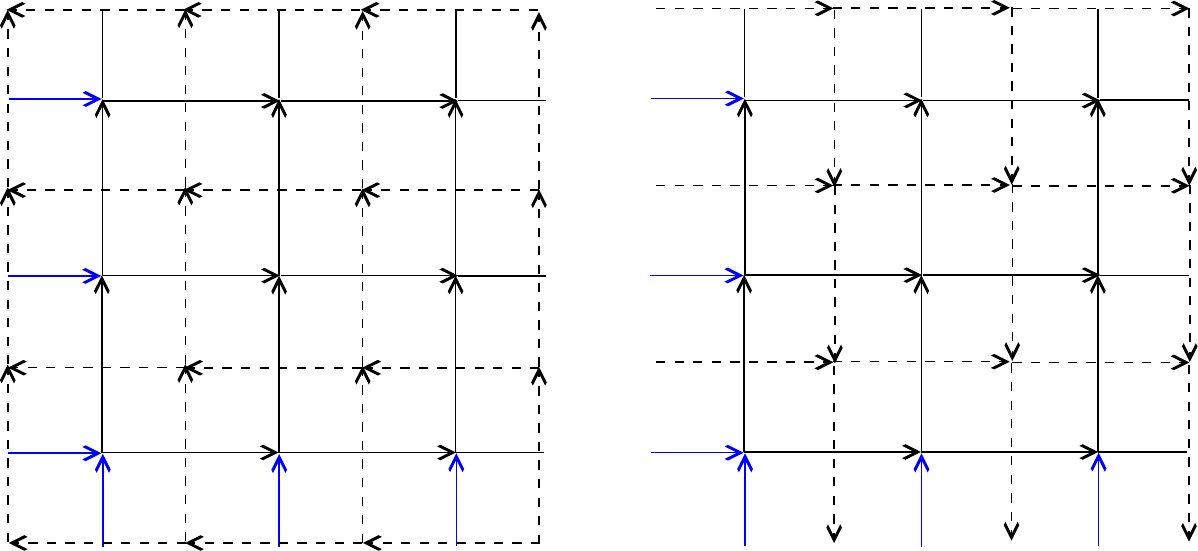
\caption{\label{fig:Enw_Ese} \new{Illustration of $E$, $E^0$ and $E^*$ introduced in Step 2 of the proof of Theorem~\ref{th:survival}.
In both the left and right figures, the black arrows represent edges in $E^0$ and the blue arrows represent edges in $E\setminus E^0$.
In the left figure, the dashed arrows represent edges in $E^{*,\mathrm{nw}}$, and the points marked with green crosses represent $n_v$ for the point $v$ shown.
In the right figure, the dashed arrows represent edges in $E^{*,\mathrm{se}}$.
In each figure, we show an example of a dual edge $e^*$ and corresponding edge $c(e^*)$ and vertices $l(e^*)$ and $r(e^*)$.
}
}
\end{figure}

For $T,\aminus ,\aplus \ge0$ and $v_0\in \R^d$, for $k\in \N_0$, let
\begin{align*}
V^0_{T,\aminus ,\aplus ,v_0}(k)&=\{\ell_1 \lfloor (\mu+\aminus v_0) T \rfloor + \ell_2 \lfloor (\mu +\aplus v_0)T \rfloor : \ell_1, \ell_2 \in \N_0, \ell_1+ \ell_2=k\}\\
 \text{and }\quad 
E^0_{T,\aminus ,\aplus ,v_0}(k)
&=\{
((x,kT),(x+y,(k+1)T)):\\
&\hspace{2cm}
 x\in V^0_{T,\aminus ,\aplus ,v_0}(k), y\in \{\lfloor (\mu+\aminus v_0) T \rfloor , \lfloor (\mu+\aplus v_0)T\rfloor \} \},
\end{align*}
and let 
$E^0_{T,\aminus ,\aplus ,v_0} = \bigcup_{k\in\N_0} E^0_{T,\aminus ,\aplus ,v_0}(k).$
Define sets of directed edges in the lattice $\Z^2$ by letting
\begin{align*}
E&=\{((\ell_1,\ell_2),(\ell_1+1,\ell_2)):\ell_1,\ell_2\in \Z\}\cup \{((\ell_1,\ell_2),(\ell_1,\ell_2+1)):\ell_1,\ell_2\in \Z\}\\
\text{and }\quad E^0 &=\{((\ell_1,\ell_2),(\ell_1+1,\ell_2)):\ell_1,\ell_2\in \N_0\}\cup \{((\ell_1,\ell_2),(\ell_1,\ell_2+1)):\ell_1,\ell_2\in \N_0\}.
\end{align*}
Suppose $\aminus \neq \aplus $, $v_0\neq 0$ and $T$ is sufficiently large that $\lfloor (\mu+\aminus v_0) T \rfloor \neq \lfloor (\mu+\aplus v_0)T\rfloor$.
Then, letting $v_1=\lfloor (\mu+\aminus v_0) T \rfloor$ and $v_2= \lfloor (\mu+\aplus v_0)T\rfloor$, define a bijective map
\begin{align*}
P_{T,\aminus ,\aplus ,v_0}: E^0_{T,\aminus ,\aplus ,v_0} &\to E^0\\
((\ell_1 v_1+\ell_2 v_2, (\ell_1+\ell_2)T),(\ell'_1 v_1+\ell'_2 v_2, (\ell'_1+\ell'_2)T))
&\mapsto ((\ell_1,\ell_2),(\ell'_1,\ell'_2)).
\end{align*}
Now set $v_0=(1,0,\ldots ,0)\in \R^d$, $\aminus  = 0$ and $\aplus  = \frac 12 a_{v_0} \in (0,a_{v_0})$.
Take $\kappa>0$ as in Proposition~\ref{prop:bound_closed_edges}; fix $\lambda, R>0$.
Take $T>1$ sufficiently large that Proposition~\ref{prop:bound_closed_edges} holds and $\lfloor (\mu+\aminus v_0) T \rfloor \neq \lfloor (\mu+\aplus v_0)T\rfloor$, and take $R'>1$ as in Proposition~\ref{prop:bound_closed_edges}.
\new{Take $\delta>0$, then take $b\in (0,1)$ sufficiently small that $\sum_{\ell=4}^\infty 3^{\ell-1} b^{\ell/2}<\delta/2$, and then take $N_0$ sufficiently large that Proposition~\ref{prop:bound_closed_edges} holds with this choice of $b$. Suppose $\Lambda$ satisfies~\eqref{eq:Lambda_assumptions}.}
Take $\xi'\in \N_0^{\Z^d}$ consisting of a finite number of particles with $\xi'(x)\le J$ $\forall x\in \Z^d$ and $\xi' (0)=J$,
and \new{let $(\xi_t)_{t\ge 0}$ be a BRWNLC with jump rate $\gamma$, jump kernel $p$ and competition kernel $\Lambda$, and} with $\xi_0=\xi'$.

For a directed edge $e\in E^0$, we say that the edge $e$ is closed if the event $\mathcal C_{P^{-1}_{T,\aminus ,\aplus ,v_0}(e)}$ occurs (recall~\eqref{eq:Cdefn}), and the edge $e$ is open otherwise.
For $u,v\in \N_0^2$, write $u\to v$ if there exists a path of open directed edges in $E^0$ from $u$ to $v$, and define the open cluster $\mathcal A:=\{v\in \N_0^2 :(0,0)\to v\}\cup \{(0,0)\}$.
Note that if $\xi_t\equiv 0$ for some $t>0$, then by the definition of the event $\mathcal C_e$ in~\eqref{eq:Cdefn}, and since $\xi_0(0)\ge J$, we must have $|\mathcal A|<\infty$.

\begin{figure}
\centering
\def\svgwidth{0.5\columnwidth}
\begingroup%
  \makeatletter%
  \providecommand\color[2][]{%
    \errmessage{(Inkscape) Color is used for the text in Inkscape, but the package 'color.sty' is not loaded}%
    \renewcommand\color[2][]{}%
  }%
  \providecommand\transparent[1]{%
    \errmessage{(Inkscape) Transparency is used (non-zero) for the text in Inkscape, but the package 'transparent.sty' is not loaded}%
    \renewcommand\transparent[1]{}%
  }%
  \providecommand\rotatebox[2]{#2}%
  \newcommand*\fsize{\dimexpr\f@size pt\relax}%
  \newcommand*\lineheight[1]{\fontsize{\fsize}{#1\fsize}\selectfont}%
  \ifx\svgwidth\undefined%
    \setlength{\unitlength}{217.14708944bp}%
    \ifx\svgscale\undefined%
      \relax%
    \else%
      \setlength{\unitlength}{\unitlength * \real{\svgscale}}%
    \fi%
  \else%
    \setlength{\unitlength}{\svgwidth}%
  \fi%
  \global\let\svgwidth\undefined%
  \global\let\svgscale\undefined%
  \makeatother%
  \begin{picture}(1,0.94162494)%
    \lineheight{1}%
    \setlength\tabcolsep{0pt}%
    \put(0,0){\includegraphics[width=\unitlength,page=1]{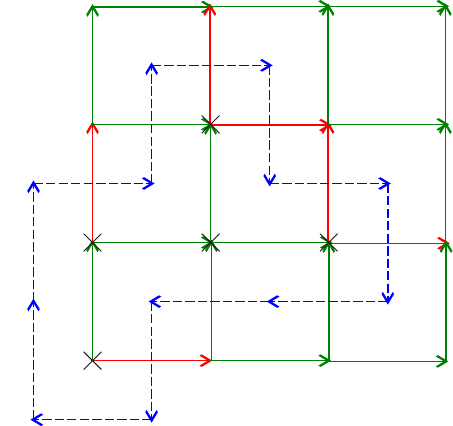}}%
    \put(-0.00140717,0.11768313){\color[rgb]{0.10196078,0.10196078,0.10196078}\makebox(0,0)[lt]{\lineheight{1.25}\smash{\begin{tabular}[t]{l}$e^*_0$\end{tabular}}}}%
    \put(0.2121874,0.17179789){\color[rgb]{0.10196078,0.10196078,0.10196078}\makebox(0,0)[lt]{\lineheight{1.25}\smash{\begin{tabular}[t]{l}$(0,0)$\end{tabular}}}}%
    \put(0.21760115,0.4327026){\color[rgb]{0.10196078,0.10196078,0.10196078}\makebox(0,0)[lt]{\lineheight{1.25}\smash{\begin{tabular}[t]{l}$(0,1)$\end{tabular}}}}%
    \put(0.47096446,0.16841552){\color[rgb]{0.10196078,0.10196078,0.10196078}\makebox(0,0)[lt]{\lineheight{1.25}\smash{\begin{tabular}[t]{l}$(1,0)$\end{tabular}}}}%
  \end{picture}%
\endgroup%

\caption{\label{fig:contour} \new{A possible realisation of the open cluster $\mathcal A$ and the cycle of dual edges $\Gamma$ in Step 2 of the proof of Theorem~\ref{th:survival}.
Green arrows represent open edges in $E^0$ and red arrows represent closed edges in $E^0$.
The points in $\mathcal A$ are indicated with crosses.
The blue dashed arrows give the cycle of dual edges $\Gamma$ containing $e_0^*$ such that for each edge $e^*\in \Gamma$, $r(e^*)\in \mathcal A$ and $l(e^*)\notin \mathcal A$.
}
}
\end{figure}

Define the set of directed edges in the dual lattice by letting $E^*=E^{*,\mathrm{nw}}\cup E^{*,\mathrm{se}}$, where
\begin{align*}
E^{*,\mathrm{nw}}&= \{((\ell-\tfrac 12 , k - \tfrac 12 ),(\ell-\tfrac 12 , k + \tfrac 12 )):\ell, k \in \N_0\}
\cup  \{((\ell+\tfrac 12 , k - \tfrac 12 ),(\ell-\tfrac 12 , k - \tfrac 12 )):\ell, k \in \N_0\},\\
E^{*,\mathrm{se}}&= \{((\ell + \tfrac 12 , k + \tfrac 12 ),(\ell + \tfrac 12 , k - \tfrac 12 )):\ell, k \in \N_0\}
\cup  \{((\ell - \tfrac 12 , k + \tfrac 12 ),(\ell + \tfrac 12 , k + \tfrac 12 )):\ell, k \in \N_0\}.
\end{align*}
\new{See Figure~\ref{fig:Enw_Ese} for an illustration.}
Note that each dual edge $e^*\in E^*$ crosses exactly one edge in $E$; write $c(e^*)$ for this edge.
When travelling along $e^*$, either the start vertex or the end vertex of $c(e^*)$ is on the right; call this vertex $r(e^*)\in \Z^2$, and call the vertex on the left $l(e^*)\in \Z^2$.
For a vertex $v=(\ell-\frac 12, k - \frac 12)$ in the dual lattice with $\ell, k\in \N_0$, let
$n_v=\{w\in \Z^2:\|v-w\|=1/\sqrt 2\}$ denote the set of the four closest vertices in $\Z^2$ to $v$.
Suppose $1\le |n_v \cap \mathcal A |\le 3$; then
\begin{align*}
1 &\le |\{v':(v,v')\in E^* \text{ with }r((v,v'))\in \mathcal A\text{ and }l((v,v'))\notin \mathcal A\}|\\
&=  |\{v'' :(v'',v)\in E^* \text{ with }r((v'',v))\in \mathcal A\text{ and }l((v'',v))\notin \mathcal A\}|.
\end{align*}
Therefore, if $|\mathcal A|<\infty$, then by following a path of distinct edges in $E^*$ starting with the edge $e^*_0:=((-\frac 12, -\frac 12),(-\frac 12, \frac 12))$ such that for each edge $e^*$ in the path, $r(e^*)\in \mathcal A$ and $l(e^*)\notin \mathcal A$,
we can see that there exists a cycle $\Gamma$ of edges in $E^*$ containing $e^*_0$ such that for each edge $e^*\in \Gamma \cap E^{*,\mathrm{se}}$, the edge in $E$ that it crosses is closed (i.e.~$\mathcal C_{P^{-1}_{T,\aminus ,\aplus ,v_0}(c(e^*))}$ occurs).
\new{See Figure~\ref{fig:contour}.}

Let $\mathcal G$ denote the set of cycles of edges in $E^*$ containing $e^*_0$.
Note that each $\Gamma\in \mathcal G$ must have length at least 4, and for $\ell\ge 4$, the number of cycles in $\mathcal G$ with length $\ell$ is at most $3^{\ell-1}$. Moreover, if $\Gamma \in \mathcal G$ has length $\ell$, then
$|\Gamma \cap E^{*,\mathrm{se}}|= \ell/2$.
Therefore by a union bound,
\begin{align*}
\psub{\xi'}{\xi_t \equiv 0\text{ for some }t>0}\le \sum_{\Gamma \in \mathcal G}\psub{\xi'}{\bigcap_{e^*\in \Gamma \cap E^{*,\mathrm{se}}}\mathcal C_{P_{T,\aminus ,\aplus ,v_0}^{-1}(c(e^*))}}
&\le \sum_{\Gamma \in \mathcal G} b^{|\Gamma \cap E^{*,\mathrm{se}}|}\\
&\le \sum_{\ell=4}^\infty 3^{\ell-1} b^{\ell/2}<\delta/2,
\end{align*}
where the second inequality follows by Proposition~\ref{prop:bound_closed_edges} under condition~\new{2}, \new{and the last inequality follows from our choice of $b$}. This establishes~\eqref{eq:thmclaim}, and completes the proof.
\end{proof}
\begin{proof}[Proof of Theorem~\ref{th:shape}]
Recall from~\eqref{eq:fatten} that for $X\subseteq \R^d$, the set $X_\epsilon$ denotes the $\epsilon$-fattening of $X$.
We begin by proving the upper bound, i.e.~we show that for every $\epsilon>0,$ we have
\begin{equation}
\label{eq:shape_upper}
\psub{\delta_0}{\exists t_0<\infty: \frac 1 t \{\xi_t > 0\} \subseteq (\mathcal I_1)_\epsilon\text{ for all $t\ge t_0$}} = 1.
\end{equation}
This follows from classical results on branching random walks:
Let $\mathcal T = ((\mathcal N_t)_{t\ge0},(X_t(u),t\ge0,u\in \mathcal N_t),(\alpha(u),u\in \mathcal U),(\beta(u),u\in \mathcal U),(\rho(u),u\in \mathcal U))$ be a BRW tree with resiliences as in Section~\ref{sec:construction}.
By our construction in Section~\ref{sec:construction}, for $t\ge 0$,
\[
\{\xi_t > 0\} \subseteq \{X_t(u): u\in\mathcal N_t\}.
\]
But for the BRW, it was shown by Biggins~\cite{Biggins1978} (see page 79 of that paper) that for every $\epsilon>0$,
\[
\p{\exists t_0<\infty: \left\{\frac 1 t X_t(u): u\in\mathcal N_t\right\} \subseteq (\mathcal I_1)_\epsilon \text{ for all }t\ge t_0} = 1.
\]
This proves the upper bound \eqref{eq:shape_upper}.

We now prove the lower bound, i.e.~we want to show that for every $\epsilon>0,$ there exists $\kappa>0$ such that for \new{$\lambda, R>0$, if $N_0$ is sufficiently large and $\Lambda$ satisfies~\eqref{eq:Lambda_assumptions0},}
we have
\[
\psub{\delta_0}{\exists t_0<\infty: \mathcal I_1\subseteq \left(\frac 1 t S_t\right)_\epsilon\text{ for all $t\ge t_0$}} \ge 1-\epsilon,
\]
where $S_t = \{\xi_t > 0\}$. 
\new{Construct the BRWNLC process $(\xi_t)_{t\ge 0}$ from the BRW tree with resiliences $\mathcal T$ as at the start of Step~1 of the proof of Theorem~\ref{th:survival}.
Recall the definition of the event $A$ in~\eqref{eq:Adefmainthm} and the definition of $\tau_2$ in~\eqref{eq:tau2defn}, and recall that we showed after~\eqref{eq:Adefmainthm} that if $A$ occurs then $\tau_2\le \frac 32 \log J$.
By the strong Markov property at time $\tau_2$, we can write
\begin{align*}
&\psub{\delta_0}{\exists t_0<\infty: \mathcal I_1\subseteq \left(\frac 1 t S_t\right)_\epsilon\text{ for all $t\ge t_0$}} \\
&\quad \ge \Esub{\delta_0}{\psub{\xi_{\tau_2}}{\exists t_0<\infty: \mathcal I_1\subseteq \left(\frac 1 {t+\tau_2} S_t\right)_\epsilon\text{ for all $t\ge t_0$}}\1_{\tau_2\le \frac 32 \log J}}.
\end{align*}
Since $\p{A^c}\le N^{-1/7}$ by~\eqref{eq:Abound} 
in the proof of Theorem~\ref{th:survival}, and since on the event $A$, at time $\tau_2$, all particles are in $\mathcal B_{(\log J)^2}(0)$, it now suffices to prove the following:
}
For every $\epsilon>0$, there exists $\kappa>0$ such that for \new{$\lambda, R>0$, if $N_0$ is sufficiently large and $\Lambda$ satisfies~\eqref{eq:Lambda_assumptions0},} for any
initial condition $\xi \in \N_0^{\Z^d}$ consisting of a finite number of particles with $\xi(y)\le J$ $\forall y \in \Z^d$ and $\xi(0)=J$, 
and any $x^*\in \mathcal B_{(\log J)^2}(0)$ and $t^*\in [0,\frac 32 \log J]$,
\[
\psub{\xi}{\exists t_0<\infty: \mathcal I_1 \subseteq \left(\frac 1 {t+t^*} (S_t+x^*)\right)_\epsilon\text{ for all $t\ge t_0$}} \ge 1-\epsilon.
\]
Recall the definition of $a_v$ in~\eqref{eq:a0defn}, and recall that we observed in~\eqref{eq:I1defn}
that $\mathcal I_1$ is convex and compact.
It follows that $\mathcal I_1=\{\mu+av:v\in \R^d, \|v\|=1, a\in [0,a_v]\}$ and hence 
by a covering argument, it is enough to show that for every $v_0\in\R^d$ with $\|v_0\|=1$, every $a\in (0,a_{v_0})$ and every $\delta,\epsilon>0$, 
there exists $\kappa>0$ such that for \new{$\lambda, R>0$, if $N_0$ is sufficiently large and $\Lambda$ satisfies~\eqref{eq:Lambda_assumptions0},} for any initial condition $\xi$ as above,
\begin{equation}
\label{eq:shape_lower_point}
    \psub{\xi}{\exists t_0<\infty: S_t \cap \ball_{\epsilon t}(t(\mu + av_0)) \ne \emptyset\text{ for all $t\ge t_0$}} \ge 1-\delta.
\end{equation}
But this is done using the exact same contour argument as in the proof of Theorem~\ref{th:survival}, except that we take $\aminus  = a - \epsilon '$ and $\aplus  = a + \epsilon '$, for $0<\epsilon ' < \epsilon$ sufficiently small that $0\le \aminus  < \aplus  < a_{v_0}$, and take $T$ sufficiently large that $d^{1/2}T^{-1}+\epsilon '<\epsilon$. Here, the last event in the definition of the event $\mathcal C_e$ from \eqref{eq:Cdefn} is used to obtain \eqref{eq:shape_lower_point} for arbitrary $t\ge t_0$ instead of only for multiples of $T$.
\end{proof}

\section{Proof of Proposition~\ref{prop:bound_closed_edges}} \label{sec:propproof}

We will use the following notation throughout Sections~\ref{sec:propproof} and~\ref{sec:badbluepf}.
\new{Fix $R_1>0$.}
For $r\ge 0$, let 
\begin{align} 
f(r)&=\exp(-\tfrac r {9R_1} \log (r+1)) \label{eq:fdefn} \\
 \label{eq:gdefn}
\text{and } \quad g(r)&=1\wedge r^{-6d-2}.
\end{align}
The function $f$ will be used to bound the probability of some `bad events' which describe the spread of particles over distances of order $r$. The function $g$ on the other hand is used to control the number of such particles in these bad events.

We will prove Proposition~\ref{prop:bound_closed_edges} using the following two results.
The first result gives us a stochastic upper bound on the particle system at time $1$ that holds for any finite initial particle configuration. One may phrase this as a `coming down from infinity' property of the process. Its proof heavily relies on the presence of some competition between particles on the same site, i.e.~on \new{assumption~(ii) in~\eqref{eq:Lambda_assumptions}.}

For $\ep>0$, let $Z_\ep$ be a random variable taking values in $\N_0$, with
\begin{equation} \label{eq:Zepsdefn}
 \p{Z_\ep \ge r}=\min (f(r+\ep^{-1}),\ep)\; \text{ for }r\in \N.
\end{equation}
Let $(Z^{(x)}_\ep )_{x\in \Z^d}$ be i.i.d., with $Z^{(x)}_\ep \stackrel{d}{=}Z_\ep $ for each $x\in \Z^d$.
\new{Recall the definition of $N=N(\Lambda)$ in~\eqref{eq:Ndefn}.}
\begin{prop}[Coming down from infinity]\label{prop:upperbound}
\new{For $\gamma,\lambda>0$,}
there exists $K_0=K_0(d,\gamma,R_1,\lambda)>1$ such that the following holds.
For $\ep>0$, there exists $N_1>0$ such that 
\new{if $\Lambda$ satisfies conditions~(i) and~(ii) in~\eqref{eq:Lambda_assumptions} with $N_0= N_1$, if $p$ satisfies~\eqref{eq:p_assumptions} with $p(x)=0$ $\forall x\notin \mathcal B_{R_1}(0)$, and if}
$\xi$ is a starting configuration consisting of a finite number of particles,
then there exists a coupling between \new{a BRWNLC $(\xi_t)_{t\ge 0}$ with jump rate $\gamma$, jump kernel $p$ and competition kernel $\Lambda$},   and $(Z_\ep^{(x)})_{x\in \Z^d}$ such that under the coupling,
$\xi_0=\xi$ and
$$
\xi_1 = \zeta + \sum_{x\in \Z^d}\zeta^{(x)},
$$
where $\zeta$, $\zeta^{(x)}\in \N_0^{\Z^d}$ with $\zeta(y)\le K_0 N$ $\forall y\in \Z^d$ and $\zeta^{(x)}(y)\le g(Z_\ep^{(x)})^{-1}K_0 N \1_{\|y-x\|<Z_\ep^{(x)}}$ $\forall y\in \Z^d.$
\end{prop}
Proposition~\ref{prop:upperbound} will be proved in Section~\ref{sec:propupper}.
The next result gives us an upper bound on the probability that a collection of edges are all closed, which holds if the (random) initial particle configuration is bounded above in the same way as $\xi_1$ is bounded above in Proposition~\ref{prop:upperbound}.
\begin{prop} \label{prop:closedinitcond}
For \new{$\gamma>0$, $p$ satisfying~\eqref{eq:p_assumptions} with $p(x)=0$ $\forall x\notin \mathcal B_{R_1}(0)$,} $v_0\in \R^d \backslash \{0\}$ and $0\le \aminus <\aplus <a_{v_0}$, there exists $\kappa=\kappa(\gamma,p,v_0,a_-,a_+)>0$ such that the following holds.
For $T>1$ sufficiently large and \new{$\lambda, R>0$},
there exists $R'>1$ such that for
every $b\in (0,1)$, there exist $\ep>0$ and $N_2>1$ such that the following holds:
\new{Suppose $\Lambda$ satisfies~\eqref{eq:Lambda_assumptions} with $N_0=N_2$, and take}
$n\in\N$ and $e_1,e_2,\ldots ,e_n\in E_{T,\aminus ,\aplus ,v_0}(0)$. 
Suppose $\xi$ is a (random) configuration of a finite number of particles, and suppose there exists a coupling between $\xi$ and $(Z_\ep^{(x)})_{x\in \Z^d}$ such that under the coupling,
\begin{equation} \label{eq:coupling_upper}
\xi = \zeta + \sum_{x\in \Z^d}\zeta^{(x)},
\end{equation}
where $\zeta$, $\zeta^{(x)}$ are as in Proposition~\ref{prop:upperbound}.
Then, \new{letting $(\xi_t)_{t\ge 0}$ be a BRWNLC with jump rate $\gamma$, jump kernel $p$ and competition kernel $\Lambda$, and recalling the definition of $\mathcal C_e$ in~\eqref{eq:Cdefn} and letting $J=\lfloor N^{1/3}\rfloor$,}
$$
\E{\psub{\xi}{\bigcap_{i=1}^n \mathcal C_{e_i}}} \le b^n.
$$
\end{prop}
This result will be proved in Section~\ref{sec:propclosedinit} below, using the strategy outlined in Section~\ref{sec:proof_outline}.
We now show how Propositions~\ref{prop:upperbound} and~\ref{prop:closedinitcond} can be used to prove Proposition~\ref{prop:bound_closed_edges}.

\begin{proof}[Proof of Proposition~\ref{prop:bound_closed_edges}]
\new{Take $\gamma>0$ and $p$ satisfying~\eqref{eq:p_assumptions}.}
For $v_0\in\R^d \backslash \{0\}$ and $0\le \aminus <\aplus <a_{v_0}$, take $\kappa>0$ and $T>1$ such that Proposition~\ref{prop:closedinitcond} holds. Take $\new{\lambda, }R>0$ and then take $R'$ as in Proposition~\ref{prop:closedinitcond}. Take $b\in (0,1)$, and then take $\ep>0$ as in Proposition~\ref{prop:closedinitcond}, and assume \new{$\Lambda$ satisfies~\eqref{eq:Lambda_assumptions} with $N_0$} sufficiently large that Propositions~\ref{prop:upperbound} and~\ref{prop:closedinitcond} hold.

\new{We begin by proving the result under condition~1, and then use this and a recursive argument using the Markov property to prove the result under condition~2.}

\paragraph{\new{Condition~1:}}
The result under condition 1 follows from Propositions~\ref{prop:upperbound} and~\ref{prop:closedinitcond}.
Indeed, take $e_1,\ldots , e_n\in E_{T,\aminus ,\aplus ,v_0}(1)$ with $e_i=((x_i,T),(y_i,2T))$, let $\xi$ be an arbitrary particle configuration consisting of a finite number of particles, and take a BRWNLC $(\xi_t)_{t\ge 0}$ with $\xi_0=\xi$.
Since $\xi_{T-1}$ consists of a finite number of particles almost surely, by Proposition~\ref{prop:upperbound} there exists a coupling between $(\xi_t)_{t\ge 0}$ and $(Z^{(x)}_\ep)_{x\in \Z^d}$ such that under the coupling,
$$
\xi_T =\zeta +\sum_{x\in \Z^d}\zeta^{(x)},
$$
where $\zeta$, $\zeta^{(x)}\in \N_0^{\Z^d}$ with $\zeta(y)\le K_0 N$ $\forall y\in \Z^d$ and $\zeta^{(x)}(y)\le g(Z_\ep^{(x)})^{-1}K_0 N \1_{\|y-x\|<Z_\ep^{(x)}}$ $\forall y\in \Z^d.$
Hence letting $e'_i=((x_i-\lfloor (\mu +\aminus v_0)T \rfloor,0),(y_i-\lfloor (\mu +\aminus v_0) T \rfloor,T))\in E_{T,\aminus,\aplus,v_0}(0)$ for $i=1,\ldots,n$,
 by Proposition~\ref{prop:closedinitcond} we have
$$
\psub{\xi}{\bigcap_{i=1}^n \mathcal C_{e_i}}=\Esub{\xi}{\psub{\xi_T(\cdot +\lfloor (\mu +\aminus v_0)T \rfloor)}{\bigcap_{i=1}^n \mathcal C_{e'_i}}} \le b^n.
$$
This implies the result under condition~1.

\paragraph{\new{Condition~2:}}
\new{We begin by proving that the following special case follows directly from Proposition~\ref{prop:closedinitcond}.
Take $e_1,\ldots , e_n\in E_{T,\aminus ,\aplus ,v_0}(0)$ and suppose $\xi\in \N_0^{\Z^d}$ consists of a finite number of particles with $\xi(x)\le N$ $\forall x\in \Z^d$.
Then we can set $\zeta:=\xi$ and $\zeta^{(x)}:=0$ $\forall x\in \Z^d$, and then~\eqref{eq:coupling_upper} holds with $\zeta,\zeta^{(x)}$ satisfying the bounds in Proposition~\ref{prop:upperbound} (since $K_0>1$).
Therefore, by Proposition~\ref{prop:closedinitcond},
\begin{equation} \label{eq:Cei_initial}
\psub{\xi}{\bigcap_{i=1}^n \mathcal C_{e_i}} \le b^n.
\end{equation}
}

Now take $e_1,\ldots,e_n \in E_{T,\aminus,\aplus,v_0}$ arbitrary, and suppose $\xi \in \N_0^{\Z^d}$ consists of a finite number of particles with $\xi(x)\le N$ $\forall x\in \Z^d$.
We first divide the edges into two sets:
\[
E_{\text{odd}} = \{e_1,\ldots,e_n\} \cap \bigcup_{\text{$k$ odd}}E_{T,\aminus ,\aplus ,v_0}(k),\quad  E_{\text{even}} = \{e_1,\ldots,e_n\} \backslash E_{\text{odd}}.
\]
\new{Also, for $k\in \N_0$, let
\[
E_k=\{e_1,\ldots , e_n\}\cap E_{T,\aminus ,\aplus ,v_0}(k).
\]
}
We distinguish between two cases:

\textbf{Case 1:} $|E_{\text{odd}}| \ge n/2$. Let $k_m$ be the largest odd $k$ such that $\new{E_k} \ne \emptyset$. 
\new{Write $E_{k_m}=\{e_{i_j} : j\le |E_{k_m}|\}$ with $e_{i_j}=((x_j,k_m T), (y_j,(k_m+1)T))$ for each $j$.}
We then write
\begin{align} \label{eq:Ce_recurs}
\psub{\xi}{\bigcap_{i=1}^n \mathcal C_{e_i}} \le \mathbb P_{\xi}\bigg(\bigcap_{e\in E_{\text{odd}}} \mathcal C_e\bigg)&= \new{\mathbb E_{\xi}\bigg[\mathbb P\bigg( \bigcap_{e\in E_{\text{odd}}} \mathcal C_e\bigg| \F_{(k_m-1)T}\bigg)  \bigg]} \notag \\
&= \new{\mathbb E_{\xi}\bigg[\1_{\bigcap_{e\in E_{\text{odd}}\setminus E_{k_m}} \mathcal C_e} \; \mathbb P\bigg( \bigcap_{e\in E_{k_m}} \mathcal C_e\bigg| \F_{(k_m-1)T}\bigg)  \bigg]} \notag \\
&= \new{\mathbb E_{\xi}\bigg[\1_{\bigcap_{e\in E_{\text{odd}}\setminus E_{k_m}} \mathcal C_e} \; \mathbb P_{\xi_{(k_m-1)T}}\bigg( \bigcap_{j=1}^{|E_{k_m}|} \mathcal C_{e_{i_j}'}\bigg)  \bigg],}
\end{align}
\new{where $e_{i_j}'=((x_j,T), (y_j,2T))$ for each $j$, and the second line holds since for each $k$ odd with $k<k_m$, for $e\in E_k$ the event $\mathcal C_e$ is $\F_{(k+1)T}$-measurable and therefore $\F_{(k_m-1)T}$-measurable, and the last line follows by the Markov property at time $(k_m-1)T$.
Write for each $j\le |E_{k_m}|$,
\[e_{i_j}'':=((x_j -(k_m-1)\lfloor (\mu+a_-v_0)T\rfloor,T), (y_j-(k_m-1)\lfloor (\mu+a_-v_0)T\rfloor,2T))
\in E_{T,\aminus ,\aplus ,v_0}(1)
\]
(recall the definition of $E_{T,\aminus ,\aplus ,v_0}(k)$ in~\eqref{eq:ETkdefn}).
Then
\[
\mathbb P_{\xi_{(k_m-1)T}}\bigg( \bigcap_{j=1}^{|E_{k_m}|} \mathcal C_{e_{i_j}'}\bigg) 
=\mathbb P_{\xi_{(k_m-1)T}(\cdot + (k_m-1)\lfloor (\mu+a_-v_0)T\rfloor )}\bigg( \bigcap_{j=1}^{|E_{k_m}|} \mathcal C_{e_{i_j}''}\bigg) 
\le b^{|E_{k_m}|}
\]
by Proposition~\ref{prop:bound_closed_edges} under condition~1 (since $\xi_{(k_m-1)T}$ consists of a finite number of particles almost surely).
Therefore, substituting into~\eqref{eq:Ce_recurs},
\begin{equation} \label{eq:Ce_recurs2}
\psub{\xi}{\bigcap_{i=1}^n \mathcal C_{e_i}} \le b^{|E_{k_m}|} \, \mathbb P_{\xi}\bigg(\bigcap_{e\in E_{\text{odd}}\setminus E_{k_m}} \mathcal C_e\bigg).
\end{equation}
By applying the above argument recursively with $k_m-2,k_m-4,\ldots ,1$ in place of $k_m$, it follows that 
\[
\psub{\xi}{\bigcap_{i=1}^n \mathcal C_{e_i}} \le \prod_{k\text{ odd},\, 1\le k\le k_m} b^{|E_{k}|}=b^{|E_{\text{odd}}|}\le b^{n/2}
\]
by our assumption in Case~1. Since $b\in (0,1)$ was arbitrary, the result follows.
}

\textbf{Case 2:} $|E_{\text{odd}}| < n/2$. This is similar to the previous case, using the edges from $E_{\text{even}}$ instead \new{of the edges from $E_{\text{odd}}$.}
\new{Indeed, let $k_m'$ be the largest even $k$ such that $\new{E_k} \ne \emptyset$.
If $k_m'>0$, then by exactly the same argument as for~\eqref{eq:Ce_recurs2}, replacing $k_m$ with $k'_m$ and replacing $E_{\text{odd}}$ with $E_{\text{even}}$ throughout, we have 
\[
\psub{\xi}{\bigcap_{i=1}^n \mathcal C_{e_i}} \le b^{|E_{k'_m}|} \, \mathbb P_{\xi}\bigg(\bigcap_{e\in E_{\text{even}}\setminus E_{k'_m}} \mathcal C_e\bigg).
\]
Applying the argument recursively with $k'_m-2,k'_m-4,\ldots ,2$ in place of $k'_m$, it follows that 
\[
\psub{\xi}{\bigcap_{i=1}^n \mathcal C_{e_i}} \le \prod_{k\text{ even},\, 2\le k\le k'_m} b^{|E_{k}|}\,
\mathbb P_{\xi}\bigg(\bigcap_{e\in E_0 } \mathcal C_e\bigg)
\le b^{|E_{\text{even}}|}\le b^{n/2},
\]
where the second inequality follows from~\eqref{eq:Cei_initial}, using that $\xi(x)\le N$ $\forall x\in \Z^d$,
and the last inequality follows by our assumption in Case~2. Again, since $b\in (0,1)$ was arbitrary, the result follows.
}
\end{proof}

\subsection{Proof of Proposition~\ref{prop:upperbound} (coming down from infinity)} \label{sec:propupper}

Take a particle configuration $\xi$ consisting of a finite number of particles. 
We first introduce a convenient construction of the BRWNLC process \new{(with some jump rate $\gamma$, jump kernel $p$ and competition kernel $\Lambda$)} on the time interval $[0,1]$, building on the construction using the BRW trees with resiliences in Section~\ref{sec:construction}. 

Take $K_1>0$ a large constant and $c_0\in (0,1)$ a small constant to be fixed later.
We let $\xi_0=\xi$, and on the time interval $[0,1-c_0]$, we simply define the BRWNLC process $(\xi_t)_{t\in [0,1-c_0]}$ using (for example) the construction in Section~\ref{sec:construction}. From time $1-c_0$ onwards, the details of the construction will be important and we will split the particles into blue and red particles.

Initially, at time $1-c_0$, all particles will be blue.
The heuristic to have in mind for the construction below is that `well-behaved' particles remain blue, but if there are too many blue particles that have come from some site $x$ to some site $y$, these `badly-behaved' particles at $y$ will be coloured red.
Our construction will ensure that by time 1, there will be at most $2K_1N$ blue particles at each site, and by establishing an upper bound on the probability of particles from $x$ being coloured red at $y$ (and, in particular, using the decay of the upper bound in terms of $\|x-y\|$), we will be able to prove Proposition~\ref{prop:upperbound}.
We now give the details of the construction.

For $x\in \Z^d$, $y\in \Z^d$, $k\in \N$ and $j\in \N$, let $\mathcal T^{\mathrm{blue},x,k,j}$ and $\mathcal T^{\mathrm{red},x,y,k,j}$ be  i.i.d.~copies of the BRW tree with resiliences $\mathcal T$ (as defined in~\eqref{eq:calTdefn}). The corresponding entries will be denoted by $\mathcal N^*$, $X^*$, $\alpha^*$, $\beta^*$, $\rho^*$, where $*$ is replaced by the corresponding superscripts.
The process $(\xi_t)_{t\in [1-c_0,1]}$ 
will be constructed as a deterministic function of $\xi_{1-c_0}$ and these BRW trees with resiliences; the collections of trees $(\mathcal T^{\mathrm{blue},x,k,j})_{x,k,j}$ and $(\mathcal T^{\mathrm{red},x,y,k,j})_{x,y,k,j}$ will encode the behaviour of blue and red particles respectively.

Note that since $\xi$ consists of a finite number of particles, 
$\sup_{x\in \Z^d}\xi_{1-c_0}(x)$ is a.s.~finite.
Let 
\begin{equation} \label{eq:kmaxdefn}
\kmax=\left \lfloor \log _2 \Big(\max\Big(\frac 1 {K_1 N}\sup_{x\in \Z^d}{\xi_{1-c_0}(x)},2\Big) \Big) \right \rfloor \in \N,
\end{equation}
so that $\xi_{1-c_0}(x)\leq 2^{\kmax+1}K_1 N$ $\forall x\in \Z^d$.
Define times $t_0 > \cdots > t_{\kmax}$ such that $t_0 = 1$, $t_{\kmax} = 1-c_0$ and $t_k - t_{k+1} = 2(t_{k+1}-t_{k+2})$ for $k=0,\ldots,\kmax-2$. Explicitly, set
\begin{equation} \label{eq:tkdefn}
t_k = 1 - c_0\frac{1-2^{-k}}{1-2^{-\kmax}} \quad \text{for }0\le k \le \kmax.
\end{equation}
\new{See Figure~\ref{fig:tk} for an illustration.} Note that 
\begin{equation} \label{eq:tkdiff}
t_{k-1}-t_k =c_0 (1-2^{-\kmax})^{-1}2^{-k}\in (c_0 2^{-k},c_0 2^{1-k}] \quad \text{for }k\in\llbracket \kmax\rrbracket.
\end{equation}

\begin{figure}
\centering
\def\svgwidth{\columnwidth}
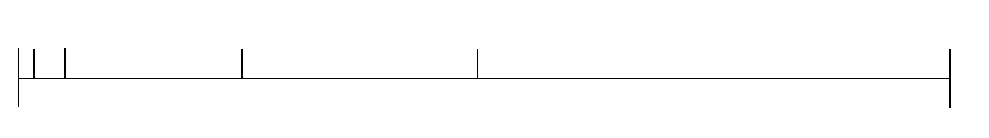
\caption{\label{fig:tk} \new{Illustration of the times $t_k$ defined in~\eqref{eq:tkdefn}.
}
}
\end{figure}

For $k\in\llbracket \kmax\rrbracket$, we successively construct 
the particle system on the time intervals $(t_k,t_{k-1}]$, starting with $k=\kmax$. The particles will be split into blue and red particles. 
Whenever a particle branches, both offspring particles inherit the colour of the parent.
At time $t_{\kmax} = 1-c_0$, all particles are blue. 
Blue particles may become red at times $t_{k-1}$, for $k\in\llbracket \kmax\rrbracket$, according to a rule that we will specify below. Red particles and their descendants remain red until time $1$.
For $0\le k \le \kmax$ and $x\in \Z^d$, we denote by $\xi^{\mathrm{blue},k}(x)$ the number of blue particles at $x$ at time $t_k$.  For $0\le k \le \kmax -1$, we need to keep track of the (blue) ancestor at time $t_{k+1}$ of the particles that turn red at time $t_{k}$. We therefore denote by $\xi^{\mathrm{red},x,k}(y)$ the number of particles at $y\in \Z^d$ which turn red at time $t_{k}$ and are descendants of a blue particle positioned at $x\in \Z^d$ at time $t_{k+1}$.

We start by defining the particle system on the time interval $(t_{\kmax},t_{\kmax-1}]$. At time $t_{\kmax}$, all particles are blue. The particle system is then constructed as in Section~\ref{sec:construction} from the BRW trees with resiliences $\mathcal T^{\mathrm{blue},x,\kmax,j}$ for $x\in\Z^d$, $j \in \llbracket \xi^{\mathrm{blue},\kmax}(x)\rrbracket$, with the root of $\mathcal T^{\mathrm{blue},x,\kmax,j}$ positioned at $x$ and with time starting at time $t_{\kmax}$. \new{(In particular, when particles are killed by competition, they and their descendants are removed from the trees.)} Recall~\eqref{eq:gdefn} and define a `threshold function'
\begin{equation} \label{eq:thrdefn}
\thr(r,k) \coloneqq c_1 2^{k}g(r)^{1/3} K_1 N \quad \text{for }r\ge 0 \text{ and }k\in \N_0,
\end{equation}
where $c_1\in (0,1)$ is chosen sufficiently small that
\begin{equation} \label{eq:c1defn}
\sum_{x\in \Z^d} c_1 g( \|x\|)^{1/3} \le 1.
\end{equation}
At time $t_{\kmax-1}$, some particles may be coloured red by the following rule: 
\emph{For each $x,y\in \Z^d$, if the blue particles at $x$ at time $t_{\kmax}$ have more than $\thr(\|x-y\|,\kmax)$ descendants at $y$ at time $t_{\kmax-1}$, then turn these descendants red at time $t_{\kmax-1}$.}

On the subsequent time intervals, the construction is similar, but takes into account the red particles created in the previous steps. Specifically, take $k\in \{1,2,\ldots , \kmax-1\}$. The particles living between times $t_k$ and $t_{k-1}$ will be subsets of the particles in the following BRW trees with resiliences:
\begin{itemize}
\item The descendants of the blue particles at time $t_k$ will come from the BRW trees with resiliences $\mathcal T^{\mathrm{blue},x,k,j}$ for $x\in\Z^d$, $j\in\llbracket \xi^{\mathrm{blue},k}(x)\rrbracket$, with the root of $\mathcal T^{\mathrm{blue},x,k,j}$ positioned at $x$ at time $t_k$. 
\item The descendants of particles which have turned red at a time $t_{k'}$, $k'\ge k$, will come from the BRW trees with resiliences $\mathcal T^{\mathrm{red},x,y,k',j}$ for $x,y\in\Z^d$ and $j\in\llbracket  \xi^{\mathrm{red},x,k'}(y)\rrbracket$, with the root of $\mathcal T^{\mathrm{red},x,y,k',j}$ positioned at $y$ at time $t_{k'}$. 
\end{itemize}
Using these BRW trees with resiliences, the particle system is constructed on the time interval $(t_k,t_{k-1}]$ by an obvious extension of the construction from  Section~\ref{sec:construction}, taking into account the different times at which the particles corresponding to the roots of the trees appear, and removing from the trees $\mathcal T^{\mathrm{red},x,y,k',j}$ with $k'>k$ those particles that have been killed by competition before time $t_k$ (and their descendants). 
\new{In particular, as in the construction in Section~\ref{sec:construction}, between times $t_k$ and $t_{k-1}$, when particles  in the trees $\mathcal T^{\mathrm{blue},x,k,j}$ and $\mathcal T^{\mathrm{red},x,y,k',j}$ are killed by competition, they and their descendants are removed.}
At time $t_{k-1}$, some blue particles may be coloured red by the same rule as above: \emph{For each $x,y\in \Z^d$, if the blue particles at $x$ at time $t_k$ have more than $\thr(\|x-y\|,k)$ descendants at $y$ at time $t_{k-1}$, then turn these descendants red at time $t_{k-1}$.} 
This completes the construction of the particle system up to time $t_0=1$.

Note that 
by the definition of $\kmax$ in~\eqref{eq:kmaxdefn}, we have $\xi^{\mathrm{blue},\kmax}(x)\leq 2^{\kmax+1}K_1N$ $\forall x\in \Z^d$. Moreover,
for $1\le k < \kmax$ and $x\in \Z^d$, by our construction,
\begin{equation*}
\xi^{\mathrm{blue},k}(x)\le \sum_{y\in\Z^d} \thr(\|y-x\|,k+1) = \sum_{y\in \Z^d} c_1 2^{k+1} g(\|y\|)^{1/3} K_1 N \le 2^{k+1}K_1 N
\end{equation*}
where the second inequality follows by~\eqref{eq:c1defn}. Summarising,  we have the following deterministic bound on the number of blue particles:
\begin{equation} \label{eq:xibluekbd}
\xi^{\mathrm{blue},k}(x)\le 2^{k+1}K_1 N \quad \forall x\in\Z^d,\ k\in\llbracket \kmax\rrbracket .
\end{equation}
It remains to control the number of red particles created during the process.

\paragraph{Controlling the creation of red particles.}

For $x,y\in \Z^d$, let $\eta^{(x)}(y)$ denote the number of red particles at $y$ at time $1$ whose ancestor turned red at time $t_{k-1}$ for some $1\le k \le \kmax$ and whose time-$t_k$ blue ancestor was at $x$.
In order to control the creation of red particles, we define a family of `bad events' $(R_{x,r})_{x\in\Z^d,\,r\in\N_0}$ which have the following properties:
\begin{itemize}
\item If $R_{x,r}$ does not occur, then at time $1$ there are at most $O(r^dg(r)^{-1}N)$ red particles whose ancestor turned red at some time $t_{k-1}$ and whose time-$t_k$ blue ancestor was at $x$, and at most $O(r^{d-1}g(r)N)$ of those particles are outside $\ball_{r}(x)$ (see Lemma~\ref{lem:ZRxr}).
\item The probability of $R_{x,r}$ is small and decays like $f(r)$ when $r\to\infty$ (see Lemma~\ref{lem:bad_red}).
\item For every $x\in\Z^d$, the events $(R_{x,r})_{r\ge0}$ only depend on the BRW trees with resiliences $(\mathcal T^{\mathrm{blue},x,k,j})_{k,j}$ and $(\mathcal T^{\mathrm{red},x,y,k,j})_{y,k,j}$. In particular, if $x_1,\ldots,x_n\in\Z^d$ are pairwise distinct and $r_1,\ldots,r_n\in \N_0$, then $R_{x_1,r_1},\ldots,R_{x_n,r_n}$ are independent.
\end{itemize}

We note that these events will also depend on the constants $K_1>0$, $c_0\in (0,1)$ and $c_1>0$.

The events $R_{x,r}$ will be defined as a union of several other bad events. The formal definitions are given below. Informally, the role of these events is as follows:
\begin{itemize}
\item $(R^{\mathrm{cr}}_{x,k})_{x\in\Z^d,k\in\N}$:  \new{for $k\le \kmax,$}  if the event $R^{\mathrm{cr}}_{x,k}$ does not occur, then no descendants of blue particles at $x$ at time $t_k$ are turned red at time $t_{k-1}$ (the abbreviation `cr' stands for `creation').
\item $(R^{\mathrm{num}}_{x,k,r})_{x\in\Z^d,k,r\in\N}$: \new{for $k\le \kmax,$} if the event $R^{\mathrm{num}}_{x,k,r}$ does not occur, then at most $g(r)^{-1/5} 2^{k+2}K_1 N$ descendants of blue particles at $x$ at time $t_k$ turn red at time $t_{k-1}$, and none of the descendants turn red outside $\ball_{r/2}(x)$  (the abbreviation `num' stands for `number').
\item $(R_{x,y,k,r})_{x,y\in\Z^d,k,r\in\N}$: \new{for $k\le \kmax,$} if the event $R_{x,y,k,r}$ does not occur and $\xi^{\mathrm{red},x,k}(y) \le g(r)^{-1/5} 2^{k+3}K_1 N$, then these $ \xi^{\mathrm{red},x,k}(y)$ red particles at $y$ at time $t_{k}$ will have at most $r^{-d+o(1)}g(r)^{-1}N$ descendants and at most $(r')^{-d+o(1)}g(2r')N$ of them spread further than $r'$, for every $r' \ge r/2$.
\end{itemize}
\new{Note that even though the \textit{roles} of the events described above only make sense for $k\le \kmax$, we will define the events $R^{\mathrm{cr}}_{x,k}$, $R^{\mathrm{num}}_{x,k,r}$ and $R_{x,y,k,r}$ themselves in such a way that they are well-defined for all $k\in \N$. 
We do this because we want an upper bound on $\xi_1$ that is uniform in $\kmax$ (in order to ensure that it is uniform in the initial configuration $\xi$), and so we want our bad events to only depend on the decorated BRW trees $(\mathcal T^{\mathrm{blue},x,k,j})_{x,k,j}$ and $(\mathcal T^{\mathrm{red},x,y,k,j})_{x,y,k,j}$ and deterministic quantities.
}

We now give the formal definition of these events, which the reader might skip at first reading. 
\new{Take $\lambda >0$; the events below will depend on $\lambda,N,K_1,c_0$ and $c_1$ (we will later assume that the competition kernel $\Lambda$ satisfies~(i) and (ii)~in~\eqref{eq:Lambda_assumptions} with $N_0$ sufficiently large).}
 For $x\in \Z^d$ and $k\in \N$, set
$$
R^{\mathrm{cr}}_{x,k} = \bigcup_{i=1}^4 R^{\mathrm{cr},i}_{x,k},
$$
where 
\begin{align*}
R^{\mathrm{cr},1}_{x,k} &= \left\{\#\{j\in\llbracket 2^{k+1}K_1 N \rrbracket : \rho^{\mathrm{blue},x,k,j}(\emptyset) \ge \tfrac 14 \lambda K_1 c_0 c_1\} > 
c_1 2^{k-2}K_1N\right\},\\
R^{\mathrm{cr},2}_{x,k} &= \left\{\#\{(j,u): j\in\llbracket 2^{k+1}K_1 N \rrbracket,\,u\in \mathcal N^{\mathrm{blue},x,k,j}_{c_0 2^{1-k}},u\neq \emptyset\} > c_1 2^{k-2}K_1 N\right\},\\
R^{\mathrm{cr},3}_{x,k} &= \left\{\#\Big\{j\in\llbracket 2^{k+1}K_1 N\rrbracket: \sup_{t \in [0, (c_0 2^{1-k})\wedge \beta^{\mathrm{blue},x,k,j}(\emptyset)]}\| X^{\mathrm{blue},x,k,j}_t(\emptyset)\|> 0\Big\} > c_1 2^{k-2}K_1 N\right\},\\
R^{\mathrm{cr},4}_{x,k} &= \Big\{\exists r\ge1: \#\{(j,u): j\in\llbracket 2^{k+1}K_1 N \rrbracket,\,u\in \mathcal N^{\mathrm{blue},x,k,j}_{c_0 2^{1-k}},\\
&\hspace{4.5cm}\sup_{s\in [0,c_0 2^{1-k}]}\|X^{\mathrm{blue},x,k,j}_{s}(u)\| \ge r\} > \thr(r,k)\Big\}.
\end{align*}
Next, for $x\in \Z^d$, $k\in \N$ and $r\ge 1$, set
\begin{align*}
R^{\mathrm{num}}_{x,k,r} 
&= \left\{\#\{(j,u): j\in \llbracket 2^{k+1}K_1 N\rrbracket,\,u\in \mathcal N^{\mathrm{blue},x,k,j}_{c_0 2^{1-k}}\} >  g(r)^{-1/5} 2^{k+2}K_1 N\right\}\\
&\qquad \cup \Big\{\exists r'\geq r/2:\#\{(j,u): j\in \llbracket 2^{k+1}K_1 N\rrbracket,\,u\in \mathcal N^{\mathrm{blue},x,k,j}_{c_0 2^{1-k}},\\
&\hspace{6 cm} \sup_{s\in [0,c_0 2^{1-k}]} \|X^{\mathrm{blue},x,k,j}_{s}(u)\|\ge r'\} > \thr(r',k)\Big\}.
\end{align*}
Now, set
\begin{equation} \label{eq:krdefn}
k_r = \lfloor (2d+3)\log_2 (r+1)\rfloor \in \N \quad \text{for } r\ge 1.
\end{equation}
For $x,y\in \Z^d$, $k\in \N$ and $r\in \N$, set
\begin{align*}
R_{x,y,k,r} 
&= \Big\{\exists r'\ge r/2: \#\{(j,u): j\in \llbracket g(r)^{-1/5} 2^{k+3}K_1 N\rrbracket,\,u\in \mathcal N^{\mathrm{red},x,y,k,j}_{c_0},\\
&\hspace{5cm}  \sup_{t\in[0,c_0]}\|X^{\mathrm{red},x,y,k,j}_t(u)\|\ge r' \} > k^{-1}_{2 r'} |\ball_{r'}(0)|^{-1} g(2r') K_1^2 N\Big\}\\
&\quad \cup \{\#\{(j,u): j\in \llbracket g(r)^{-1/5} 2^{k+3}K_1 N\rrbracket,\,u\in \mathcal N^{\mathrm{red},x,y,k,j}_{c_0}\}\ge  \tfrac 12 k_r^{-1} |\ball_{r/2}(0)|^{-1} g(r)^{-1}K_1^2 N\}.
\end{align*}
Finally, for $x\in \Z^d$ and $r\in\N_0$, we set
\begin{align}
R_{x,0}&=\bigcup_{k=1}^{\infty} R^{\mathrm{cr}}_{x,k} \label{eq:Rx0defn} \\
\text{and }\quad R_{x,r} &= \bigcup_{k=k_r+1}^{\infty} R^{\mathrm{cr}}_{x,k} \cup \bigcup_{k=1}^{k_r} R^{\mathrm{num}}_{x,k,r} \cup \bigcup_{y\in \ball_{r/2}(x)}\bigcup_{k=1}^{k_r-1} R_{x,y,k,r} \quad \text{for }r\in \N. \label{eq:Rxrdefn}
\end{align}

We now choose the constants $c_0$ and $K_1$, depending on the parameters $d,\gamma,R_1$ and $\lambda$. 
(Recall that we already chose $c_1$ in~\eqref{eq:c1defn}.)
This can also be skipped at first reading.
Take $c_0\in (0, \min(e^{-1}, (R_1 \gamma e)^{-1}, (2e^{2R_1+1}R_1\gamma )^{-1}))$ sufficiently small that
\begin{equation} \label{eq:c0defn}
c_0e^2/(1-c_0e)\le c_1/16, \;
1-ec_0 \ge e^{-c_1/32}, \;
1-e^{-\gamma c_0}\le c_1/16,
\end{equation}
and 
\begin{equation} \label{eq:c0defn2}
e^{c_0-\frac r {R_1}\log \left(\frac r {R_1 \gamma c_0 e}\right)}
\le \tfrac 14 c_1 g(r+1)^{1/3} \quad \forall r\ge 1,
\end{equation}
and
\begin{equation} \label{eq:c0defn3}
e(1-ec_0)^{-1} (ec_0)^{2r-1}+(2e^{2R_1+1}R_1 c_0 \gamma)^{r/R_1}\le \tfrac 14 c_1 g(r+1)^{1/3} \quad \forall r\ge 1.
\end{equation}
\new{Before choosing $K_1$, we make the following observations.
Note that by the definition of $k_r$ in~\eqref{eq:krdefn} and the definition of $g$ in~\eqref{eq:gdefn}, 
and since $\mathcal B_{r/2}(0)\subseteq \{(x_1,\ldots ,x_d)\in \Z^d:|x_i|<r/2\, \forall i\}$,
we have that for $r\ge 1$,
\begin{align*}
2^{k_r+4}\cdot 4k_r |\mathcal B_{r/2}(0)|g(r)^{4/5}
&\le 2^6 (2d+3)\log_2 (r+1)\cdot (r+1)^{2d+3}(r+1)^d r^{-\frac 45 (6d+2)}\\
&\le 2^{3d+9}(2d+3)\log_2(r+1) r^{-\frac 95 d+\frac 75},
\end{align*}
since $r+1\le 2r$.
Therefore, since $d\ge 1$ and so $-\frac 95 d+\frac 75<0$, we have
\begin{equation} \label{eq:K11}
K_1^{(1)}:=\sup_{r\ge 1}\left( 2^{k_r+4}\cdot 4k_r |\mathcal B_{r/2}(0)|g(r)^{4/5} \right)<\infty.
\end{equation}
Moreover, because $2^{k_{2(r+1)}+3} e^{c_0} \cdot 2 k_{2(r+1)} |\ball_{r+1}(0)| g(2(r+1))^{-6/5} $ is bounded above by a polynomial in $r$ for $r\ge 1$, we have that 
\begin{equation} \label{eq:K12}
K_1^{(2)}:=\sup_{r\ge 1}\left( 2^{k_{2(r+1)}+3} e^{c_0} \cdot 2 k_{2(r+1)} |\ball_{r+1}(0)| g(2(r+1))^{-6/5} e^{-\frac{r}{R_1}\log \left( \frac{r}{R_1 \gamma c_0 e}\right)} \right)<\infty.
\end{equation}
}

\new{We can now choose $K_1$:} take $K_1>\new{\max(16, K_1^{(1)}, K_1^{(2)})}$ sufficiently large that
\begin{equation} \label{eq:K0defn}
16 e^{-\frac 14 \lambda K_1 c_0 c_1 }<c_1.
\end{equation}
\new{Note that by~\eqref{eq:K11} and since $K_1>K_1^{(1)}$ we have}
\begin{equation} \label{eq:K0defn2}
\tfrac 14 k_r^{-1} |\ball_{r/2}(0)|^{-1} g(r)^{-4/5}K_1 >2^{k_r+4} \quad \forall r\ge 1,
\end{equation}
and \new{by~\eqref{eq:K12} and since $K_1>K_1^{(2)}$ we have}
\begin{equation} \label{eq:K0defn3}
2^{k_{2(r+1)}+3} e^{c_0} e^{-\frac{r}{R_1}\log \left( \frac{r}{R_1 \gamma c_0 e}\right)}\le \tfrac 12 k_{2(r+1)}^{-1} |\ball_{r+1}(0)|^{-1} g(2(r+1))^{6/5}K_1 \quad \forall r\ge 1.
\end{equation}

Now we can state the two main intermediate lemmas in the proof of Proposition~\ref{prop:upperbound}.
Recall that for $x,y\in \Z^d$, we let $\eta^{(x)}(y)$ denote the number of red particles at $y$ at time 1 whose ancestor turned red at some time $t_{k-1}$ and whose time-$t_k$ blue ancestor was at $x$.
The first lemma gives a deterministic bound on $\eta^{(x)}$ if the bad event $R_{x,r}$ does not occur.
\begin{lemma} \label{lem:ZRxr}
\new{Suppose the competition kernel $\Lambda$ satisfies~(ii) in~\eqref{eq:Lambda_assumptions}. Then}
for $x\in \Z^d$ and $r\in \N_0$,
$$
(R_{x,r})^c \subseteq \{\eta^{(x)}(y)\le g(r)^{-1} K_1^2 N \1_{\|y-x\|<r} +g(\|y-x\|) K_1^2 N \1_{\|y-x\|\ge r}\; \forall y\in \Z^d\}.
$$
\end{lemma}
We will prove Lemma~\ref{lem:ZRxr} at the end of this section.
The second lemma bounds the probability of the bad event $R_{x,r}$; we will prove this result in Section~\ref{sec:badbluepf}.
\begin{lemma}
\label{lem:bad_red}
For every $\ep>0$, there exists $N_3>0$ such that if $N\geq N_3$,
\new{and if $p$ satisfies~\eqref{eq:p_assumptions} with $p(x)=0$ $\forall x\notin \mathcal B_{R_1}(0)$,}
$$\P(R_{x,r}) \le \min(\ep, f(r+1+\ep ^{-1}))\quad  \forall x\in \Z^d, r\in \N_0.$$
\end{lemma}

We now show how Lemmas~\ref{lem:ZRxr} and~\ref{lem:bad_red} can be used to prove Proposition~\ref{prop:upperbound}.
\begin{proof}[Proof of Proposition~\ref{prop:upperbound}]
Take $K_0>1$ sufficiently large that
\begin{equation} \label{eq:chooseK0}
K_0 >\max \Big(K_1^2 , 2K_1 +\sum_{x\in \Z^d}g(\|x\|) K_1^2 \Big),
\end{equation}
and fix $\ep>0$.
\new{Take $N_3$ as in Lemma~\ref{lem:bad_red}. Suppose $\Lambda$ satisfies conditions~(i) and~(ii) in~\eqref{eq:Lambda_assumptions} with $N_0= N_3$, and $p$ satisfies~\eqref{eq:p_assumptions} with $p(x)=0$ $\forall x\notin \mathcal B_{R_1}(0)$, and construct the BRWNLC with jump rate $\gamma$, jump kernel $p$ and competition kernel $\Lambda$ as at the beginning of this section.}
Recall that we let $\xi^{\mathrm{blue},0}(x)$ denote the number of blue particles at $x\in \Z^d$ at time $t_0=1$.
Then for $y\in \Z^d$, we have
\begin{equation} \label{eq:xi1blueeta}
\xi_1(y)=\xi^{\mathrm{blue},0}(y)+\sum_{x\in \Z^d}\eta^{(x)}(y).
\end{equation}
For $y\in \Z^d$, by our construction, and then using~\eqref{eq:thrdefn},
\begin{equation} \label{eq:xiblue0bound}
\xi^{\mathrm{blue},0}(y)\le \sum_{x\in \Z^d} \thr(\|x-y\|,1) = \sum_{x\in \Z^d} c_1 2 g(\|x\|)^{1/3} K_1 N \le 2 K_1 N
\end{equation}
by~\eqref{eq:c1defn}.
Note that for $x\in \Z^d$, the events $(R_{x,r})_{r\in \N_0}$ depend only on the BRW trees $(\mathcal T^{\mathrm{blue},x,k,j})_{k,j\in \N}$ and $(\mathcal T^{\mathrm{red},x,y,k,j})_{y\in \Z^d,k,j\in \N}$.
By Lemma~\ref{lem:bad_red}, we have that \new{since $N=N(\Lambda)\ge N_3$}, for $x\in \Z^d$ and $r\in \N_0$,
$$
\p{\bigcap_{r'=0}^r R_{x,r'} }\le \p{R_{x,r}}\le \min(\ep,f(r+1+\ep^{-1})) = \p{Z_\ep^{(x)}\ge r+1},
$$
where the last equality follows by~\eqref{eq:Zepsdefn}.
Therefore we can couple $(\mathcal T^{\mathrm{blue},x,k,j})_{x\in \Z^d, k,j\in \N}$ and $(\mathcal T^{\mathrm{red},x,y,k,j})_{x,y\in \Z^d, k,j\in \N}$
with $(Z^{(x)}_\ep)_{x\in \Z^d}$ in such a way that 
for each $x\in \Z^d$, for $r\in \N_0$,
if $Z_\ep^{(x)}\le  r$ then $\cup_{r'=0}^r (R_{x,r'})^c$ occurs.
\new{In particular, for each $x\in \Z^d$, since $Z^{(x)}_\ep \in \N_0$, by the previous sentence with $r=Z^{(x)}_\ep$ we have that $\cup_{r'=0}^{Z^{(x)}_\ep} (R_{x,r'})^c$ occurs. Therefore}
under this coupling, for each $x\in \Z^d$, there exists $r_x\in \N_0 \cap [0,Z^{(x)}_\ep]$ such that $(R_{x,r_x})^c$ occurs, and so
by Lemma~\ref{lem:ZRxr}, for $y\in \Z^d$,
\begin{align} \label{eq:etaxyZbd}
\eta^{(x)}(y)&\le g(r_x)^{-1} K_1^2 N \1_{\|y-x\|<r_x}+g(\|y-x\|) K_1^2 N \1_{\|y-x\|\ge r_x}
\notag \\
&\le g(Z^{(x)}_\ep)^{-1} K_1^2 N \1_{\|y-x\|<Z^{(x)}_\ep }+g(\|y-x\|) K_1^2 N \1_{\|y-x\|\ge Z^{(x)}_\ep},
\end{align}
where the second line follows since $r\mapsto g(r)^{-1}$ is non-decreasing and $g(r)\le 1$ for $r\ge 0$ by~\eqref{eq:gdefn}.
By~\eqref{eq:xi1blueeta}, for $y\in \Z^d$ we can write 
$$
\xi_1(y)=\Big( \xi^{\mathrm{blue},0}(y)+\sum_{x\in \Z^d}\eta^{(x)}(y)\1_{\|y-x\|\ge Z^{(x)}_\ep}\Big)+ \sum_{x\in \Z^d}\eta^{(x)}(y)\1_{\|y-x\|<Z^{(x)}_\ep }.
$$
For $x,y\in \Z^d$, let
$$
\zeta(y)=\xi^{\mathrm{blue},0}(y)+\sum_{z\in \Z^d}\eta^{(z)}(y)\1_{\|y-z\|\ge Z^{(z)}_\ep}
\quad \text{and}\quad
\zeta^{(x)}(y)=\eta^{(x)}(y)\1_{\|y-x\|<Z_\ep^{(x)}}.
$$
Then $\xi_1=\zeta +\sum_{x\in \Z^d}\zeta^{(x)}$, and by~\eqref{eq:etaxyZbd} and~\eqref{eq:chooseK0}, for $x,y\in \Z^d$,
$$\zeta^{(x)}(y)\le g(Z^{(x)}_\ep)^{-1} K_0 N \1_{\|y-x\|<Z^{(x)}_\ep }.$$
Moreover,
by~\eqref{eq:xiblue0bound} and~\eqref{eq:etaxyZbd},
$$
\zeta(y)
\le 2K_1 N+\sum_{z\in \Z^d}g(\|y-z\|) K_1^2 N \le K_0 N,
$$
where the second inequality follows by~\eqref{eq:chooseK0}.
\end{proof}
We now finish this section by proving Lemma~\ref{lem:ZRxr}. We first prove that on the event $(R^{\mathrm{cr}}_{x,k})^c$, no descendants of blue particles at $x$ at time $t_k$ are turned red at time $t_{k-1}$:
\begin{lemma}
\label{lem:no_red} 
\new{Suppose the competition kernel $\Lambda$ satisfies~(ii) in~\eqref{eq:Lambda_assumptions}. Then}
for $x\in \Z^d$ and $k\in \llbracket \kmax\rrbracket$, 
\[
(R^{\mathrm{cr}}_{x,k})^c \subseteq \Big\{\xi^{\mathrm{red},x,k-1}(y)=0 \, \forall y \in \Z^d\Big\}.
\]
\end{lemma}
\begin{proof}[Proof of Lemma~\ref{lem:no_red}]
We first explain the heuristics behind the argument. We differentiate between the particles turning red \new{at $x$ and those which turn red at some site $y\ne x$.} 

\new{The proof consists of the following steps:
\begin{itemize}
\item Step 1: We suppose the event $R^{\mathrm{cr},1}_{x,k}$ does not occur, and we consider the number of particles that stay at the site $x$ until time $t_k +t$ without branching or jumping, and are not killed by competition (we will call this quantity $S^{x,k}_t$ for $t\le t_{k-1}-t_k$).
If the number of particles surviving at $x$ until time $t_{k-1}$ is large then the on-site competition assumption ((ii) in~\eqref{eq:Lambda_assumptions}) implies that only particles with high resilience can survive. But on the event $(R^{\mathrm{cr},1}_{x,k})^c$, only a small number of the blue particles starting at $x$ have high resilience, which tells us that the number of particles surviving at $x$ until time $t_{k-1}$ (without branching or jumping) must be small.
\item Step 2: We suppose the event $R^{\mathrm{cr}}_{x,k}$ does not occur.
We now consider the total number of particles at $x$ at time $t_{k-1}$ descended from blue particles at $x$ at time $t_k$; such particles must be either particles that survive at $x$ without branching or jumping, or particles that branch before time $t_{k-1}$, or particles that jump before time $t_{k-1}$. Using Step 1, the definition of the event $(R^{\mathrm{cr},2}_{x,k})^c$ and the definition of the event $(R^{\mathrm{cr},3}_{x,k})^c$ respectively, we can bound the number of each of these sets of particles. This allows us to show that the total number of particles at $x$ at time $t_{k-1}$ descended from blue particles at $x$ at time $t_k$ is less than the threshold, and so none of these particles turn red at $x$.
\item Step 3: We suppose the event $R^{\mathrm{cr},4}_{x,k}$ does not occur. For each $y\ne x$, on the event $(R^{\mathrm{cr},4}_{x,k})^c$, the number of descendants at $y$ at time $t_{k-1}$ of the blue particles at $x$ does not exceed the threshold, and so none are coloured red.
\end{itemize}}

We now get to the formal details. Let $x\in \Z^d$ and $k\in \llbracket \kmax\rrbracket$. Recall that $\xi^{\mathrm{blue},k}(x) \le 2^{k+1}K_1N$ by \eqref{eq:xibluekbd}. The descendants on the time interval $[t_k,t_{k-1}]$ of the blue particles at $x$ at time $t_k$ are encoded by the BRW trees with resiliences $(\mathcal T^{\mathrm{blue},x,k,j})_{j\leq \xi^{\mathrm{blue},k}(x)}.$ For $j\leq \xi^{\mathrm{blue},k}(x)$ and $t\leq t_{k-1}-t_k$, write
\begin{equation} \label{eq:calAdefn}
\mathcal A_t^{\mathrm{blue},x,k,j}:=
\Big\{u \in \mathcal N_t^{\mathrm{blue},x,k,j}: \rho^{\mathrm{blue},x,k,j}(v) \ge \int_{\alpha^{\mathrm{blue},x,k,j}(v)}^{\beta^{\mathrm{blue},x,k,j}(v)\wedge t}K_{t_k +s}(x+X^{\mathrm{blue},x,k,j}_s(v))\, ds \; \forall v\prec u \Big\}
\end{equation} 
for the set of particles at time $t_k+t$ from the BRW tree $\mathcal T^{\mathrm{blue},x,k,j}$ which have not been killed by competition, and whose ancestors were not killed by competition.

\paragraph{\new{\textbf{Step~1:}}}
Suppose that the event $(R^{\mathrm{cr},1}_{x,k})^c$ occurs.
For $t\in [0,t_{k-1}-t_k]$, let
\begin{equation} \label{eq:Sxktdefn}
S^{x,k}_t :=\#\{j\leq \xi^{\mathrm{blue},k}(x):\emptyset \in \mathcal A_t^{\mathrm{blue},x,k,j}, \sup_{s\in [0,t]} \|X_s^{\mathrm{blue},x,k,j}(\emptyset)\|=0 \}.
\end{equation}
Then $S^{x,k}_t$ is non-increasing in $t$.
Note that $\xi_t(x)\ge S^{x,k}_{t-t_k}$ for $t\in [t_k,t_{k-1}]$.
Suppose (aiming for a contradiction) that
$S^{x,k}_{t_{k-1}-t_k}> c_1 2^{k-2}K_1N$.
Then for $t\in [t_k,t_{k-1}]$, \new{using~(ii) in~\eqref{eq:Lambda_assumptions}} in the second inequality,
$$K_t(x)
\geq \Lambda(0)\xi_t(x)
\geq \lambda N^{-1}S^{x,k}_{t-t_k}
> \lambda c_1 2^{k-2}K_1 
.$$
It follows by~\eqref{eq:calAdefn} that for $j\leq \xi^{\mathrm{blue},k}(x)$, if $\emptyset \in \mathcal A ^{\mathrm{blue},x,k,j}_{t_{k-1}-t_k}$ with $X_s^{\mathrm{blue},x,k,j}(\emptyset)=0$ $ \forall s\leq t_{k-1}-t_k$ then $\rho^{\mathrm{blue},x,k,j}(\emptyset)\geq \lambda c_1 2^{k-2}K_1 (t_{k-1}-t_k)$.
Note that 
$2^{k-2}(t_{k-1}-t_k)\geq \frac 14 c_0$ by~\eqref{eq:tkdiff},
and so it follows by~\eqref{eq:Sxktdefn} that
$$
S^{x,k}_{t_{k-1}-t_k}\le \#\{j\leq \xi^{\mathrm{blue},k}(x):\rho^{\mathrm{blue},x,k,j}(\emptyset)\ge \tfrac 14 \lambda  K_1 c_0 c_1\}.
$$
By the definition of the event $(R^{\mathrm{cr},1}_{x,k})^c$ and by~\eqref{eq:xibluekbd} we have that 
$$
\#\{j\leq \xi^{\mathrm{blue},k}(x):\rho^{\mathrm{blue},x,k,j}(\emptyset)\ge \tfrac 14 \lambda  K_1 c_0 c_1\} \le c_1 2^{k-2}K_1 N,
$$
which contradicts the assumption that $S^{x,k}_{t_{k-1}-t_k}> c_1 2^{k-2}K_1N$. Therefore $S^{x,k}_{t_{k-1}-t_k}\leq c_1 2^{k-2}K_1N$.

\paragraph{\new{\textbf{Step~2:}}}
Now suppose that the event $(R^{\mathrm{cr}}_{x,k})^c$ occurs.
Then since $\mathcal A_{t_{k-1}-t_k}^{\mathrm{blue},x,k,j}\subseteq \mathcal N_{t_{k-1}-t_k}^{\mathrm{blue},x,k,j}$ for each $j\le \xi^{\mathrm{blue},k}(x)$, and by the definition of $S_t^{x,k}$ in~\eqref{eq:Sxktdefn},
\begin{align*}
&\#\{(j,u):j\leq \xi^{\mathrm{blue},k}(x),u \in \mathcal A^{\mathrm{blue},x,k,j}_{t_{k-1}-t_k}\}\\
&\quad \leq S^{x,k}_{t_{k-1}-t_k}
 +\#\{(j,u):j\leq \xi^{\mathrm{blue},k}(x),u \in \mathcal N^{\mathrm{blue},x,k,j}_{t_{k-1}-t_k}, u\neq \emptyset\}\\
&\qquad\quad + \#\{j\leq \xi^{\mathrm{blue},k}(x): \emptyset \in \mathcal N_{t_{k-1}-t_k}^{\mathrm{blue},x,k,j}, \sup_{t\in [0,t_{k-1}-t_k]}\|X^{\mathrm{blue},x,k,j}_t(\emptyset)\|> 0\}\\
&\quad \leq c_1 2^{k-2}K_1 N +c_1 2^{k-2}K_1 N +c_1 2^{k-2}K_1 N\\
&\quad < c_1 2^{k}K_1 N,
\end{align*}
where the second inequality follows by the definition of the events $(R^{\mathrm{cr},2}_{x,k})^c$ and $(R^{\mathrm{cr},3}_{x,k})^c$ and since $\xi^{\mathrm{blue},k}(x)\le 2^{k+1}K_1 N$ by~\eqref{eq:xibluekbd} and $t_{k-1}-t_k\le c_0 2^{1-k}$ by~\eqref{eq:tkdiff}.
Hence, in particular, the blue particles at $x$ at time $t_k$ have less than $c_1 2^k K_1 N$ descendants at $x$ at time $t_{k-1}$, and so,
since $g(0)=1$ and therefore $\mathrm{thr}(0,k)=c_12^kK_1 N$, by our construction we have $\xi^{\mathrm{red},x,k-1}(x)=0$.

\paragraph{\new{\textbf{Step~3:}}}
\new{Suppose that the event $(R^{\mathrm{cr},4}_{x,k})^c$ occurs.}
By the definition of the event $(R^{\mathrm{cr},4}_{x,k})^c$, and since $\xi^{\mathrm{blue},k}(x)\le 2^{k+1}K_1 N$ by~\eqref{eq:xibluekbd} and $t_{k-1}-t_k\le c_0 2^{1-k}$ by~\eqref{eq:tkdiff},  for $r\geq 1$, we have
\begin{align*}
\#\{(j,u):j\leq \xi^{\mathrm{blue},k}(x),u\in \mathcal N^{\mathrm{blue},x,k,j}_{t_{k-1}-t_k},\|X^{\mathrm{blue},x,k,j}_{t_{k-1}-t_k}(u)\|\geq r\}
\leq \thr(r,k).
\end{align*}
In particular,
for any $y\neq x$, the blue particles at $x$ at time $t_k$ have at most $\thr(\|x-y\|,k)$ descendants at $y$ at time $t_{k-1}$. Hence
by our construction, $\xi^{\mathrm{red},x,k-1}(y)=0$ $\forall y \neq x$.

We have now established that on the event $(R^{\mathrm{cr}}_{x,k})^c$, we have
 $\xi^{\mathrm{red},x,k-1}(y)=0$ $\forall y \in \Z^d$. This proves the lemma.
 \end{proof}
 
 We now finally prove Lemma~\ref{lem:ZRxr}.
 
 \begin{proof}[Proof of Lemma~\ref{lem:ZRxr}]
Take $x\in \Z^d$ and $r\in \N_0$, and suppose the event $(R_{x,r})^c$ occurs. 
 We begin by considering the case $r=0$, \new{and then consider the case $r\ge 1$.}
 
\paragraph{\new{\textbf{Case $r=0$:}}} In this case, \new{by the definition of $R_{x,0}$ in~\eqref{eq:Rx0defn}, we have that} $(R^{\mathrm{cr}}_{x,k})^c$ occurs for each $k\ge 1$. \new{Hence} $\xi^{\mathrm{red},x,k-1}(y)=0$ $\forall y \in \Z^d$, $k \in\llbracket\kmax\rrbracket$, by Lemma~\ref{lem:no_red}.
 Therefore $\eta^{(x)}(y)=0$ $\forall y \in \Z^d$ and the statement of the lemma holds.

\paragraph{\new{\textbf{Case $r\ge 1$:}}} From now on we assume $r\ge 1$. 
\new{Recall the definition of the event $R_{x,r}$ in~\eqref{eq:Rxrdefn}.
The proof in this case consists of the following steps:
\begin{itemize}
\item Step 1: We bound $\xi^{\mathrm{red},x,k-1}(y)$ for each $k\ge 1$ and $y\in \Z^d$, using Lemma~\ref{lem:no_red} and the definition of the events $(R^{\mathrm{num}}_{x,k,r})^c$. 
\item Step 2: We bound $\eta^{(x)}(y')$ for $y'\notin \ball_r(x)$, using the bounds established in Step~1 and the definition of the events $(R_{x,y,k,r})^c$.
\item Step 3: We establish another bound on $\eta^{(x)}(y')$ that holds for all $y'\in \Z^d$, using the bounds established in Step~1 and the definition of the events $(R_{x,y,k,r})^c$.
\end{itemize}
We then combine the bounds from Steps~2 and~3 to conclude.
We now give the details of the proof.
}

\paragraph{\new{\textbf{Step 1:}}}
Recall the definition of $k_r$ in~\eqref{eq:krdefn}.
 For $k>k_r$,
since $(R^{\mathrm{cr}}_{x,k})^c$ occurs we have that
 $\xi^{\mathrm{red},x,k-1}(y)=0$ $\forall y \in \Z^d$, by Lemma~\ref{lem:no_red}.
For $k \in \llbracket k_r\rrbracket$, by the definition of the event $(R^{\mathrm{num}}_{x,k,r})^c$ and since $\xi^{\mathrm{blue},k}(x)\le 2^{k+1}K_1 N$ by~\eqref{eq:xibluekbd} and $t_{k-1}-t_k\le c_0 2^{1-k}$ by~\eqref{eq:tkdiff},
\begin{equation} \label{eq:lem2*}
\xi^{\mathrm{red},x,k-1}(y)\leq g(r)^{-1/5} 2^{k+2}K_1 N \quad \forall y \in \Z^d.
\end{equation}
Moreover, for $y\in \Z^d$ with $\|x-y\|\ge r/2$, 
\new{again by the definition of the event $(R^{\mathrm{num}}_{x,k,r})^c$ and since $\xi^{\mathrm{blue},k}(x)\le 2^{k+1}K_1 N$ and $t_{k-1}-t_k\le c_0 2^{1-k}$,}
the blue particles at $x$ at time $t_k$ have at most $\thr(\|y-x\|,k)$ descendants at $y$ at time $t_{k-1}$, so by our construction,
\begin{equation} \label{eq:xiredoutside}
\xi^{\mathrm{red},x,k-1}(y)=0\quad  \forall y \notin \ball_{r/2}(x).
\end{equation}
\new{To summarise step~1,} we have now established that no blue particles at $x$ at a time $t_k$ with $k>k_r$ have descendants that turn red at time $t_{k-1}$, and no blue particles at $x$ at any time $t_k$ with $k\ge 1$ have descendants that turn red at time $t_{k-1}$ outside $\ball_{r/2}(x)$. \new{Moreover,~\eqref{eq:lem2*} gives us a bound on the number of descendants of blue particles at $x$ at  time $t_k$ that turn red at time $t_{k-1}$ at $y$  for any $y$ and any $k\le k_r$.}

\paragraph{\new{\textbf{Step 2:}}}
Note \new{first} that for $y\in \ball_{r/2}(x)$ and $y'\notin \ball_r(x)$, we have 
$\|y'-y\|\ge \|x-y'\|-\|x-y\|> r/2> \|x-y\|$, which implies
\begin{equation} \label{eq:triangleineq}
\|x-y'\|\le \|x-y\|+\|y-y'\|< 2\|y'-y\|.
\end{equation}
For $y'\notin \ball_r(x)$, by summing over the possible times $t_k$ and locations $y$ at which descendants of blue particles at $x$ may turn red, and then by~\eqref{eq:lem2*} and~\eqref{eq:triangleineq},
and using that $1-t_0=0$ and $y'\notin \mathcal B_{r/2}(x)$ to remove the $k=0$ term from the sum,
\begin{align} \label{eq:xiredbound}
&\eta^{(x)}(y') \notag \\
&\le \sum_{k=0}^{k_r-1} \sum_{y \in \ball_{r/2}(x)} \#\{(j,u): j \le \xi^{\mathrm{red},x,k}(y), u\in \mathcal N_{1-t_k}^{\mathrm{red},x,y,k,j}, X_{1-t_k}^{\mathrm{red},x,y,k,j}(u)+y=y'\} \notag \\
&\le \sum_{k=1}^{k_r-1} \sum_{y \in \ball_{r/2}(x)} \#\{(j,u): j \le g(r)^{-1/5}2^{k+3}K_1 N, u\in \mathcal N_{1-t_k}^{\mathrm{red},x,y,k,j}, \|X_{1-t_k}^{\mathrm{red},x,y,k,j}(u)\|>\tfrac 12 \|x-y'\|\} \notag \\
& \leq \sum_{k=1}^{k_r-1 }\sum_{y\in \ball_{r/2}(x)} k_{\|x-y'\|}^{-1} |\ball_{\|x-y'\|/2}(0)|^{-1} g(\|x-y'\|) K_1^2 N \notag \\
&=g(\|x-y'\|) K_1^2 N,
\end{align}
where the third inequality follows by the definition of the events $(R_{x,y,k,r})^c$ and since $1-t_k \le c_0$ by~\eqref{eq:tkdefn} and $\|x-y'\|\ge r$, and the last inequality follows since $\|x-y'\|\ge r$.

\paragraph{\new{\textbf{Step 3:}}}
For any $y'\in \Z^d$, 
again by summing over possible times $t_k$ and locations $y$ at which descendants of blue particles at $x$ may turn red, and then
by~\eqref{eq:lem2*} and the definition of $(R_{x,y,k,r})^c$ for $y\in \ball_{r/2}(x)$ and $1\le k <k_r$,
\begin{align} \label{eq:xirednum}
\eta^{(x)}(y')
&\le \sum_{k=1}^{k_r-1} \sum_{y \in \ball_{r/2}(x)} \#\{(j,u): j \le \xi^{\mathrm{red},x,k}(y), u\in \mathcal N_{1-t_k}^{\mathrm{red},x,y,k,j}\} +\xi^{\mathrm{red},x,0}(y')\notag \\
& \leq \sum_{k=1}^{k_r-1 }\sum_{y\in \ball_{r/2}(x)} \tfrac 12 k_r^{-1}|\ball_{r/2}(0)|^{-1} g(r)^{-1} K_1^2 N
+g(r)^{-1/5}8K_1 N \notag \\
&\le K_1^2 g(r)^{-1} N,
\end{align}
where the last line follows since we chose $K_1>16$ (before~\eqref{eq:K0defn}) and so $\frac 12 K_1^2 +8K_1\le K_1^2$.

\new{Therefore, for any $y'\in \Z^d$, using~\eqref{eq:xirednum} in the case $\|y'-x\|<r$ and~\eqref{eq:xiredbound} in the case $\|y'-x\|\ge r$ we have
\[
\eta^{(x)}(y')\le g(r)^{-1} K_1^2 N \1_{\|y'-x\|<r} +g(\|y'-x\|) K_1^2 N \1_{\|y'-x\|\ge r},
\]
which completes the proof.
}
\end{proof}

\subsection{Proof of Proposition~\ref{prop:closedinitcond}} \label{sec:propclosedinit}

Take $T>1$. We now construct the BRWNLC particle system \new{(with some jump rate $\gamma$, jump kernel $p$ and competition kernel $\Lambda$)} on the time interval $[0,T]$, using a construction that is different from the one used in Section~\ref{sec:propupper}, but still based on the BRW trees with resiliences construction in Section~\ref{sec:construction}.
Take a (random) particle configuration $\xi$ consisting of a finite number of particles with $\xi=\zeta+\sum_{x\in \Z^d}\zeta^{(x)}$, where $\zeta,\zeta^{(x)}\in \N_0^{\Z^d}$ are \new{coupled with $(Z_\ep^{(x)})_{x\in \Z^d}$} as in Proposition~\ref{prop:closedinitcond} (for some $\varepsilon>0$ to be specified later).
We construct $(\xi_t)_{t\in [0,T]}$ with $\xi_0=\xi$.

The particles will be split into blue, red and yellow particles.
The reader is referred to Section~\ref{sec:proof_outline} for a heuristic description of the roles of the blue, red and yellow particles in the construction.

For $x,y\in \Z^d$ and $j\in \N$, let $\mathcal T^{\mathrm{blue},x,j}$, $\mathcal T^{\mathrm{red},x,y,j}$ and $\mathcal T^{\mathrm{yellow},x,j}$ be i.i.d.~copies of the BRW tree with resiliences $\mathcal T$ (defined in~\eqref{eq:calTdefn}).
The corresponding entries will be denoted by $\mathcal N^*$, $X^*$, $\alpha^*$, $\beta^*$, $\rho^*$, where $*$ is replaced by the corresponding superscripts.
The process $(\xi_t)_{t\in [0,T]}$ will be constructed as a deterministic function of $\zeta$, $(\zeta^{(x)})_{x\in \Z^d}$ and the BRW trees with resiliences 
$(\mathcal T^{\mathrm{blue},x,j})_{x\in \Z^d,j\in \N}$, $(\mathcal T^{\mathrm{red},x,y,j})_{x,y\in \Z^d,j\in \N}$ and $(\mathcal T^{\mathrm{yellow},x,j})_{x\in \Z^d,j\in \N}$, which will encode the behaviour of the blue, red and yellow particles respectively.

Particles at time $0$ are initially coloured either blue or red, and given a label $x\in \Z^d$, according to the following rule:
For each $y\in \Z^d$, each particle at $y$ that contributes to $\zeta$ is coloured blue and given label $y$, and each particle at $y$ that contributes to $\zeta^{(x)}$ for some $x\in \Z^d$ is coloured red, and given label $x$.
We then recolour and relabel some particles according to the following rule:
If $\xi(y)\ge J$, and if $\zeta(y)<J$, then choose $J-\zeta(y)$ red particles at $y$ according to some arbitrary rule, and recolour them blue with label $y$, so that there are exactly $J$ blue particles at $y$.
\new{(Recall that we set $J=\lfloor N^{1/3} \rfloor$.)}

As mentioned above, during the time interval $[0,T]$, particles will be coloured blue, red or yellow. Whenever a particle branches, both offspring particles inherit the colour and label of the parent particle.
For $t\in [0,T]$ and $x,y\in \Z^d$, we let $\xi^{\mathrm{blue},x}_t(y)$ (resp.~$\xi^{\mathrm{red},x}_t(y)$, $\xi^{\mathrm{yellow},x}_t(y)$) denote the number of blue (resp.~red, yellow) particles with label $x$ at location $y$ at time $t$.
Similarly, for $t\in [0,T]$ and $y\in \Z^d$, we let $\xi^{\mathrm{blue}}_t(y)$ (resp.~$\xi^{\mathrm{red}}_t(y)$, $\xi^{\mathrm{yellow}}_t(y)$) denote the total number of blue (resp.~red, yellow) particles at location $y$ at time $t$.

For every site $x\in\Z^d$, there may be some time $\tau(x)\in [0,T)$ at which blue or red particles at that site are coloured yellow. 
These are the only times in the time interval $(0,T)$ at which particles change colour.
The particles living between times $0$ and $T$ will be subsets of the particles in the following BRW trees with resiliences:
\begin{itemize}
\item The descendants of the blue particles at $x$ at time $0$ will come from the BRW trees with resiliences $\mathcal T^{\mathrm{blue},x,j}$ for $1\leq j \leq \xi^{\mathrm{blue}}_0(x)$, with the root of $\mathcal T^{\mathrm{blue},x,j}$ positioned at $x$ at time $0$. 
\item The descendants of red particles with label $x$ at $y$  at time $0$ will come from the BRW trees with resiliences $\mathcal T^{\mathrm{red},x,y,j}$ for $1\leq j \leq \xi^{\mathrm{red},x}_0(y)$, with the root of $\mathcal T^{\mathrm{red},x,y,j}$ positioned at $y$ at time $0$.
\item The descendants of particles which turn yellow at $x$ at time $\tau(x)$ will come from the BRW trees with resiliences $\mathcal T^{\mathrm{yellow},x,j}$, $j\ge 1$, with the root of $\mathcal T^{\mathrm{yellow},x,j}$ positioned at $x$ at time $\tau(x)$.
\end{itemize}
The process $(\xi_t)_{t\in[0,T]}$ is then constructed from these BRW trees with resiliences as in Section~\ref{sec:construction}, 
removing the particles that turn yellow at times $\tau(x)$ (and their descendants) from their original trees  $\mathcal T^{\mathrm{blue},x',j}$ or $\mathcal T^{\mathrm{red},x',y,j}$ from time $\tau(x)$ onwards.

Having explained the idea of the construction, we can now specify the times $\tau(x)$, $x\in\Z^d$. 
We let
\begin{equation} \label{eq:tauxdefn}
\tau(x) =\inf\{t\in [0,T]: \xi^{\mathrm{red}}_t(x)+\xi^{\mathrm{blue}}_t(x)\geq \lfloor N^{1/2} \rfloor \},
\end{equation}
with the convention that $\inf \emptyset = \infty$. At time $\tau(x)$, if $\tau(x) < T$, choose $\lfloor N^{1/2} \rfloor$ blue or red particles at $x$ according to some arbitrary rule and turn them yellow with label $x$.
This completes the construction of $(\xi_t)_{t\in [0,T]}$.

\paragraph{Bad events.}

We now define some `bad events' for the BRW trees with resiliences $\mathcal T^{\mathrm{blue},x,j}$, $\mathcal T^{\mathrm{red},x,y,j}$ and $\mathcal T^{\mathrm{yellow},x,j}$ which control the blue, red and yellow particles respectively.
We will give informal descriptions of what happens on the complements of these bad events; these descriptions will be made precise in the proofs of Lemmas~\ref{lem:Ce_inclusion} and~\ref{lem:Ktbds} below.
Let $L>0$ and $K_2>0$ be large constants and \new{let }$c_2>0$ \new{and $\epsilon_1$} be small constants to be chosen later.
\new{Fix $\lambda >0$ and recall that $K_0=K_0(d,\gamma,R_1,\lambda)>1$ is defined in Proposition~\ref{prop:upperbound}; the events below will depend on $T,L,K_2,c_2,\epsilon_1,K_0,J$ and $N$ (we will later assume that the competition kernel $\Lambda$ satisfies~\eqref{eq:Lambda_assumptions} with $N_0$ sufficiently large and with some $R$ and a suitable $\kappa$, and that the jump kernel $p$ satisfies~\eqref{eq:p_assumptions} with $p(x)=0$ $\forall x\notin \mathcal B_{R_1}(0)$).}

The first two events make sure that particles issued from a point $x$ will arrive close to a point $y$, as long as the killing rate is not `too large' along their trajectory. 

For $x,y\in\Z^d$, set
\begin{align} \label{eq:Bxydefn}
B_{x,y} &= \Big\{ \#\Big\{ (j,u): j\in\llbracket J\rrbracket,\, u\in \mathcal N_T^{\mathrm{blue},x,j},\,
 \rho^{\mathrm{blue},x,j}(v) \ge c_2 T\,\, \forall v \prec u, \notag \\
&\hspace{4cm}X^{\mathrm{blue},x,j}_t(u) \in \ball _L((y-x)t/T)\,\forall t\in[0,T]\Big\} < J |\ball _L(0)| \Big\} \notag \\
&\quad \new{\cup \Big\{ \#\Big\{ j\in\llbracket J\rrbracket : \beta^{\mathrm{blue},x,j}(\emptyset) > T,
 \rho^{\mathrm{blue},x,j}(\emptyset) \ge K_2 T, X^{\mathrm{blue},x,j}_t(\emptyset) =0\,\forall t\in[0,T]\Big\} =0\Big\}.}
\end{align}
On the event $(B_{x,y})^c$, if there are at least $J$ blue particles at $x$ at time $0$, and if the killing rate is at most $c_2$ in a tube from $(x,0)$ to $(y,T)$ with radius $L$, and no particles are turned yellow in this tube, then there will be at least $J$ descendants of the blue particles at some site $y' \in \ball_L(y)$ at time $T$.
\new{Moreover, if there are at least $J$ blue particles at $x$ at time $0$, and if the killing rate at $x$ is at most $K_2$ on the time interval $[0,t]$ for some $t\le T$, and no particles are turned yellow at $x$ before time $t$, then there will be at least one particle at $x$ at time $t$.}

Now for $x,y\in \Z^d$, set
\begin{align} \label{eq:Yxydefn}
Y_{x,y} = &\Big\{ \#\Big\{ j \in\llbracket N^{1/2} \rrbracket:
\beta^{\mathrm{yellow},x,j}(\emptyset) > T,\, \rho^{\mathrm{yellow},x,j}(\emptyset) \ge K_2 T, \notag \\
&\hspace{2cm}X^{\mathrm{yellow},x,j}_t(\emptyset) \in \ball_L((y-x)(\tfrac t{\epsilon_1 T}\wedge 1))\,\forall t\in[0,T]
\Big\} < J|\ball_{L}(0)| \Big\} \notag \\
&\cup \new{\{\#\{j\in \llbracket N^{1/2}\rrbracket :\beta^{\mathrm{yellow},x,j}(\emptyset)>T, \rho^{\mathrm{yellow},x,j}(\emptyset)\ge K_2 T, X_t^{\mathrm{yellow},x,j}(\emptyset)=0\,\forall t\in [0,T]\}=0\}.}
\end{align}
On the event $(Y_{x,y})^c$, roughly speaking, if the time $\tau(x)$ is triggered, a sufficiently large fraction of yellow particles created at $x$ will quickly move close to $y$ (within a time $\epsilon_1 T$) and then stay in the vicinity of $y$. More precisely, if $\tau(x)\le (1-\epsilon_1)T$ and the killing rate is at most $K_2$ in both a tube from $(x,\tau(x))$ to $(y,\tau(x)+\epsilon_1 T)$ of radius $L$ and a tube from $(y,\tau(x)+\epsilon_1 T)$ to $(y,T)$ of radius $L$, then there will be at least $J$ descendants of the yellow particles from $x$ at some site $y' \in \ball_L(y)$ at time $T$.
\new{Moreover, if the time $\tau(x)$ is triggered, and if the killing rate at $x$ is at most $K_2$ on the time interval $[\tau(x),t]$ for some $t\le T$, then there will be at least one particle at $x$ at time $t$.}

In order to make sure the killing rate is not `too large', we have to control the growth and spread of the particles. 
For $x\in \Z^d$ and $r\ge 0$, set
\begin{align*}
B^{\mathrm{spread}}_{x,r}
&=\Big\{\sup_{t\in [0,T]}\#\{(j,u):j\in \llbracket K_0 N \rrbracket,u\in \mathcal N^{\mathrm{blue},x,j}_t, \|X^{\mathrm{blue},x,j}_t (u)\|\geq r/2 \}\geq e^{3T} g(r)K_0 N
\Big\}\\
&\hspace{1cm}\cup \{ \#\{(j,u):j\in \llbracket K_0 N \rrbracket,u\in \mathcal N^{\mathrm{blue},x,j}_T\}\ge e^{3T} g(r)^{-2} K_0 N
\}.
\end{align*}
Since $J\le N <K_0 N$ and $\zeta(x)\le K_0 N$,
by our construction we have
$\xi^{\mathrm{blue}}_0(x)\le K_0 N$, and so, on the event $(B^{\mathrm{spread}}_{x,r})^c$, the blue particles at $x$ at time $0$ have fewer than $e^{3T} g(r) K_0 N$ descendants outside $\ball_{r/2}(x)$, and fewer than $e^{3T} g(r)^{-2}K_0 N$ descendants in total, at all times in $[0,T]$.

For $x\in \Z^d$ and $r\ge 0$, set 
\begin{align*}
Y^{\mathrm{spread}}_{x,r}
&=\Big\{\sup_{t\in [0,T]}\#\{(j,u):j\in \llbracket N^{1/2} \rrbracket,u\in \mathcal N^{\mathrm{yellow},x,j}_t, \|X^{\mathrm{yellow},x,j}_t (u)\|\geq r/2 \}\geq e^{3T} g(r) N^{1/2}
\Big\}\\
&\hspace{1cm} \cup \{ \#\{(j,u):j\in \llbracket N^{1/2} \rrbracket,u\in \mathcal N^{\mathrm{yellow},x,j}_T\}\ge e^{3T} g(r)^{-2} N^{1/2}
\}.
\end{align*}
On the event $(Y^{\mathrm{spread}}_{x,r})^c$, any yellow particles created at $x$ at time $\tau(x)$ have fewer than $e^{3T} g(r) N^{1/2}$ descendants outside $\ball_{r/2}(x)$, and fewer than $e^{3T} g(r)^{-2}N^{1/2}$ descendants in total, at all times in $[\tau(x),T]$.

Finally, the next event controls the spread of the red particles with label $x$ in a ball around $x$.
For $x\in \Z^d$ and $r\ge 0$, set
\begin{align*}
R_{x,r}^{\mathrm{spread}}
&= \Big\{\sup_{t\in[0,T]}\#\{(y,j,u): y\in \ball_{(r/4)\vee 1}(x),j\in \llbracket g(r/4)^{-1} K_0 N\rrbracket,\,u\in \mathcal N^{\mathrm{red},x,y,j}_t,\\
&\hspace{7.5cm} \|X^{\mathrm{red},x,y,j}_t(u)\|\ge r/4\} \ge e^{3T} g(r) K_0 N\Big\}\\
&\quad \cup \{\#\{(y,j,u): y\in \ball_{(r/4)\vee 1}(x),j\in \llbracket g(r/4)^{-1} K_0 N\rrbracket,\,u\in \mathcal N^{\mathrm{red},x,y,j}_{T}\}\ge e^{3T} g(r)^{-2} K_0 N\}.
\end{align*}

\new{Recall that the initial condition $\xi$ is coupled with $(Z_\ep^{(x)})_{x\in \Z^d}$ for some $\varepsilon>0$.}
For $x\in \Z^d$, $r\ge 0$ and $\ep>0$,
define a bad event involving particles of all colours by letting
\begin{align} \label{eq:Pxrep_defn}
P_{x,r,\ep} = \{Z^{(x)}_\ep> r/4\} \cup R^{\mathrm{spread}}_{x,r}  \cup B^{\mathrm{spread}}_{x,r} \cup Y^{\mathrm{spread}}_{x,r}.
\end{align}
We will see in the proof of Lemma~\ref{lem:Ktbds} below that on the complement of this event, for any $y\notin \ball_{r/2}(x)$, blue and red particles with label $x$ contribute at most $2e^{3T} g(r) K_0N$ particles at $y$ at any time in $[0,T]$.
Moreover, yellow particles with label $x$ contribute at most $e^{3T} g(r)N^{1/2}$ particles at $y$ at any time in $[0,T]$.

\new{From now on, fix $\gamma>0$ and $p$ satisfying~\eqref{eq:p_assumptions} with $p(x)=0$ $\forall x\notin \mathcal B_{R_1}(0)$.}
Let $(X_t)_{t\ge 0}$ denote a continuous-time random walk started at 0 with jump rate $\gamma$ and jump kernel $p$.
We need the following result, which will be deduced from a large deviation principle for the process $(X_t)_{t\ge0}$ in Section~\ref{sec:badbluepf}.
This result will allow us to bound the probabilities of the bad events $B_{x,y}$ and $Y_{x,y}$ for suitable $x$ and $y$.
\begin{lemma} \label{lem:blue1}
Let $v_0\in\R^d\backslash \{0\}$ and let $a\in [0,a_{v_0})$. There exists $\delta_0=\delta_0(\new{\gamma,p}, v_0,a) \in  \new{(0,1/3)}$ such that for $\epsilon>0$, there exists $t_0=t_0(\epsilon,\new{\gamma,p}, v_0,a)>1$ such that the following holds. For $t\ge t_0$ and $x,y \in \R^d$ with $\|y-x-(\mu +av_0)t\|\le \delta_0 t$,
$$
\p{X_s \in \ball_{\epsilon t}((y-x)s/t ) \, \, \forall s\in [0,t]}\ge e^{-(1-3\delta_0)t}.
$$
\end{lemma}
From now on, fix $v_0\in \R^d\backslash \{0\}$, $0\le \aminus <\aplus <a_{v_0}$ and $R>0$.
We now choose the constants $\kappa$, $c_2$, $R'$, $L$, $T$, $\epsilon_1$  and $K_2$.
First let $\delta_0=\min(\delta_0(\new{\gamma,p}, v_0,\aminus ),\delta_0(\new{\gamma,p}, v_0,\aplus ))>0$ as in Lemma~\ref{lem:blue1}, and \new{recalling the definition of $\mu$ in~\eqref{eq:mudefn},} let
\begin{align} 
\epsilon_2 &= \min(\tfrac 13 \delta_0 , \tfrac 19 \max(\|\mu+\aminus v_0\|,\|\mu+\aplus v_0\|))>0, \notag \\
\epsilon_1 &=\tfrac 14 \epsilon_2 (\max(\|\mu+\aminus v_0\|,\|\mu+\aplus v_0\|))^{-1}>0,\label{eq:epsilon1defn}\\
\text{ and }\qquad 
\epsilon_0 &= \min(\tfrac 17 \epsilon_1 \delta_0, \tfrac 1 {10} \epsilon_2)>0. \notag 
\end{align}
Note that, in particular, $\epsilon_1\le 1/36$.
Let $\kappa =4\epsilon_0^{-1}$ and let $t_0=\max(t_0(\epsilon_0,\new{\gamma,p,}v_0,\aminus ),t_0(\epsilon_0,\new{\gamma,p,}v_0,\aplus ))$ as in Lemma~\ref{lem:blue1}.
 We now choose $T>1$; we will take $T$ sufficiently large that several conditions hold. 
For $t\ge 0$, let $D_t= \max(\|\lfloor (\mu+\aminus v_0)t\rfloor \|,\|\lfloor (\mu+\aplus v_0)t\rfloor \|)$.
 First, take $T$ sufficiently large that
\begin{equation} \label{eq:Tdefn}
\begin{aligned}
&T\ge t_0, \;
 \epsilon_0 T \ge R,\;
 \epsilon_1 D_T<\tfrac 13 \epsilon_2 T,\;
 \epsilon_2 T\le \tfrac 12 \delta_0 T-d^{1/2}, \; 4\epsilon_2 T +1 <\tfrac 12 D_T, \\
& D_T>60 R_1
  \quad  \text{ and }\quad 
 T\ge \delta_0^{-1} \log (2|\ball _{\epsilon_0 T}(0)|)+1.
 \end{aligned}
\end{equation}
Also take $T$ sufficiently large that $\lfloor (\mu+a_-v_0)T\rfloor \neq \lfloor (\mu+a_+ v_0)T\rfloor$.
Moreover, (by considering the cases $r>8R_1\gamma Te$ and $r\le 8R_1\gamma Te$ separately) suppose $T$ is sufficiently large that
\begin{equation} \label{eq:C1defn}
|\ball_{r/4}(0)|g(r/4)^{-1} \min\left(1, e^{-\frac r {4R_1} \log \left( \frac r {4R_1 \gamma  T e}\right)}\right)\le \tfrac 12 e^{2T} g(r) \quad \forall r\ge 1 \quad \text{ and } \quad 2 \le  e^{2T}.
\end{equation}
Also take $T$ sufficiently large that
$e^{\frac 12 e^{-T}}\le 1+e^{-T}\le 2$, and so $(1-e^{-T})e^{\frac 12 e^{-T}}\le 1-e^{-2T}<1$ and
$\log(e^{\frac 12 e^{-T}}/(1-(1-e^{-T})e^{\frac 12 e^{-T}}))\le 2T +\log 2$. Then
take $T$ sufficiently large that
\begin{equation} \label{eq:C1defn2}
\tfrac 14 e^{2T}g(r)^{-2}\ge |\ball_{(r/4)\vee 1}(0)|g(r/4)^{-1}(2T+\log 2) \quad \forall r\ge 0,
\end{equation}
which is possible since $g(r)=r^{-6d-2}$ for $r\ge 1$, by~\eqref{eq:gdefn}.
Finally, letting
\begin{equation} \label{eq:K2defn}
K_3=\sum_{z\in \Z^d}(2K_0+1)g(\lfloor \|z\|\rfloor )<\infty,
\end{equation}
and using that $\epsilon_0 \kappa=4$ by our choice of $\kappa$ after~\eqref{eq:epsilon1defn}, take $T$ sufficiently large that
\begin{equation} \label{eq:kappadefn}
e^{3T}\sum_{z\in \Z^d,\|z\|\ge \epsilon_0 T} e^{-\kappa \|z\|}
\Big( K_3 + 
|\ball_{\|z\|}(0)|(2K_0+1)g(2\|z\|)^{-2}\Big)
<\tfrac 13 \delta_0 T^{-1}.
\end{equation}
We can now define the remaining constants. Let
\begin{equation} \label{eq:R'Ldefn}
c_2 = T^{-1}\delta_0 , \quad 
R' =\epsilon_2 T \quad \text{and} \quad L=\epsilon_0T .
\end{equation}
Take $K_2$ sufficiently large that 
\begin{equation} \label{eq:K1defn}
K_2>2e^{3T} \sum_{z\in \Z^d}(2K_0+1)g(\lfloor \|z\|\rfloor)
\end{equation}
and
\begin{equation} \label{eq:K1defn2}
K_2>2e^{3T} \sum_{z\in \Z^d}(\1_{\|z\|<R} +e^{-\kappa \|z\|})|\ball_{\|z\|}(0)|(2K_0+1)g(2 \|z\|)^{-2}.
\end{equation}

Now we can state the two main intermediate lemmas in the proof of Proposition~\ref{prop:closedinitcond}.
The first lemma bounds the probabilities of the bad events; we will prove this result in Section~\ref{sec:badbluepf}.
Let $D:=D_T=\max(\|\lfloor (\mu+\aminus v_0)T\rfloor \|,\|\lfloor (\mu+\aplus v_0)T\rfloor \|)>0$.
\begin{lemma}
\label{lem:bad_blue}
For every $\ep>0$ sufficiently small, there exists $N_0>0$ such that for $N\geq N_0$, the following holds for $a\in \{\aminus ,\aplus \}$:
\begin{enumerate}
\item $\P(B_{x,y}) < \ep$ $\forall x,y \in \Z^d$ with $\|y-x-(\mu +av_0)T\|<\delta_0 T$.
\item $\P(Y_{x,y}) < \ep$ $\forall x,y \in \Z^d$ with $\|y-x-r(\mu +av_0)T\|<\delta_0 r T$ for some $r\in [\epsilon_1,1]$ or $x=y$.
\item $\P(P_{x,r,\ep}) \le \min(2\ep, f(r+\ep ^{-1})^{1/5}/|\ball_{r+3D/2}(0)|)$ $\forall r\in \N_0, x\in \Z^d$.
\end{enumerate}
\end{lemma}
The second lemma says that if an edge $e$ in the renormalization grid is closed, i.e.~if the event $\mathcal C_e$ defined in \eqref{eq:Cdefn} occurs, then one of a collection of bad events must occur; we will prove this result at the end of this section.
\begin{lemma} \label{lem:Ce_inclusion}
For any $\ep>0$, there exists $N_0>0$ such that \new{if $\Lambda$ satisfies~\eqref{eq:Lambda_assumptions}}, the following holds.
If the coupling in~\eqref{eq:coupling_upper} holds for a (random) particle configuration $\xi\in \N_0^{\Z^d}$ consisting of a finite number of particles, 
then under the construction of $(\xi_t)_{t\in [0,T]}$ with $\xi_0=\xi$ at the start of Section~\ref{sec:propclosedinit} \new{with jump rate $\gamma$, jump kernel $p$ and competition kernel $\Lambda$},
for $e = ((x_0,0),(y_0,T))\in E_{T,\aminus ,\aplus ,v_0}(0)$,
\begin{align}
\label{eq:Ce_inclusion}
\mathcal C_e \subseteq \bigcup_{x \in \ball_{R'}(x_0)} B_{x,y_0} \cup \bigcup_{t\in [0,(1-\epsilon_1)T],x \in H_t} Y_{x,y_0} \cup \bigcup_{t\in [0,T],x \in H_t} Y_{x,x} 
\cup \bigcup_{t\in [0,T], x\in H_t} \bigcup_{y\in\Z^d} P_{y,\lfloor \|y-x\|\rfloor,\ep },
\end{align}
where $H_t = \ball_{\frac{T-t}T R'+3L}(x_0+\frac{t}T(y_0-x_0))$ for $t\in [0,T]$.
\end{lemma}
Before proving Lemma~\ref{lem:Ce_inclusion}, we now show how we can deduce Proposition~\ref{prop:closedinitcond} from Lemma~\ref{lem:bad_blue} and Lemma~\ref{lem:Ce_inclusion}.
\begin{proof}[Proof of Proposition~\ref{prop:closedinitcond}]
\new{We begin by choosing $\varepsilon$ and $N_2$.}
Recall that  $D=\max(\|\lfloor (\mu+\aminus v_0)T\rfloor \|,$ $\| \lfloor  (\mu+\aplus v_0) T \rfloor \|)$,
and let $D':=\|\lfloor (\mu+\aminus v_0)T\rfloor - \lfloor  (\mu+\aplus v_0) T \rfloor \|>0$.
Recall from~\eqref{eq:Tdefn} that $D >60 R_1$.
Take $b\in (0,1)$, and then
take $\ep >0$ sufficiently small that Lemma~\ref{lem:bad_blue} holds and that $\ep^{-1} \ge 3D/2$, $(\ep^{-1}+1)^{-D/(60R_1)}<b/2$, $\ep < b/(10|\ball_{R' +3L+D}(0)|)$
and 
\begin{equation} \label{eq:sumb2}
\sum_{r=1}^\infty (r+\ep^{-1}+1)^{-\frac D {60 R_1}} \le b/2. 
\end{equation}
Then take \new{$N_2$} sufficiently large that Lemmas~\ref{lem:bad_blue} and~\ref{lem:Ce_inclusion} hold \new{with $N_0=N_2$} for this choice of $\ep$.
Suppose \new{$\Lambda$ satisfies~\eqref{eq:Lambda_assumptions} with $N_0=N_2$} and suppose the coupling in~\eqref{eq:coupling_upper} holds for a random particle configuration $\xi$ consisting of a finite number of particles; construct $(\xi_t)_{t\in [0,T]}$ with $\xi_0=\xi$ as at the start of Section~\ref{sec:propclosedinit}.

 Let $e_1,e_2,\ldots ,e_n\in E_{T,\aminus ,\aplus ,v_0}(0)$. Write  $e_i = ((x_i,0),(y_i,T))$ for each $i$. 
 We can suppose without loss of generality that 
 \begin{equation} \label{eq:xixjnottooclose}
 \|x_i - x_j\| \ge 3D\quad \forall i\ne j;
 \end{equation}
 otherwise we choose a subset of edges of size at least $ n/(6\lceil D/D'\rceil )$ such that this is the case and replace $b$ by $b^{1/(6\lceil D/D'\rceil)}$. 
 
\new{We will now define a `local' bad event and a `non-local' bad event for each edge $e_i$;
these events will be defined in such a way that (using Lemma~\ref{lem:Ce_inclusion}) if $\mathcal C_{e_i}$ occurs then either the local or the non-local bad event for the edge $e_i$ must occur, and also the local events for different edges $e_i$ will be independent.}
 For $i=1,\ldots,n$ the `local' bad event for the edge $e_i$ \new{is defined as follows}:
\begin{equation} \label{eq:Lidefn}
L_i = \bigcup_{x \in \ball_{R'}(x_i)} B_{x,y_i} \cup \bigcup_{t\in [0,(1-\epsilon_1)T],x \in H^i_t} Y_{x,y_i} \cup \bigcup_{t\in [0,T],x \in H^i_t} Y_{x,x} 
\cup \bigcup_{t\in [0,T], x\in H^i_t}  P_{x,0,\ep },
\end{equation}
where $H^i_t = \ball_{\frac{T-t}T R'+3L}(x_i+\frac{t}T(y_i-x_i))$ for $t\in [0,T]$.

\new{We now use Lemma~\ref{lem:bad_blue} to bound the probability of the event $L_i$.} Note \new{first} that for $x\in \ball_{R'}(x_i)$, by~\eqref{eq:Tdefn} and~\eqref{eq:R'Ldefn} we have 
$$
\|y_i-x-(y_i-x_i)\|<R'<\delta_0 T-d^{1/2},
$$
and so $\p{B_{x,y_i}}<\ep$ by Lemma~\ref{lem:bad_blue}.
For $t\in [0,(1-\epsilon_1)T]$ and $x\in H^i_t$, by~\eqref{eq:Tdefn} and~\eqref{eq:R'Ldefn} again, and since by~\eqref{eq:epsilon1defn}, $3L=3\epsilon_0 T\le \frac 3 7 \epsilon_1 \delta_0 T<\frac 12 \delta_0 T\cdot \frac{T-t}T$, we have
$$
\|y_i-x-\tfrac{T-t}T (y_i-x_i)\|<\tfrac{T-t}T R'+3L <\tfrac{T-t}T (\delta_0 T -d^{1/2}),
$$
and so $\p{Y_{x,y_i}}<\ep$ by Lemma~\ref{lem:bad_blue}.
For $t\in [0,T]$ and $x\in H^i_t$, we have $\|x-x_i\|<R'+3L+D$.
Hence, since we chose $\ep < b/(10|\ball_{R' +3L+D}(0)|)$, by Lemma~\ref{lem:bad_blue} and a union bound we have $$\P(L_i) \le b/2 \quad \text{for }i=1,\ldots, n.$$
 Furthermore, since $\epsilon_0<\epsilon_2$ by~\eqref{eq:epsilon1defn}, and so
  by~\eqref{eq:R'Ldefn} and~\eqref{eq:Tdefn}
 \begin{equation} \label{eq:R'D}
     R'+3L<4\epsilon_2 T<D/2 -1,
 \end{equation}
we have $\ball_{R'  +3L+D}(x_i) \cap  \ball_{R'  +3L+D}(x_j) =\emptyset$ for $i\neq j$,
and therefore  the events $L_i$ for $i=1,\ldots , n$ are independent by definition. 
  
Now define a `non-local' bad event for the edge $e_i$ as
\begin{equation} \label{eq:NLdefn}
N\! L_i =  \bigcup_{y\in\Z^d}  \bigcup_{r \in \N,\ |r - \|y-x_i\|| < 3D /2} P_{y,r,\ep}.
\end{equation}
Since $R'+3L+D+1< 3D/2$ by~\eqref{eq:R'D}, and since for $x,y \in \Z^d$,
$|\lfloor \|y-x\|\rfloor - \|y-x_i\| |\le 1 + \|x_i-x\|$,
we have
\begin{equation} \label{eq:inNL}
\bigcup_{y\in \Z^d}\bigcup_{x \in \ball_{R'+3L+D}(x_i),\; x \neq y}P_{y,\lfloor \|y-x\|\rfloor,\ep } \subseteq N\! L_{i}.
\end{equation}
We summarise: By Lemma~\ref{lem:Ce_inclusion},~\eqref{eq:Lidefn} and~\eqref{eq:inNL},
and since $H^i_t \subseteq \ball_{R'+3L+D}(x_i)$ $\forall t\in [0,T]$,
\[
\mathcal C_{e_i} \subseteq L_i \cup N\! L_i \quad \text{for }i=1,\ldots,n,
\]
and the events $L_1,\ldots,L_n$ are independent with $\p{L_i}\le b/2$ for each $i$. 
Hence,
\begin{align}
\label{eq:first}
\bigcap_{i=1}^n \mathcal C_{e_i} \subseteq \bigcup_{S\subseteq \llbracket n \rrbracket} \left(\bigcap_{i\not\in S} L_i \cap \bigcap_{i\in S} N\! L_i \right).
\end{align}
Recalling~\eqref{eq:NLdefn}, we rewrite the last intersection as
\begin{align*}
\bigcap_{i\in S} N\! L_i = \bigcap_{i\in S} \bigcup_{r\in\N} \bigcup_{\{y \in\Z^d: |r-\|y-x_i\||<3D /2\}} P_{y,r,\ep} \subseteq \bigcup_{(r_i)_{i\in S}\in\N^{S}} \bigcap_{i\in S} \bigcup_{y\in \ball_{r_i + 3D /2}(x_i)} P_{y,r_i,\ep}.
\end{align*}
For $i=1,\ldots , n$, let
\begin{equation} \label{eq:NLirdefn} N\! L_{i,r} = \bigcup_{y\in \ball_{r + 3D /2}(x_i)} P_{y,r,\ep} \quad \text{for }r\in \N, \quad \text{ and } 
\quad N\! L_{i,0} =L_i,
\end{equation}
so that 
$\cap_{i\in S} N\! L_i \subseteq \cup_{(r_i)_{i\in S}\in \N^{S}}\cap_{i\in S} N\! L_{i,r_i}$.
Then by~\eqref{eq:first}, letting $r_i=0$ for $i\notin S$,
\begin{align}
\label{eq:second}
\bigcap_{i=1}^n \mathcal C_{e_i}\subseteq
\bigcup_{(r_i)_{i=1}^n\in \N^n_0}\bigcap_{i=1}^n N\! L_{i,r_i}.
\end{align}

We now claim that for each $(r_i)_{i=1}^n \in\N^n_0$,
\begin{equation}
\label{eq:claim}
\p{ \bigcap_{i=1}^n N\! L_{i,r_i}} \le \prod_{i=1}^n \left(f(r_i +\ep ^{-1})^{\frac 3 {20}D (r_i+\ep ^{-1})^{-1}}\1_{r_i>0}+\frac b 2 \1_{r_i=0}\right).
\end{equation}
\new{Fix $(r_i)_{i=1}^n \in\N^n_0$. In order to prove~\eqref{eq:claim}, we will prove that for any subset $A\subseteq \llbracket n \rrbracket$, we have
\begin{equation}
\label{eq:claimA}
\p{ \bigcap_{i\in A} N\! L_{i,r_i}} \le \prod_{i\in A} \left(f(r_i +\ep ^{-1})^{\frac 3 {20}D (r_i+\ep ^{-1})^{-1}}\1_{r_i>0}+\frac b 2 \1_{r_i=0}\right).
\end{equation}
We will prove this by induction on $|A|$ as follows.
First note that~\eqref{eq:claimA} trivially holds for $A=\emptyset$.
Now take $k\in \{0,1,\ldots, n-1\}$ and suppose that~\eqref{eq:claimA} holds for all $A\subseteq \llbracket n \rrbracket$ with $|A|\le k$.
Take $A^*\subseteq \llbracket n \rrbracket$ with $|A^*|= k+1$; we now show that~\eqref{eq:claimA} holds with $A=A^*$ by considering two cases.}

\noindent \new{\textbf{Case 1:} $r_i=0$ $\forall i\in A^*$. }
\new{In this case we have $N\! L_{i,r_i}=L_i$ for each $i\in A^*$} \new{by~\eqref{eq:NLirdefn}}, and since the events $L_i$ are independent with $\p{L_i}\le b/2$ \new{for each $i$}, we \new{have
\[
\p{ \bigcap_{i\in A^*} N\! L_{i,r_i}} =\prod_{i\in A^*} \p{L_i} \le \left( \frac b 2 \right)^{|A^*|},
\]
which gives us~\eqref{eq:claimA} with $A=A^*$.}

\noindent \new{\textbf{Case 2:} $r_i>0$ for some $i\in A^*$. 
In this case, take $i^*\in A^*$ such that $r_{i^*}=\max_{i\in A^*}r_i >0$ (making an arbitrary choice if there is more than one possible choice of $i^*$).
Let
\begin{equation} \label{eq:S*defn}
S^*:= \{i\in A^* : x_i \in \ball_{2r_{i^*} + 3D }(x_{i^*})\}.
\end{equation}
}
Note that \new{$i^* \in S^*$}, and for every $i\in \new{A^* \setminus S^*}$, we have $r_i \le r_{i^*}$ and so $\ball_{r_i+3D /2}(x_i) \cap \ball_{r_{i^*}+3D /2}(x_{i^*}) = \emptyset$. 

\new{We now claim that
the family of events $(N\! L_{i,r_i})_{i\in A^* \setminus S^*}$ is independent of $N\! L_{i^*,r_{i^*}}$.
Indeed, for $i\in \llbracket n \rrbracket$, by~\eqref{eq:NLirdefn} and~\eqref{eq:Lidefn} we have that the event $N\! L_{i,0}$ is defined in terms of events $P_{y,0,\ep}$, $B_{y,y'}$ and $Y_{y,y'}$ with $y'\in \Z^d$ and $y\in H^i_t$ for some $t\in [0,T]$ (using that $\mathcal B_{R'}(x_i)\subseteq H^i_0$).
We have that}
$H^i_t \subseteq \ball_{R'+3L+D}(x_i)$ $\forall t\in [0,T]$, and $R'+3L+D<3D/2$ by~\eqref{eq:R'D}, and
\new{so for any $r\in \N_0$, (using~\eqref{eq:NLirdefn} again for the case $r>0$) the event $N\! L_{i,r}$ is defined in terms of events $P_{y,r,\ep}$, $B_{y,y'}$ and $Y_{y,y'}$ with $y'\in \Z^d$ and $y\in \ball_{r+3D/2}(x_i)$. Since}
 the families of events $((P_{y,r,\ep})_{r\in \N_0},(B_{y,y'})_{y'\in \Z^d},(Y_{y,y'})_{y'\in \Z^d})$ are independent for different $y$,
 \new{it follows that $(N\! L_{i,r_i})_{i\in A^* \setminus S^*}$ is independent of $N\! L_{i^*,r_{i^*}}$.}
Therefore
\begin{equation} \label{eq:NLi*}
\p{ \bigcap_{i\in \new{A^*}} N\! L_{i,r_i}}  \le \p{N\! L_{i^*,r_{i^*}} \cap \bigcap_{\new{i\in A^* \backslash S^*}}N\! L_{i,r_i}} = \p{N\! L_{i^*,r_{i^*}}}\p{\bigcap_{\new{i\in A^* \setminus S^*}} N\! L_{i,r_i}}.
\end{equation}
\new{Since $i^*\in S^*$, and therefore $|A^* \setminus S^*|<|A^*|$, we have by our induction hypothesis that~\eqref{eq:claimA} holds with $A=A^* \setminus S^*$, i.e.
\begin{equation} \label{eq:NLA*minusS*}
\p{\bigcap_{\new{i\in A^* \setminus S^*}} N\! L_{i,r_i}}
\le \prod_{i\in A^* \setminus S^*} \left(f(r_i +\ep ^{-1})^{\frac 3 {20}D (r_i+\ep ^{-1})^{-1}}\1_{r_i>0}+\frac b 2 \1_{r_i=0}\right).
\end{equation}
Since $r_{i^*}>0$, by~\eqref{eq:NLirdefn}, item~3 in Lemma \ref{lem:bad_blue}, and a simple union bound we have
\begin{equation} \label{eq:NLi*r*bound1}
\p{N\! L_{i^*,r_{i^*}}} \le f(r_{i^*}+\ep^{-1})^{1/5}.
\end{equation}
Furthermore, by the definition of $S^*$ in~\eqref{eq:S*defn},}
since $x_i \in (\lfloor (\mu+\aminus v_0) T \rfloor - \lfloor (\mu+\aplus v_0)T \rfloor)  \Z$ $\forall i\in \llbracket n \rrbracket$, and \new{since} $\|x_i - x_{i'}\| \ge 3D$ for all $i\ne i'$ \new{by~\eqref{eq:xixjnottooclose}, we have}
\begin{equation} \label{eq:S*bound}
|S^*| \le \frac{2}{3D}(2 r_{i^*}+3D)\le \frac 4 {3D}(r_{i^*}+\ep^{-1}),
\end{equation}
\new{where the second inequality follows since at the start of the proof,}
we chose $\ep$ sufficiently small that $\ep^{-1} \ge 3D/2$. Hence,
\new{using~\eqref{eq:S*bound} and then~\eqref{eq:NLi*r*bound1},}
\begin{align} \label{eq:pNLest}
\p{N\! L_{i^*,r_{i^*}}} 
&\le \p{N\! L_{i^*,r_{i^*}}}^{|S^*|/(\frac 4 {3D}(r_{i^*}+\ep^{-1}))} \notag \\
&\le \left(f(r_{i^*}+\ep^{-1})^{\frac 3 {20}D(r_{i^*}+\ep^{-1})^{-1}}\right)^{|S^*|} \notag \\
&\le \prod_{i\in S^*} \left(f(r_i +\ep^{-1})^{\frac 3{20} D(r_{i}+\ep^{-1})^{-1}}\1_{r_i>0}+\frac b 2 \1_{r_i=0}\right),
\end{align}
where the last inequality follows \new{since $r_i\le r_{i^*}$ $\forall i\in S^*$ and}
using the fact that, by~\eqref{eq:fdefn}, the function $r\mapsto f(r+\ep^{-1})^{(r+\ep^{-1})^{-1}}=\exp(-\frac 1 {9R_1}\log (r+\ep^{-1} +1))=(r+\ep^{-1}+1)^{-1/(9R_1)}$ is decreasing in $r$, and since we chose $\ep$ sufficiently small that $f(\ep^{-1})^{\frac 3 {20}D\ep}=(\ep^{-1}+1)^{-D/(60R_1)}<b/2$. 
\new{By combining~\eqref{eq:NLi*},~\eqref{eq:NLA*minusS*} and~\eqref{eq:pNLest}, it follows that~\eqref{eq:claimA} holds with $A=A^*$.}

\new{We have now established that~\eqref{eq:claimA} holds with $A=A^*$ in each of Cases~1 and~2. Hence by induction,~\eqref{eq:claimA} holds for any $A\subseteq \llbracket n \rrbracket$, and so in particular the claim~\eqref{eq:claim} holds.}

\new{We can now complete the proof:} using \eqref{eq:second} and \eqref{eq:claim}, we get
\[
\E{\psub{\xi}{\bigcap_{i=1}^n \mathcal C_{e_i} }}  
\le \sum_{(r_i)_{i=1}^n\in \N_0^n}\p{\bigcap_{i=1}^n N\! L_{i,r_i}}
\le \left(\sum_{r=1}^\infty f(r +\ep^{-1})^{\frac 3{20} D(r+\ep^{-1})^{-1}}+\frac b 2 \right)^n.
\]
Then by the definition of $f$ in~\eqref{eq:fdefn},
\[
\sum_{r=1}^\infty f(r +\ep^{-1})^{\frac 3{20}D(r+\ep^{-1})^{-1}}
=\sum_{r=1}^\infty e^{-\frac 1{60 R_1}D \log (r+\ep^{-1}+1)}
=\sum_{r=1}^\infty (r+\ep^{-1}+1)^{-\frac D {60 R_1} } \le \frac{b}{2}
\]
by our choice of $\ep$ in~\eqref{eq:sumb2}.
It follows that  
\[
\E{\psub{\xi}{\bigcap_{i=1}^n \mathcal C_{e_i}}} \le b^n,
\]
which was to be proven.
\end{proof}
In the remainder of this section, we will prove Lemma~\ref{lem:Ce_inclusion};
we will use the following result in the proof.
\begin{lemma} \label{lem:Ktbds}
\new{There exists $N_0>0$ such that if $\Lambda$ satisfies~\eqref{eq:Lambda_assumptions}}, for any $\ep>0$ and $y\in \Z^d$, if the coupling in~\eqref{eq:coupling_upper} holds for a random particle configuration $\xi\in \N_0^{\Z^d}$ consisting of a finite number of particles, then under the construction of $(\xi_t)_{t\in [0,T]}$ with $\xi_0=\xi$ at the start of Section~\ref{sec:propclosedinit} \new{with jump rate $\gamma$, jump kernel $p$ and competition kernel $\Lambda$}, if $\bigcap_{x\in \Z^d}(P_{x, \lfloor \|x-y\|\rfloor, \ep})^c$ occurs then
\begin{enumerate}
\item $K_t(y)< K_2 \; \forall t\in [0,T]$,
\item $\sum_{y'\in \Z^d}\Lambda(y-y')\xi^{\mathrm{yellow}}_t(y')\le c_2/3 \; \forall t\in [0,T]$,
\item $\sum_{y'\in \Z^d, \|y'-y\|\ge L}\Lambda(y-y')\xi_t(y')\le c_2/3 \; \forall t\in [0,T]$.
\end{enumerate}
\end{lemma}

\begin{proof}
Recall that for $x,y\in \Z^d$ and $t\in [0,T]$, we write $\xi^{\mathrm{blue},x}_t(y)$ (resp.~$\xi^{\mathrm{red},x}_t(y)$, $\xi^{\mathrm{yellow},x}_t(y)$) to denote the number of blue (resp.~red, yellow) particles at $y$ at time $t$ which have label $x$.
\new{The proof consists of the following steps:
\begin{itemize}
\item Step 1: We use the definition of the event $P_{x,r,\ep}$ in~\eqref{eq:Pxrep_defn} to give bounds on $\xi^{\mathrm{blue},x}_t$, $\xi^{\mathrm{yellow},x}_t$ and $\xi^{\mathrm{red},x}_t$ that hold when the event $(P_{x,r,\ep})^c$ occurs. These bounds will be used throughout the rest of the proof.
\item Step 2: We prove item 1 in the statement of the lemma.
\item Step 3: We prove item 2 in the statement of the lemma.
\item Step 4: We prove item 3 in the statement of the lemma.
\end{itemize}
We now give the details of the proof.
}

\paragraph{\new{Step 1:}}
Take $x\in \Z^d$ and $r\ge 0$, and suppose the event $(P_{x,r,\ep})^c$ occurs.
Note that by our construction, since $\zeta(x)\le K_0 N$ and $J<K_0N$ we have 
$
\xi^{\mathrm{blue}}_0(x)\le K_0 N.
$
Hence by the definitions of the events $(B^{\mathrm{spread}}_{x,r})^c$ and $(Y^{\mathrm{spread}}_{x,r})^c$, 
for $t\in [0,T]$ and $y\notin \ball_{r/2}(x)$, 
\begin{equation} \label{eq:blueyellowbd}
\xi^{\mathrm{blue},x}_t(y)\leq e^{3T} g(r) K_0 N\quad\text{ and }\quad
\xi^{\mathrm{yellow},x}_t(y)\leq e^{3T} g(r)  N^{1/2},
\end{equation}
and for any $t\in [0,T]$ and $y\in \Z^d$,
\begin{equation} \label{eq:blueyellownum}
 \xi^{\mathrm{blue},x}_t(y)\leq e^{3T}g(r)^{-2} K_0 N\quad\text{ and }\quad
\xi^{\mathrm{yellow},x}_t(y)\leq e^{3T}g(r)^{-2} N^{1/2}.
\end{equation}
Moreover, since $Z^{(x)}_\ep \le r/4$, 
by the definition of the event $(P_{x,r,\varepsilon})^c$
we have $$\zeta^{(x)}(y)\le g(r/4)^{-1}K_0 N \1_{\|y-x\|<r/4}\quad \forall y\in \Z^d,$$ and so
by the definition of the event $(R^{\mathrm{spread}}_{x,r})^c$, for $t\in [0,T]$ and $y\notin \ball_{r/2}(x)$,
\begin{equation} \label{eq:redbdfar}
\xi^{\mathrm{red},x}_t(y)\leq e^{3T} g(r) K_0 N
\end{equation}
and for any $t\in [0,T]$ and $y\in \Z^d$,
\begin{equation} \label{eq:redbdclose}
 \xi^{\mathrm{red},x}_t(y)\leq e^{3T}g(r)^{-2} K_0 N.
\end{equation}

\paragraph{\new{Step 2: Proof of item 1.}}
Take $y\in \Z^d$ and suppose the event $\bigcap_{x\in \Z^d}(P_{x,\lfloor \|x-y\|\rfloor,\ep})^c$ occurs. For $t\in [0,T]$,
by~\eqref{eq:Ktdefn},
\begin{align}  \label{eq:Ktsumbd}
K_t(y)&=\sum_{y'\in \Z^d}\Lambda(y-y')\xi_t(y') \notag \\
&=\sum_{x,y' \in \Z^d}\Lambda(y-y')(\xi^{\mathrm{red},x}_t(y')+\xi^{\mathrm{blue},x}_t(y')+\xi^{\mathrm{yellow},x}_t(y')).
\end{align}
We will split the sum on the right hand side of~\eqref{eq:Ktsumbd} according to whether $y'\notin \ball_{\lfloor \|x-y\|\rfloor /2}(x)$ or $y'\in \ball_{\lfloor \|x-y\|\rfloor /2}(x)$.
First, by~\eqref{eq:blueyellowbd} and~\eqref{eq:redbdfar}, for $N$ sufficiently large,
\begin{align} \label{eq:Ktfivestar}
&\sum_{x,y' \in \Z^d, y'\notin \ball_{\lfloor \|x-y\|\rfloor /2}(x)}\Lambda(y-y')(\xi^{\mathrm{red},x}_t(y')+\xi^{\mathrm{blue},x}_t(y')+\xi^{\mathrm{yellow},x}_t(y')) \notag \\
&\quad \le 
 \sum_{x,y' \in \Z^d, y' \notin \ball_{\lfloor \|x-y \|\rfloor /2}(x)}\Lambda(y-y')e^{3T}
g(\lfloor \|x-y\|\rfloor) (2K_0 N +N^{1/2}) \notag \\ 
&\quad \le 
 \sum_{y' \in \Z^d}\Lambda(y-y')\sum_{x\in \Z^d}
e^{3T} g(\lfloor \|x-y\|\rfloor) (2K_0  +1) N \notag \\ 
&\quad < N^{-1} \cdot \tfrac 12 K_2 N \notag \\
&\quad =\tfrac 12 K_2,
\end{align}
where the third inequality follows by the definition of $K_2$ in~\eqref{eq:K1defn} and since $\sum_{z\in \Z^d} \Lambda(z)=N^{-1}$.
To bound the other terms in the sum on the right hand side of~\eqref{eq:Ktsumbd}, 
by~\eqref{eq:blueyellownum} and~\eqref{eq:redbdclose}, we have
\begin{align} \label{eq:sumyinball}
&\sum_{x,y' \in \Z^d, y'\in \ball_{\lfloor \|x-y\|\rfloor /2}(x)}\Lambda(y-y')(\xi^{\mathrm{red},x}_t(y')+\xi^{\mathrm{blue},x}_t(y')+\xi^{\mathrm{yellow},x}_t(y')) \notag \\
&\quad \le  \sum_{x,y' \in \Z^d, \|x-y'\|<\lfloor \|x-y\|\rfloor /2}\Lambda(y-y')e^{3T}
g(\lfloor \|x-y\|\rfloor )^{-2}(2K_0 N+N^{1/2}).
\end{align}
Now note that for
$x,y' \in \Z^d$ with $\|x-y'\|<\lfloor \|x-y\|\rfloor /2$ we have 
$$
\|y'-y\|\ge \|x-y\|-\|x-y'\|>\|x-y'\|.
$$
It follows that if $\|x-y'\|<\lfloor \|x-y\|\rfloor /2$ then
\begin{equation} \label{eq:insideball}
x\in  \ball_{\|y'-y\|}(y').
\end{equation}
Therefore by~\eqref{eq:sumyinball},
\begin{align*}
&\sum_{x,y' \in \Z^d, y'\in \ball_{\lfloor \|x-y\|\rfloor /2}(x)}\Lambda(y-y')(\xi^{\mathrm{red},x}_t(y')+\xi^{\mathrm{blue},x}_t(y')+\xi^{\mathrm{yellow},x}_t(y'))\\
&\quad \le e^{3T} \sum_{y'\in \Z^d}\Lambda(y-y')\sum_{x \in \ball _{\|y'-y\|}(y')}
g(\lfloor \|x-y\|\rfloor)^{-2} (2K_0 N+N^{1/2})
\\
&\quad \le e^{3T} \sum_{y'\in \Z^d}\Lambda(y-y')|\ball _{\|y'-y\|}(0)|
(2K_0+1)g(2\|y'-y\|)^{-2} N\\
&\quad < \tfrac 12 K_2,
\end{align*}
where the second inequality follows since $\|x-y\|\le \|x-y'\|+\|y'-y\|< 2\|y'-y\|$ for $x\in \ball_{\|y'-y\|}(y')$, and since $g(r)^{-2}$ is non-decreasing in $r$, and the last inequality follows
by the definition of $K_2$ in~\eqref{eq:K1defn2}
and since $N\Lambda(z)\le \1_{\|z\|<R}+e^{-\kappa \|z\|}$ $\forall z\in \Z^d$ 
\new{because $\Lambda$ satisfies~\eqref{eq:Lambda_assumptions}.}
By~\eqref{eq:Ktsumbd} and~\eqref{eq:Ktfivestar}, item~1 now follows.

\paragraph{\new{Step 3: Proof of item 2.}} Again suppose $\cap_{x\in \Z^d}(P_{x,\lfloor \|x-y\|\rfloor,\ep})^c$ occurs, and take $t\in [0,T]$. Then by~\eqref{eq:blueyellowbd} and~\eqref{eq:blueyellownum},
\begin{align}  \label{eq:yellowbdinit}
\sum_{y' \in \Z^d}\Lambda(y-y')\xi^{\mathrm{yellow}}_t(y')&=\sum_{x,y' \in \Z^d}\Lambda(y-y')\xi^{\mathrm{yellow},x}_t(y') \notag \\
&\le \sum_{x,y' \in \Z^d, y'\notin \ball_{\lfloor \|x-y \|\rfloor /2}(x)}\Lambda(y-y')
e^{3T} g(\lfloor \|x-y\|\rfloor)N^{1/2} \notag \\
&\qquad +\sum_{x,y' \in \Z^d, y' \in \ball_{\lfloor \|x-y \|\rfloor /2}(x)}\Lambda(y-y')
e^{3T}g(\lfloor  \|x-y\|\rfloor)^{-2}N^{1/2}.
\end{align}
For the first sum on the right hand side of~\eqref{eq:yellowbdinit}, we have
\begin{align} \label{eq:yellowbd2}
& \sum_{x,y' \in \Z^d, y'\notin \ball_{\lfloor \|x-y \|\rfloor /2}(x)}\Lambda(y-y')
e^{3T} g(\lfloor \|x-y\|\rfloor)N^{1/2} \notag \\
&\quad \le  \sum_{y' \in \Z^d}\Lambda(y-y')
e^{3T} \sum_{x\in \Z^d}g(\lfloor \|x-y\|\rfloor)N^{1/2} \notag \\
&\quad = N^{-1/2}e^{3T}\sum_{z\in \Z^d} g(\lfloor \|z\|\rfloor),
\end{align}
where the last line follows since $\sum_{z\in \Z^d}\Lambda(z)=N^{-1}$.
For the second sum on the right hand side of~\eqref{eq:yellowbdinit},
by~\eqref{eq:insideball}, and then since $g(r)^{-2}$ is non-decreasing in $r$ and $\|x-y\|\le \|x-y'\|+\|y'-y\|< 2\|y'-y\|$ for $x\in \ball_{\|y'-y\|}(y')$,
we have
\begin{align} \label{eq:yellowbd3}
&\sum_{x,y' \in \Z^d, y' \in \ball_{\lfloor \|x-y \|\rfloor /2}(x)}\Lambda(y-y')
e^{3T}g(\lfloor  \|x-y\|\rfloor)^{-2}N^{1/2} \notag \\
&\quad \le \sum_{y' \in \Z^d}\Lambda(y-y')
e^{3T}\sum_{x\in \ball_{\|y'-y\|}(y')}g(\lfloor  \|x-y\|\rfloor)^{-2}N^{1/2} \notag \\
&\quad \le \sum_{y' \in \Z^d}\Lambda(y-y') e^{3T} |\ball_{\|y'-y\|}(0)| g(2\|y'-y\|)^{-2} N^{1/2} \notag \\
&\quad \le \sum_{z\in \Z^d} (\1_{\|z\|<R}+e^{-\kappa \|z\|}) e^{3T} |\ball_{\|z\|}(0)| g(2\|z\|)^{-2} N^{-1/2},
\end{align}
where the last line follows since $N\Lambda(z)\le \1_{\|z\|<R}+e^{-\kappa\|z\|}$ $\forall z\in \Z^d$ \new{because $\Lambda$ satisfies~\eqref{eq:Lambda_assumptions}.}
Taking $N$ sufficiently large, item~2 follows from~\eqref{eq:yellowbdinit},~\eqref{eq:yellowbd2} and~\eqref{eq:yellowbd3}.

\paragraph{\new{Step 4: Proof of item 3.}}
Once again suppose $\cap_{x\in \Z^d}(P_{x,\lfloor \|x-y\|\rfloor,\ep})^c$ occurs; then for $t\in [0,T]$,
\begin{align*} 
&\sum_{y'\in \Z^d, \|y'-y\|\ge L}\Lambda(y-y')\xi_t(y') \notag \\
&\quad =\sum_{y',x \in \Z^d, \|y'-y\|\ge L}\Lambda(y-y')(\xi^{\mathrm{red},x}_t(y')+\xi^{\mathrm{blue},x}_t(y')+\xi^{\mathrm{yellow},x}_t(y')) \notag \\
&\quad \le \sum_{y',x \in \Z^d, \|y'-y\|\ge L, y'\notin \ball_{\lfloor \|x-y \|\rfloor /2}(x)}\Lambda(y-y')e^{3T}
g(\lfloor \|x-y\|\rfloor )(2K_0 N +N^{1/2})
 \notag \\
&\qquad +\sum_{y',x \in \Z^d, \|y'-y\|\ge L, y'\in \ball_{\lfloor \|x-y \|\rfloor /2}(x)}\Lambda(y-y')e^{3T}
g(\lfloor \|x-y\|\rfloor )^{-2} (2K_0 N+N^{1/2}),
\end{align*}
where the inequality follows from~\eqref{eq:blueyellowbd},~\eqref{eq:blueyellownum},~\eqref{eq:redbdfar} and~\eqref{eq:redbdclose}.
Therefore, using that $L\ge R$ by~\eqref{eq:Tdefn} and~\eqref{eq:R'Ldefn} \new{and using condition (iii) in~\eqref{eq:Lambda_assumptions}}, and (for the second sum) using~\eqref{eq:insideball} and that $\|x-y\|<2\|y'-y\|$ for $x\in \ball_{\|y'-y\|}(y')$,
\begin{align*} 
&\sum_{y'\in \Z^d, \|y'-y\|\ge L}\Lambda(y-y')\xi_t(y') \notag \\
&\quad \le \sum_{y' \in \Z^d, \|y'-y\|\ge L} N^{-1} e^{-\kappa \|y-y'\|}\cdot e^{3T}
\sum_{z\in \Z^d} g(\lfloor \|z\|\rfloor )(2K_0 N+N^{1/2})
 \notag \\
&\qquad +\sum_{y' \in \Z^d, \|y'-y\|\ge L}N^{-1} e^{-\kappa \|y-y'\|} 
|\ball_{\|y'-y\|}(0)| e^{3T}g(2\|y'-y\|)^{-2} (2K_0 N+N^{1/2}) \notag \\
&\quad \le e^{3T} \sum_{z\in \Z^d, \|z\|\ge L}e^{-\kappa \|z\|}(K_3+|\ball_{\|z\|}(0)| (2K_0+1)g(2\|z\|)^{-2}),
\end{align*}
where $K_3$ is defined in~\eqref{eq:K2defn}.
Hence by~\eqref{eq:kappadefn}, and since $L=\epsilon_0 T$ and $c_2 =\delta_0 T^{-1}$ by~\eqref{eq:R'Ldefn}, item~3 follows, which completes the proof.
\end{proof}

We finish this section by proving Lemma~\ref{lem:Ce_inclusion}.
\begin{proof}[Proof of Lemma~\ref{lem:Ce_inclusion}]
Suppose \new{$N_0$} is sufficiently large that Lemma~\ref{lem:Ktbds} holds, \new{and suppose $\Lambda$ satisfies~\eqref{eq:Lambda_assumptions}.} Recall from the statement of the lemma that
$$
H_t = \ball_{\frac{T-t}T R'+3L}(x_0+\tfrac{t}T(y_0-x_0)),\quad t\in [0,T],
$$
and define
$$
\mathcal H = \bigcup_{t\in[0,T]} H_t = \operatorname{conv}\left(\ball_{R'+3L}(x_0)\cup \ball_{3L}(y_0)\right),
$$
where $\operatorname{conv}(S)$ denotes the convex hull of the set $S$.
Define the event
\begin{equation} \label{eq:defn_Aevent}
A =\bigcup_{x \in \ball_{R'}(x_0)} B_{x,y_0} \cup \bigcup_{x\in \bigcup_{t\in [0,(1-\epsilon_1)T]} H_t} Y_{x,y_0} \cup \bigcup_{x \in \mathcal H} Y_{x,x} 
\cup \bigcup_{x\in\mathcal H} \bigcup_{y\in\Z^d} P_{y,\lfloor \|y-x\|\rfloor,\ep }.
\end{equation}
\new{To prove the result, we need to show that the event $A^c\cap \mathcal C_e$ cannot occur. By the definition of the event $\mathcal C_e$ in~\eqref{eq:Cdefn}, it suffices to show that if $A^c$ occurs and if $\xi_0(x)\ge J$ for some $x\in \ball_{R'}(x_0)$ then
\begin{enumerate}
\item there exists $y\in \ball_{R'}(y_0)$ with $\xi_T(y)\ge J$ \quad and
\item for each $t\in [0,T]$, there exists $x'\in \ball_{R'}(x_0)$ with $\xi_t(x')>0.$
\end{enumerate}
}

From now on suppose $A^c$ occurs.
Moreover, suppose $\xi_0(x_1)\ge J$ for some $x_1\in \ball_{R'}(x_0)$.
By our construction, this implies that $\xi_0^{\mathrm{blue}}(x_1)\ge J$.
\new{By~\eqref{eq:defn_Aevent}, for $x\in \mathcal H$, we have that
$\cap_{y\in \Z^d} (P_{y,\lfloor \|y-x\|\rfloor,\ep})^c$ occurs, and so by Lemma~\ref{lem:Ktbds} item 1, 
\begin{equation} \label{eq:Ktlarge}
K_s(x)
< K_2\quad \forall  x\in \mathcal H,\,s\in [0,T].
\end{equation}
We will show first that item~2 above holds; in particular we will show that
\begin{equation} \label{eq:xisustain}
\xi_t(x_1)>0 \quad \forall t\in [0,T].
\end{equation}
Indeed, for $t\le \tau(x_1)\wedge T$, by counting the blue particles with label $x_1$ that stay at $x_1$ until time $t$ without branching or being killed by competition,
and then by~\eqref{eq:Ktlarge} and since $x_1\in \ball_{R'}(x_0)\subseteq \mathcal H$, 
we have 
\begin{align*}
&\xi_t(x_1)\\
&\ge \#\Big\{ j\le \xi^{\mathrm{blue}}_0(x_1) : \beta^{\mathrm{blue},x_1,j}(\emptyset) > t,
 \rho^{\mathrm{blue},x_1,j}(\emptyset) > \int_{0}^{t}K_{s}(x_1)ds , X^{\mathrm{blue},x_1,j}_s(\emptyset) =0\,\forall s\in[0,t]\Big\}\\
 &\ge \#\Big\{ j\in \llbracket J \rrbracket : \beta^{\mathrm{blue},x_1,j}(\emptyset) > T,
 \rho^{\mathrm{blue},x_1,j}(\emptyset) \ge K_2 T , X^{\mathrm{blue},x_1,j}_s(\emptyset) =0\,\forall s\in[0,T]\Big\} .
\end{align*}
By the definition of the event $(B_{x_1,y_0})^c$ in~\eqref{eq:Bxydefn}, it follows that $\xi_t(x_1)>0$ $\forall t\in [0,\tau(x_1)\wedge T]$.

If $\tau(x_1)< T$, then for $t\in [ \tau(x_1),T]$, by counting the yellow particles with label $x_1$ that stay at $x_1$ until time $t$ without branching or being killed by competition, and then by~\eqref{eq:Ktlarge}, 
we have 
\begin{align*}
\xi_t(x_1)
&\ge \#\Big\{ j\in \llbracket N^{1/2} \rrbracket : \beta^{\mathrm{yellow},x_1,j}(\emptyset) > t-\tau(x_1),
 \rho^{\mathrm{yellow},x_1,j}(\emptyset) > \int_{0}^{t-\tau(x_1)}K_{\tau(x_1)+s}(x_1)ds , \\
 &\hspace{8cm} X^{\mathrm{yellow},x_1,j}_s(\emptyset) =0\,\forall s\in[0,t-\tau(x_1)]\Big\}\\
 &\ge \#\Big\{ j\in \llbracket N^{1/2} \rrbracket : \beta^{\mathrm{yellow},x_1,j}(\emptyset) > T,
 \rho^{\mathrm{yellow},x_1,j}(\emptyset) \ge K_2 T , X^{\mathrm{yellow},x_1,j}_s(\emptyset) =0\,\forall s\in[0,T]\Big\} .
\end{align*}
By the definition of the event $(Y_{x_1,x_1})^c$ in~\eqref{eq:Yxydefn}, it follows that $\xi_t(x_1)>0$ $\forall t\in [ \tau(x_1)\wedge T,T]$.
This completes the proof of the claim~\eqref{eq:xisustain}.
}

\new{It remains to show that item~1 above holds, i.e.~that} we must have $\xi_T(y)\ge J$ for some $y\in \ball_{R'}(y_0)$. We consider two cases. See Figure~\ref{fig:strategy_1} and Figure~\ref{fig:strategy_2} for a schematic illustration of the strategy in each of these cases.

\noindent \textbf{Case 1:}
Suppose first that $\tau(x)> t$ for all $t\in [0,T]$ and $x\in \ball_{\frac{T-t}T R'+2L}(x_0+\frac{t}T(y_0-x_0))$.

By the definition of $\tau(x)$ in~\eqref{eq:tauxdefn},
we have that 
\begin{equation} \label{eq:redbluebd}
\xi^{\mathrm{red}}_{t}(x)+\xi^{\mathrm{blue}}_{t}(x)< N^{1/2}\; \;\forall t\in [0,T], \, x\in \ball_{\frac{T-t}T R'+2L}(x_0+\tfrac{t}T(y_0-x_0)).
\end{equation}
\new{Since we are assuming that $A^c$ occurs,} for $t\in [0,T]$
 and $x\in \ball_{\frac{T-t}T R' +2L}(x_0+\frac t T (y_0-x_0))\subseteq H_t$,  \new{by~\eqref{eq:defn_Aevent}} the event
$\cap_{y\in \Z^d}(P_{y,\lfloor \|y-x\|\rfloor,\ep })^c$ occurs.
Therefore, for $t\in [0,T]$ and $x \in \ball_{\frac{T-t}T R'+L}(x_0+\frac t T (y_0-x_0))$, we have
 \begin{align} \label{eq:Ktbluebd}
 K_{t}(x)
 &=\sum_{y\in \Z^d} \Lambda (x-y)\xi_{t}(y) \notag \\
 &\le \sum_{y \in \Z^d} \Lambda (x-y) \xi^{\mathrm{yellow}}_{t}(y)
 +\sum_{y\in \Z^d, \|x-y\|<L} \Lambda (x-y) \max_{x'\in \ball_{L}(x)}(\xi^{\mathrm{red}}_{t}(x')+\xi^{\mathrm{blue}}_{t}(x')) \notag \\
&\qquad  +\sum_{y\in \Z^d, \|x-y\|\ge L}\Lambda(x-y)\xi_t(y) \notag \\
 &\le \tfrac 13 c_2 +\sum_{y\in \Z^d}\Lambda(x-y)N^{1/2}+\tfrac 13 c_2 \notag \\
 &<c_2
 \end{align}
 for $N$ sufficiently large, 
where the first inequality comes from considering the contributions from yellow particles, red and blue particles in $\ball_L(x)$, and particles outside $\ball_L(x)$ separately, and the second inequality follows 
 by Lemma~\ref{lem:Ktbds} item~2 for the first term,~\eqref{eq:redbluebd} (since $\ball_L(x) \subseteq \ball_{\frac{T-t}T R'+2L}(x_0 + \frac t T (y_0-x_0)))$ for the second term and by Lemma~\ref{lem:Ktbds} item~3 for the third term, and the last inequality since $\sum_{z\in \Z^d}\Lambda(z)=N^{-1}$.
 
Recall that
we have $\xi^{\mathrm{blue}}_0(x_1)\geq J$ for some $x_1\in \ball_{R'}(x_0)$.
 Therefore, for $t\in [0,T]$ and $y\in \ball_L(x_1+(y_0-x_1)\frac t T)\subseteq \ball_{\frac{T-t}T R' +L}(x_0+(y_0-x_0)\frac t T)$, by~\eqref{eq:Ktbluebd} we have $K_{t}(y)<c_2$, and by our assumption for Case~1, we have $\tau(y)>t$.
Then by counting blue particles with label $x_1$ which are not killed by competition or turned into yellow particles before time $T$, and which are in $\ball_L(y_0)$ at time $T$, we have
\begin{align*}
\sum_{y\in \ball_L(y_0)}\xi^{\mathrm{blue},x_1}_{T}(y)
&\ge \#\{(j,u):j\le \xi^{\mathrm{blue}}_0(x_1), u\in \mathcal N_T^{\mathrm{blue},x_1,j},\\
&\qquad \qquad \rho^{\mathrm{blue},x_1,j}(v)> \int_{\alpha^{\mathrm{blue},x_1,j}(v)}^{\beta^{\mathrm{blue},x_1,j}(v)\wedge T}K_{t}(x_1+X_t^{\mathrm{blue},x_1,j}(v))dt \; \forall v\prec u,\\
&\qquad \qquad X^{\mathrm{blue},x_1,j}_T(u)\in \ball_L(y_0-x_1),\;
\tau(x_1+X_t^{\mathrm{blue},x_1,j}(u))>t \; \forall t\in [0,T]
\}\\
&\ge \#\Big\{ (j,u): j\in\llbracket J\rrbracket,\, u\in \mathcal N_T^{\mathrm{blue},x_1,j},\,
 \rho^{\mathrm{blue},x_1,j}(v) \ge c_2 T\,\, \forall v \prec u,\\
&\hspace{5cm} X^{\mathrm{blue},x_1,j}_t(u) \in \ball _L((y_0-x_1)t/T)\,\forall t\in[0,T]\Big\}\\
&\ge J |\ball_L(0)|,
\end{align*}
where the last inequality follows by the definition of the event $(B_{x_1,y_0})^c$ \new{in~\eqref{eq:Bxydefn}}.
Note that by~\eqref{eq:R'Ldefn} and~\eqref{eq:epsilon1defn}, we have $L\le R'$.
Hence, by the pigeonhole principle, there exists $y\in \ball_L(y_0)\subseteq \ball_{R'}(y_0)$ such that $\xi_{T}(y)\geq J$. 
\\

\noindent \textbf{Case 2:}
Suppose instead that there exist $t^*\in [0,T]$ and
$
x^*\in \ball_{\frac{T-t^*}T R'+2L}(x_0+\frac{t^*}T(y_0-x_0))$ such that $\tau(x^*)\le t^*$. 

Note that $\ball_L(x^*) \subseteq H_{t^*}$; in particular, $x^*\in H_{t^*}$.
Let
$$
y^*=\begin{cases}
y_0, & \text{if }t^*\le(1-\epsilon_1)T\\
x^*, & \text{otherwise.}
\end{cases}
$$
Note that $(Y_{x^*,y^*})^c$ occurs, by the definition of the event $A^c$.
Also note that $\ball_L(y^*) \subseteq \mathcal H$ since  $\ball_L(x^*)\subseteq H_{t^*}$ and $\ball_L(y_0)\subseteq H_T$. Since $\mathcal H$ is a convex set, it follows that 
\begin{equation}
\label{eq:x*y*H}
\bigcup_{r\in[0,1]} \ball_L(x^*+r(y^*-x^*)) \subseteq \mathcal H.
\end{equation}
Using \eqref{eq:Ktlarge} and~\eqref{eq:x*y*H}, we get that
\begin{equation} \label{eq:Ktaux*}
K_{\tau(x^*)+s}(x)< K_2 \quad \forall s\in [0,T-\tau(x^*)],\, x\in \ball_L(x^*+(y^*-x^*)(\tfrac s {\epsilon_1 T}\wedge 1)).
\end{equation}

Note that $x^*+(y^*-x^*)(\frac{T -\tau(x^*)}{\epsilon_1 T}\wedge 1)=y^*$ by the definition of $y^*$.
At time $\tau(x^*)$, by our construction, $\lfloor N^{1/2} \rfloor$ red or blue particles at $x^*$ are turned yellow and given label $x^*$.
By counting the descendants of these particles which are not killed by competition before time $T$ and are in $\ball_L(y^*)$ at time $T$, we obtain
\begin{align*}
&\sum_{y\in \ball_L(y^*)} \xi^{\mathrm{yellow},x^*}_{T}(y)\\
&\ge \#\{(j,u):j\in \llbracket N^{1/2} \rrbracket, u \in \mathcal N^{\mathrm{yellow},x^*,j}_{ T -\tau(x^*)},\\
&\qquad \qquad \rho^{\mathrm{yellow},x^*,j}(v)> \int_{\alpha^{\mathrm{yellow},x^*,j}(v)}^{\beta^{\mathrm{yellow},x^*,j}(v)\wedge ( T -\tau(x^*)) }K_{\tau(x^*)+t}(x^*+X_t^{\mathrm{yellow},x^*,j}(v))dt \; \forall v\prec u,\\
&\hspace{8cm} X^{\mathrm{yellow},x^*,j}_{ T -\tau(x^*)}(u)\in \ball_L(y^*-x^*)\}\\
&\ge \#\Big\{ j \in\llbracket N^{1/2} \rrbracket:
\emptyset \in \mathcal N^{\mathrm{yellow},x^*,j}_T,\, \rho^{\mathrm{yellow},x^*,j}(\emptyset) \ge K_2 T,\\
&\hspace{3cm} X^{\mathrm{yellow},x^*,j}_t(\emptyset) \in \ball_L((y^*-x^*)(\tfrac t{\epsilon_1 T}\wedge 1))\,\forall t\in[0,T]
\Big\} \\
&\ge J |\ball_L(0)|,
\end{align*}
where the second inequality follows by~\eqref{eq:Ktaux*}, and the last inequality follows
by the definition of the event $(Y_{x^{*},y^*})^c$ \new{in~\eqref{eq:Yxydefn}}. Using the pigeonhole principle again, it follows that $\xi_{T}(y)\geq J$ for some $y\in \ball_L(y^*)$.

Now suppose $t^*\le (1-\epsilon_1)T$ and so $y^*=y_0$. Then since $L\le R'$ by~\eqref{eq:R'Ldefn} and~\eqref{eq:epsilon1defn}, we have $\xi_{T}(y)\geq J$ for some $y\in \ball_L(y_0)\subseteq \ball_{R'}(y_0)$.
Suppose instead that $t^*>(1-\epsilon_1)T$, and so $y^*=x^*$; then
\begin{align*}
\|x^*-y_0\|&<\tfrac{T-t^*}T R'+2L+\|\tfrac{T-t^*}T (y_0-x_0)\|\\
&<\epsilon_1 R'+2L +\epsilon_1 \max(\|\lfloor (\mu+\aminus v_0) T\rfloor \|, \|\lfloor (\mu+\aplus v_0) T\rfloor \|)\\
&<R'-L,
\end{align*}
where the last line follows since, by~\eqref{eq:epsilon1defn}, $\epsilon_1<1/3$ and $\epsilon_0<\frac 19 \epsilon_2$, and since $\epsilon_1 D_T <\frac 13 \epsilon_2 T$ by~\eqref{eq:Tdefn}, and $R'=\epsilon_2 T$, $L=\epsilon_0 T$ by~\eqref{eq:R'Ldefn}.
Therefore $\ball_L(y^*)\subseteq \ball_{R'}(y_0)$, and
 again we must have $\xi_{T}(y)\geq J$ for some $y\in \ball_{R'}(y_0)$.
 
\new{We have now established that} in both cases, we have $\xi_{T}(y)\geq J$ for some $y\in \ball_{R'}(y_0)$, \new{which completes the proof.}
\end{proof}

\begin{figure}
\centering
\includegraphics[scale=1]{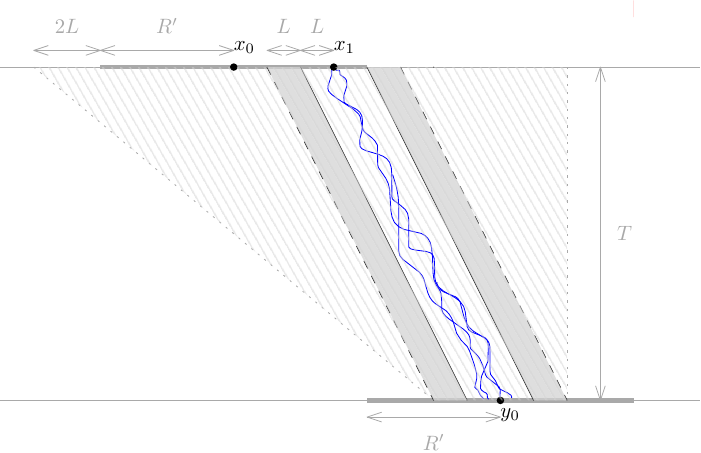}
\caption{\label{fig:strategy_1} 
\new{Illustration of the strategy for Case~1 in the proof of Lemma~\ref{lem:Ce_inclusion}.
We suppose that a good event $A^c$ occurs, and that at time $0$ there are at least $J=\lfloor N^{1/3}\rfloor$ blue particles at some $x_1\in\ball_{R'}(x_0)$.
We want to show that these particles will have at least $J$ descendants at some site in $\ball_L(y_0)\subseteq \ball_{R'}(y_0)$ at time $T$.
We illustrate the space regions $\ball_{R'}(x_0)$ at time $0$ and $\ball_{R'}(y_0)$ at time $T$ with thick grey lines.\\
The hatched space-time region in the figure (including the grey shaded region) is the set of $(x,t)$ such that $t\in [0,T]$ and $x\in \ball_{\frac{T-t'}T R'+2L}(x_0+\frac {t'}T (y_0-x_0))$ for some $t'\in [t,T]$.
In Case~1, we assume that $\tau(x)>t$ for all $t\in [0,T]$ and $x\in \ball_{\frac {T-t}T R' +2L}(x_0+\frac t T (y_0-x_0))$, i.e.~there is no space-time point of the form $(x,\tau(x))$ inside the hatched space-time region, or in other words, no yellow particles are created inside the hatched space-time region.
This implies that the number of blue and red particles in the hatched space-time region is small (less than $N^{1/2}$ at each site).\\
The hatched space-time region contains a tube of radius $L$ around the space-time line from $x_1$ at time $0$ to $y_0$ at time $T$ (this tube is the set of $(x,t)$ such that $t\in [0,T]$ and $x\in \ball_L(x_1+\frac t T(y_0-x_1))$ and is delimited by solid lines in the figure).
By the definition of the good event $A^c$, if the killing rate inside this tube is sufficiently small (less than the small constant $c_2$), then the blue particles at $x_1$ at time $0$ have at least $J$ descendants at some site in $\ball_L(y_0)$ at time $T$ whose trajectories stay inside the tube and which are not killed by competition or turned into yellow particles. (Here we use our assumption in Case~1 that no yellow particles are created inside the hatched space-time region.)
To ensure that the killing rate inside the tube is small, the hatched region also contains an additional `buffer' of width $L$ around the tube (the shaded region in the figure). 
Within the tube, by the definition of the good event $A^c$, the contribution to the killing rate from particles outside the buffer is bounded by $c_2/3$ by virtue of point 3 from Lemma~\ref{lem:Ktbds}. Furthermore, the contribution to the killing rate from particles inside the buffer is small because of point 2 from Lemma~\ref{lem:Ktbds} (for the yellow particles) and because the number of blue and red particles is small there by our assumption in Case~1. The total killing rate inside the tube is therefore small, as required.
}
}
\end{figure}

\begin{figure}
\centering
\includegraphics[scale=1]{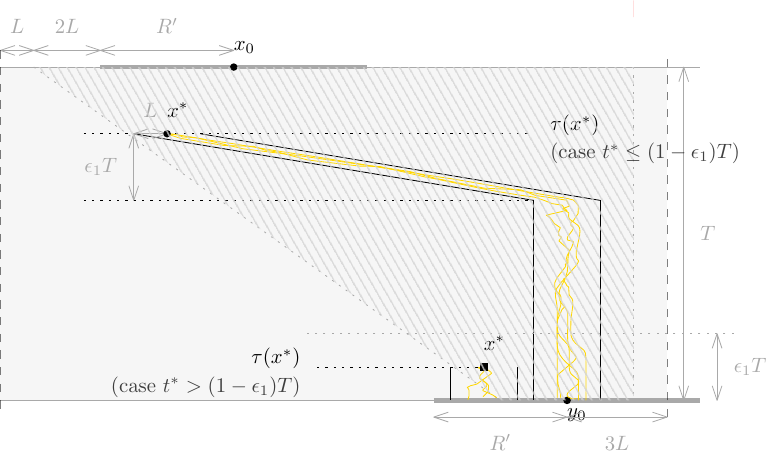}
\caption{\label{fig:strategy_2}
\new{Illustration of the strategy for Case~2 in the proof of Lemma~\ref{lem:Ce_inclusion}.
We suppose that a good event $A^c$ occurs, and we want to show that there are at least $J=\lfloor N^{1/3}\rfloor$ particles at some site in $\ball_{R'}(y_0)$ at time $T$.
We illustrate the space regions $\ball_{R'}(x_0)$ at time $0$ and $\ball_{R'}(y_0)$ at time $T$ with thick grey lines.\\
The hatched space-time region in the figure is the set of $(x,t)$ such that $t\in [0,T]$ and $x\in \ball_{\frac{T-t'}T R'+2L}(x_0+\frac {t'}T (y_0-x_0))$ for some $t'\in [t,T]$.
In Case~2, we assume that $\tau(x^*)\le t^*$ for some $t^*\in [0,T]$ and $x^*\in \ball_{\frac{T-t^*}T R'+2L}(x_0+\frac {t^*}T (y_0-x_0))$, i.e.~yellow particles appear at some point $(x^*,\tau(x^*))$ within the hatched space-time region.
The number of these yellow particles is $\lfloor N^{1/2}\rfloor$ by our construction.
We want to show that these particles have at least $J$ descendants at some site in $\ball_{R'}(y_0)$ at time $T$.
It is enough to consider the descendants whose trajectories up to time $T$ stay within distance $L$ of a certain space-time polyline. 
This polyline is chosen in such a way that the tube of radius $L$ around the polyline stays within the space region denoted by $\mathcal H$ in the proof. 
(The space-time region $\mathcal H \times [0,T]$ is delimited by dashed lines in the figure.)
In the space-time region $\mathcal H \times [0,T]$, the killing rate is bounded by a (large) constant $K_2$, by virtue of point 1 from Lemma~\ref{lem:Ktbds} and using the definition of the good event $A^c$. 
The good event $A^c$ then also ensures that at least $J$ descendants of the $\lfloor N^{1/2}\rfloor$ yellow particles travel along the tube without being killed by competition and reach some site in $\ball_{R'}(y_0)$ at time $T$.\\
The choice of the polyline depends on $t^*$. If $t^* > (1-\epsilon_1)T$, where $\epsilon_1$ is a small constant, then $\ball_L(x^*) \subseteq \ball_{R'}(y_0)$ and the polyline is simply the constant $x^*$. If $t^* \le (1-\epsilon_1)T$, the polyline moves from $x^*$ to $y_0$ within time $\epsilon_1 T$ and then stays there.
We illustrate these two possible cases for $t^*$ in the figure; in each case, the tube of radius $L$ around the polyline is delimited by solid lines in the figure.
}
}
\end{figure}

\section{Missing proofs of lemmas}
\label{sec:badbluepf}

In this section, we prove the remaining lemmas from Section~\ref{sec:propproof}. In Section~\ref{subsec:concentration}, we first recall a few concentration bounds on binomial and sums of independent geometric random variables, which are used in the following sections.  In Section~\ref{subsec:lembadred}, we prove Lemma~\ref{lem:bad_red}, which was used in Section~\ref{sec:propupper} to prove the `coming down from infinity' property of the BRWNLC (Proposition~\ref{prop:upperbound}). Finally, in Section~\ref{subsec:lembadblue}, we prove the lemmas missing for Proposition~\ref{prop:closedinitcond}, namely Lemmas~\ref{lem:blue1} and~\ref{lem:bad_blue}.

\new{Recall that we have fixed $R_1,\gamma, \lambda>0$ and $p$ satisfying~\eqref{eq:p_assumptions} with $p(x)=0$ $\forall x\notin \ball_{R_1}(0)$.}

\subsection{Some concentration inequalities}
\label{subsec:concentration}

For $n\in \N$ and $p_0\in [0,1]$, let $Y(n,p_0)$ be a random variable following the binomial distribution with parameters $n$ and $p_0$, i.e.,
\begin{equation} \label{eq:Ybin}
    Y(n,p_0)\sim \text{Bin}(n,p_0).
\end{equation}
The following classical concentration inequalities can be found e.g.~in Theorem~2.3 in McDiarmid~\cite{mcdiarmid}: for all $n\in\N$, $p_0\in[0,1]$ and $\varepsilon>0$,
\begin{align}
\label{eq:concen_upper}
\P(Y(n,p_0) \ge (1+\varepsilon)np_0) &\le e^{-\frac{\varepsilon^2 np_0}{2(1+\varepsilon/3)}}\\
\label{eq:concen_lower}
\text{and }\quad \P(Y(n,p_0) \le (1-\varepsilon)np_0) &\le e^{-\frac12 \varepsilon^2np_0}.
\end{align}
Let $\mathcal T^1$, $\mathcal T^2,\ldots$ be i.i.d.~copies of the BRW tree with resiliences $\mathcal T$ as defined in~\eqref{eq:calTdefn}, with $\mathcal T^i = ((\mathcal N^i_t)_{t\ge0},(X^i_t(u),t\ge0,u\in \mathcal N^i_t),(\alpha^i(u),u\in \mathcal U),(\beta^i(u),u\in \mathcal U),(\rho^i(u),u\in \mathcal U))$.
The next lemma uses Lemma~\ref{lem:yule} and bounds on sums of independent geometric random variables to give useful bounds on the size of certain sets of particles in the BRW trees.
\begin{lemma} \label{lem:geomsum}
Let $n\in\N$, $m\ge0$. Then
\begin{align}
\label{eq:geom1}
\forall s,t\ge0 \text{ with }(1-e^{-t})e^s < 1, \, \p{\#\{(j,u):j\in \llbracket n \rrbracket , u\in \mathcal N^j_t \}\ge m}
&\le \frac {e^{-s(m-n)}} {(1-(1-e^{-t})e^s)^n},\\
\label{eq:geom1a}
\forall t\ge0 \text{ with }te < 1, \, \p{\#\{(j,u):j\in \llbracket n \rrbracket , u\in \mathcal N^j_t \}\ge m}
&\le \frac {e^{-(m-n)}} {(1-te)^n},\\
\label{eq:geom2}
\forall t\ge0 \text{ with }te < 1, \, \p{\#\{(j,u):j\in \llbracket n \rrbracket , u\in \mathcal N^j_t, u\ne \emptyset \}\ge m}
&\le e^{(te^2/(1-te))n - m}.
\end{align}
\end{lemma}
\begin{proof}
Take $t\ge 0$.
For $j\in \N$, let $G_j = |\mathcal N^j_t |$, so that the random variable on the left hand side of \eqref{eq:geom1} equals $\sum_{j=1}^n G_j$. Now, $G_1,G_2,\ldots $ are i.i.d.~with $G_1 \sim \text{Geom}(e^{-t})$, by Lemma~\ref{lem:yule}. Note that
for $s\ge 0$ such that $e^s(1-e^{-t})<1$,
\[
\E{e^{sG_1}} = \frac {e^{-t}e^s }{1-(1-e^{-t})e^s} \le \frac {e^s }{1-(1-e^{-t})e^s},
\]
using that $e^{-t} \le 1$. Hence by Markov's inequality (or a Chernoff bound),
\begin{align*}
\p{\sum_{j=1}^n G_j \ge m} \le e^{-sm} \left(\frac{e^s}{1-(1-e^{-t})e^s}\right)^n = \frac {e^{-s(m-n)}} {(1-(1-e^{-t})e^s)^n},
\end{align*}
which is the first inequality \eqref{eq:geom1}. The second inequality \eqref{eq:geom1a} readily follows by setting $s=1$ and using that $1-e^{-t} \le t$.

For the third inequality, we note that the random variable on the left hand side of \eqref{eq:geom2} equals $\sum_{j=1}^n G'_j$, where $G'_j = G_j\1_{(G_j \ge 2)}$. Using the above expression for $\E{e^{sG_1}}$ with $s=1$, and using that $1-e^{-t} \le t$, we have that
\begin{align*}
\E{e^{G'_1}}
=\E{e^{G_1}+(1-e)\1_{G_1=1}}
 = \E{e^{G_1}} + (1-e)e^{-t} &\le e^{-t}\left(1+e\left(\frac{1}{1-te} -1\right)\right) \\
 &= e^{-t}\left(1+\frac{te^2}{1-te}\right)\\
& \le \exp\left(\frac{te^2}{1-te}\right),
\end{align*}
using furthermore that $e^{-t}\le 1$ and $1+t'\le e^{t'}$ $\forall t'\in\R$ for the last inequality. Using Markov's inequality again, we get that
\[
\p{\sum_{j=1}^n G'_j \ge m} \le e^{(te^2/(1-te))n - m},
\]
which is the third inequality \eqref{eq:geom2}.
\end{proof}

\subsection{Proof of Lemma~\ref{lem:bad_red}} \label{subsec:lembadred}
As in Section~\ref{subsec:concentration},
let $\mathcal T^1$, $\mathcal T^2,\ldots$ be i.i.d.~copies of the BRW tree with resiliences $\mathcal T$ as defined in~\eqref{eq:calTdefn}, with $\mathcal T^i = ((\mathcal N^i_t)_{t\ge0},(X^i_t(u),t\ge0,u\in \mathcal N^i_t),(\alpha^i(u),u\in \mathcal U),(\beta^i(u),u\in \mathcal U),(\rho^i(u),u\in \mathcal U))$.
We start with a lemma that gives us two bounds on the number of particles in $(\mathcal T^i)_{i\le n}$ that move a large distance from the origin.
\begin{lemma} \label{lem:movefar}
For $n\in \N$ and $m,r,t>0$,
\begin{align} \label{eq:lemmovefarst1}
\p{\#\{(j,u):j\in \llbracket n \rrbracket , u\in \mathcal N^j_t , \sup_{s\in [0,t]}\|X^j_s(u)\|\ge r\}\ge m}
\le m^{-1} n e^t e^{-\frac r {R_1} \log \left( \frac r {R_1\gamma t e}\right)}.
\end{align}
For $n\in \N$, $m,t>0$ and $r\ge 0$, if either $r>0$ and $n e^t e^{-\frac r {R_1} \log \left( \frac r {R_1\gamma t e}\right)}\le m/2$, or $r\ge 0$ and $n e^t \le m/2$, then
\begin{align} \label{eq:lemmovefarst2}
\p{\#\{(j,u):j\in \llbracket n \rrbracket , u\in \mathcal N^j_t , \sup_{s\in [0,t]}\|X^j_s(u)\|\ge r\}\ge m}
\le 8m^{-2} n e^{2t}.
\end{align}
\end{lemma}
\begin{proof}
Take $n\in \N$, $m,t>0$ and $r\ge 0$.
For $j\in \N$, let
$$
A_t^j =\#\{u\in \mathcal N^j_t :\sup_{s\in [0,t]}\|X^j_s(u)\|\ge r\}.
$$
Let $(N_t)_{t\ge 0}$ denote a Poisson process with rate $\gamma$,
and let $(X_t)_{t\ge 0}$ denote a continuous-time random walk starting at 0 with jump rate $\gamma$ and jump kernel $p$.
By the many-to-one lemma (Lemma~\ref{lem:many-to-one}), and then since $p(x)=0$ for $\|x\|\ge R_1$,
\begin{align} \label{eq:lemmovefar*}
\E{A_t^j} = e^t \p{\sup_{s\in [0,t]}\|X_s\|\ge r} \le e^t \p{N_t\ge r/R_1}.
\end{align}
Suppose first that $r\ge R_1 \gamma t$. Then by Markov's inequality,
\begin{align} \label{eq:lemmovefar*new}
\p{N_t\ge r/R_1}
\le e^{-\frac r {R_1} \log \left( \frac r {R_1\gamma t }\right)}\E{e^{ \log \left( \frac r {R_1\gamma t }\right) N_t}} 
&=e^{-\frac r {R_1} \log \left( \frac r {R_1\gamma t }\right)} e^{\gamma t \left( \frac r {R_1 \gamma t} -1\right)} \notag \\
&\le e^{-\frac r {R_1} \log \left( \frac r {R_1\gamma t e}\right)}.
\end{align}
Note that if instead $0<r< R_1 \gamma t$ then we have
\begin{align} \label{eq:lemmovefar**new}
\p{N_t\ge r/R_1}
\le 1
&\le e^{-\frac r {R_1} \log \left( \frac r {R_1\gamma t e}\right)}.
\end{align}
Therefore, for any $r>0$, by Markov's inequality and then by~\eqref{eq:lemmovefar*},
$$
\p{\sum_{j=1}^n A_t^j \ge m }\le m^{-1} \sum_{j=1}^n \E{A^j_t} \le m^{-1} n e^t e^{-\frac r {R_1} \log \left( \frac r {R_1 \gamma t e}\right)},
$$
which completes the proof of~\eqref{eq:lemmovefarst1}.

It remains to prove~\eqref{eq:lemmovefarst2}.
Now suppose either $r>0$ and $n e^t e^{-\frac r {R_1} \log \left( \frac r {R_1\gamma t e}\right)}\le m/2$ or $r\ge 0$ and $ne^t \le m/2$. Then by~\eqref{eq:lemmovefar*},~\eqref{eq:lemmovefar*new} and~\eqref{eq:lemmovefar**new}, $\sum_{j=1}^n \E{A_t^j}\le m/2$ and so
\begin{equation} \label{eq:lemmovefar2}
\p{\sum_{j=1}^n A_t^j \ge m } \le \p{\sum_{j=1}^n \left(A_t^j -\E{A_t^j}\right)\ge \tfrac 12 m }
\le 4m^{-2} \cdot n \text{Var}(A_t^1)
\end{equation}
by Chebychev's inequality.
Then since $A_t^1\le |\mathcal N^1_t|$ and $|\mathcal N^1_t|\sim \text{Geom}(e^{-t})$ by Lemma~\ref{lem:yule},
$$
 \text{Var}(A_t^1) \le \E{(A_t^1)^2} \le \E{|\mathcal N^1_t|^2}=(2-e^{-t})e^{2t}.
$$
The result follows by~\eqref{eq:lemmovefar2}.
\end{proof}
In the next three lemmas, we bound the probabilities of the bad events in the definition of $R_{x,r}$;
this will allow us to prove Lemma~\ref{lem:bad_red}.
\begin{lemma} \label{lem:Rno}
For $N$ sufficiently large, for $k\in \N$ and $x\in \Z^d$,
$$
\p{ R^{\mathrm{cr}}_{x,k} }\le  e^{-(2^k N)^{1/(2d+2)}}.
$$
\end{lemma}
\begin{proof}
Let $k\in\N$ and $x\in\Z^d$. Recall that $R^{\mathrm{cr}}_{x,k} = \bigcup_{i=1}^4 R^{\mathrm{cr},i}_{x,k}$.
We will bound $\p{R^{\mathrm{cr},i}_{x,k}}$ for each $i\in \{1,2,3,4\}$ separately.

\noindent \textbf{Case $i=1$.}
Recall from~\eqref{eq:K0defn} that we chose $K_1$ sufficiently large that $16e^{-\frac 14 \lambda K_1 c_0 c_1}<c_1$.
Then since $(\rho^{\mathrm{blue},x,k,j}(\emptyset))_{j\in \N}$ are i.i.d.~with distribution Exp(1), and recalling~\eqref{eq:Ybin},
\begin{align} \label{eq:Rno1}
\p{R^{\mathrm{cr},1}_{x,k}}
&= \p{Y(\lfloor 2^{k+1}K_1 N \rfloor, e^{-\frac 14 \lambda K_1 c_0 c_1})>c_1 2^{k-2} K_1 N} \notag\\
&\le  \p{Y(\lfloor 2^{k+1}K_1 N \rfloor, e^{-\frac 14 \lambda K_1 c_0 c_1})>2e^{-\frac 14 \lambda K_1 c_0 c_1}\cdot \lfloor 2^{k+1}K_1 N \rfloor} \notag\\
&\le e^{-\frac 3 8 e^{-\frac 14 \lambda K_1 c_0 c_1} \lfloor 2^{k+1}K_1 N \rfloor},
\end{align}
where the last line follows by \eqref{eq:concen_upper}.

\noindent \textbf{Case $i=2$.} Applying \eqref{eq:geom2} from Lemma~\ref{lem:geomsum} with $n = \lfloor 2^{k+1}K_1 N \rfloor$, $m = (c_1/8)n$ and $t = c_02^{1-k}\le c_0$, and using that $c_0e<1$ and $c_0e^2/(1-c_0e)\le c_1/16$ by \eqref{eq:c0defn}, we get that
\begin{align} \label{eq:Rno2}
\p{R^{\mathrm{cr},2}_{x,k}} \le e^{- \frac{c_1}{16} \lfloor 2^{k+1}K_1 N \rfloor}.
\end{align}

\noindent \textbf{Case $i=3$.}
Recall from~\eqref{eq:c0defn} that we chose $c_0>0$ sufficiently small that $1-e^{-\gamma c_0}\le c_1/16$.
Therefore
\begin{align} \label{eq:Rno3}
\p{R^{\mathrm{cr},3}_{x,k}}
&\le \p{Y(\lfloor 2^{k+1}K_1 N \rfloor, 1-e^{-\gamma c_0 2^{1-k}})>\tfrac 18 c_1 \lfloor 2^{k+1}K_1 N \rfloor} \notag \\
&\le \p{Y(\lfloor 2^{k+1}K_1 N \rfloor, \tfrac 1{16} c_1)>\tfrac 18 c_1 \lfloor 2^{k+1}K_1 N \rfloor} \notag \\
&\le e^{-\frac 38 \cdot \frac 1{16} c_1 \lfloor 2^{k+1}K_1 N \rfloor},
\end{align}
where the last line follows by \eqref{eq:concen_upper}.

\noindent \textbf{Case $i=4$.}
By a union bound, and since, by~\eqref{eq:thrdefn}, $r\mapsto \thr(r,k)$ is decreasing in $r\ge 1$, 
\begin{equation}\label{eq:Rno4*}
    \p{R^{\mathrm{cr},4}_{x,k}}\le \sum_{r=1}^\infty p_r,
\end{equation}
where, for $r\in \N$, we let
\begin{align*} 
&p_r:= \mathbb P \Big(\#\{(j,u):j \in \llbracket 2^{k+1} K_1 N\rrbracket , u \in \mathcal N^{\mathrm{blue},x,k,j}_{c_0 2^{1-k}}, \notag \\
&\hspace{4cm} \sup_{s\in [0,c_0 2^{1-k}]}\|X^{\mathrm{blue},x,k,j}_{s}(u)\|\ge r\}
> \thr(r+1,k)\Big).
\end{align*}
For $r\in \N$, by~\eqref{eq:lemmovefarst1} in Lemma~\ref{lem:movefar} with $n=\lfloor 2^{k+1}K_1 N \rfloor$, $t=c_0 2^{1-k}$ and $m=\thr(r+1,k) = c_1 g(r+1)^{1/3} 2^k K_1 N$,
we have
\begin{align} \label{eq:Rno4star}
p_r &\le c_1^{-1} g(r+1)^{-1/3} (2^k K_1 N)^{-1} 2^{k+1}K_1 N e^{c_0} e^{-\frac r {R_1} \log \left(\frac r {R_1 \gamma 2^{1-k}c_0 e} \right)} \notag \\
&\le 2c_1^{-1} e^{c_0} g(r+1)^{-1/3} e^{-\frac r {R_1} \log \left(\frac r {R_1 \gamma c_0 e} \right)}.
\end{align}
We will use this bound for $r\ge\lfloor(2^k N)^{1/(2d+\frac 53)}\rfloor$; we now establish another bound on $p_r$ which will be stronger than~\eqref{eq:Rno4star} for smaller values of $r$.

For $j,r\in \N$, let 
$$
N_{j,r}=|\mathcal N^{\mathrm{blue},x,k,j}_{c_0 2^{1-k}}|\1_{\{\exists u \in \mathcal N^{\mathrm{blue},x,k,j}_{c_0 2^{1-k}},\, s\in [0,c_0 2^{1-k}]\text{ s.t. }\|X^{\mathrm{blue},x,k,j}_{s}(u)\|\ge r\}}.
$$
Note that if $N_{j,r}>0$, then since the jump kernel $p$ is supported on $\ball_{R_1}(0)$, there must be at least $r/R_1$ jumps made by particles in $\mathcal T^{\mathrm{blue},x,k,j}$ by time $c_0 2^{1-k}$. Now recall from Section~\ref{sec:brw} that, conditioned on the Yule tree describing the underlying branching process, the jump times of the particles follow a Poisson process on the tree with intensity measure equal to $\gamma$ times the length measure of the tree. The total length of the branches of the tree up to time $t$ is trivially bounded by $t|\mathcal N^{\mathrm{blue},x,k,j}_t|$ and by Lemma~\ref{lem:yule}, $|\mathcal N^{\mathrm{blue},x,k,j}_t|\sim\operatorname{Geom}(e^{-t})$. It follows that we can couple $N_{j,r}$ with a pair of random variables $(G,Z)$ such that $G\sim \text{Geom}(e^{-c_0 2^{1-k}})$ and, conditional on $G$, $Z\sim \text{Poisson}( c_0 2^{1-k}\gamma G)$, and such that
$$
N_{j,r}\le G \1_{\{Z\ge r/R_1\}}.
$$
Therefore, for $r\in \N$, by Markov's inequality,
\begin{align} \label{eq:GYest}
p_r
&\le \p{\sum_{j=1}^{\lfloor 2^{k+1}K_1 N \rfloor } N_{j,r}> \thr(r+1,k)} 
\le e^{- \thr(r+1,k)} \E{e^{G \1_{Z\ge r/R_1}}}^{\lfloor 2^{k+1}K_1 N \rfloor }.
\end{align}
We can bound the expectation on the right hand side by writing
\begin{align} \label{eq:EGbd}
\E{e^{G \1_{Z\ge r/R_1}}}
&\le \E{e^{G }\1_{G\ge 2r} }+\E{e^{ 2r \1_{Z\ge r/R_1}}\1_{G< 2r }} \notag \\
&\le \sum_{\ell = 2 r }^\infty e^{ \ell} (1-e^{-c_0 2^{1-k}})^{\ell -1}
+1+ e^{ 2r} \p{Z\ge r/R_1, G< 2 r}.
\end{align}
Note that if $X\sim \text{Poisson}(c)$ for some $c>0$, then for $y\ge c$, by Markov's inequality,
$$
\p{X\ge y}\le e^{-y\log(y/c)} \E{e^{X \log(y/c)}}
=e^{-y \log (y/c)} e^{c (y/c -1)}
\le e^{y(1-\log (y/c))}.
$$
Therefore, since $r/R_1\ge c_0 2^{1-k}\gamma \cdot 2r$ by our assumption above~\eqref{eq:c0defn}, by conditioning on $G$ we have
$$
\p{Z\ge r/R_1, G<2r}\le e^{\frac r {R_1}(1-\log (R_1^{-1}(c_02^{2-k}\gamma )^{-1}))}.
$$
Since $1-e^{-t}\le t$ for $t\ge 0$, and $ec_0<1$ by our assumption above~\eqref{eq:c0defn}, it follows from~\eqref{eq:EGbd} that
\begin{align*}
\E{e^{G \1_{Z\ge r/R_1}}}
&\le e\sum_{\ell = 2r }^\infty  (e c_0 2^{1-k})^{\ell -1}
+1+ e^{2 r} e^{\frac r {R_1} (1+\log (R_1 c_0 2^{2-k}\gamma))}\\
&\le e(e c_0)^{2r-1}(1-ec_0)^{-1} +1+(2e^{2R_1+1}R_1 c_0\gamma )^{r/R_1}\\
&\le 1+\tfrac 14 c_1 g(r+1)^{1/3},
\end{align*}
where the last line follows by~\eqref{eq:c0defn3} and since $r\ge 1$.
Hence by~\eqref{eq:GYest},
\begin{align*}
p_r
&\le e^{-c_1 g(r+1)^{1/3} 2^k K_1 N}(e^{\frac 14 c_1 g(r+1)^{1/3}})^{\lfloor 2^{k+1}K_1 N\rfloor }
\le e^{- c_1 g(r+1)^{1/3} 2^{k-1} K_1 N} .
\end{align*}
Therefore, by~\eqref{eq:Rno4*} and~\eqref{eq:Rno4star}, and since, by the definition of $g$ in~\eqref{eq:gdefn},
$g(r+1)^{1/3}\ge (2^k N)^{-(2d+\frac 23)/(2d+\frac 53)}$ for $1\le r\le (2^k N)^{1/(2d+\frac 53)}-1$, we have
\begin{align} \label{eq:Rno4}
\p{R^{\mathrm{cr},4}_{x,k}} 
&\le \sum_{r=1}^\infty \min\Big(2c_1^{-1} e^{c_0} g(r+1)^{-1/3} e^{-\frac r {R_1} \log \left(\frac r {R_1 \gamma c_0 e} \right)},
e^{- c_1 g(r+1)^{1/3} 2^{k-1} K_1 N} \Big) \notag \\
&\le \sum_{r=1}^{\lfloor (2^k N)^{1/(2d+\frac 53)}\rfloor -1}e^{-c_1 (2^k N)^{-(2d+\frac 23)/(2d+\frac 53)}2^{k-1}K_1 N} \notag \\
&\qquad +\sum_{r=\lfloor(2^k N)^{1/(2d+\frac 53)}\rfloor}^\infty 2c_1^{-1} e^{c_0} (r+1)^{(6d+2)/3} e^{-\frac r {R_1} \log (\frac r {R_1 \gamma c_0 e})}
\notag \\
&\le \tfrac 12 e^{-(2^k N)^{1/(2d+2)}}
\end{align}
for $N$ sufficiently large.

The result now follows from~\eqref{eq:Rno1},~\eqref{eq:Rno2},~\eqref{eq:Rno3} and~\eqref{eq:Rno4}, and since 
$$\p{R^{\mathrm{cr}}_{x,k}}\le \sum_{i=1}^4 \p{R^{\mathrm{cr},i}_{x,k}}$$ by a union bound.
\end{proof}

\begin{lemma} \label{lem:Rnum}
For $N$ sufficiently large, for $r\in \N$, $k\in \N$ and $x\in \Z^d$,
$$
\p{R^{\mathrm{num}}_{x,k,r}}\le  \min (f(r)^2 ,N^{-1/2}).
$$
\end{lemma}
\begin{proof}
Take $x\in \Z^d$, $k\in \N$ and $r\in \N$.
By a union bound, and since $r'\mapsto g(r')$ is non-increasing for $r'\ge 0$,
\begin{align} \label{eq:lemRnum*}
&\p{R^{\mathrm{num}}_{x,k,r}} \notag \\
&\le \p{\#\{(j,u): j\in \llbracket 2^{k+1}K_1 N\rrbracket,\,u\in \mathcal N^{\mathrm{blue},x,k,j}_{c_0 2^{1-k}}\} >  g(r)^{-1/5} 2^{k+2}K_1 N} 
 + \sum_{r'=\lfloor r/2 \rfloor\vee 1}^\infty q_{r'},
\end{align}
where
\begin{align*}
    & q_{r'}:=\mathbb P \Big(\#\{(j,u): j\in \llbracket 2^{k+1}K_1 N\rrbracket,\,u\in \mathcal N^{\mathrm{blue},x,k,j}_{c_0 2^{1-k}},\notag \\
&\hspace{5cm} \sup_{s\in [0,c_0 2^{1-k}]} \|X^{\mathrm{blue},x,k,j}_{s}(u)\|\ge r'\} > c_1 g(r'+1)^{1/3} 2^{k}K_1 N\Big).
\end{align*}
Recall from our assumption before~\eqref{eq:c0defn} that $c_0 2^{1-k}\le c_0<e^{-1}$.
For the first term on the right hand side of~\eqref{eq:lemRnum*}, by \eqref{eq:geom1a} from Lemma~\ref{lem:geomsum} with $t=c_0 2^{1-k}$, $n=\lfloor 2^{k+1}K_1 N \rfloor $ and $m=g(r)^{-1/5} 2^{k+2}K_1 N$, we have
\begin{align} \label{eq:lemRnum*new}
&\p{\#\{(j,u): j\in \llbracket 2^{k+1}K_1 N\rrbracket,\,u\in \mathcal N^{\mathrm{blue},x,k,j}_{c_0 2^{1-k}}\} > g(r)^{-1/5} 2^{k+2}K_1 N} \notag \\
&\le \left(\frac{e^{1-2 g(r)^{-1/5}}}{1-ec_0 } \right)^{2^{k+1}K_1 N} \notag \\
&\le e^{(1+\frac 1{32} c_1-2 g(r)^{-1/5})2^{k+1}K_1 N} \notag \\
&\le e^{- g(r)^{-1/5} 2^{k}K_1 N},
\end{align}
where the second inequality holds since $1-ec_0\ge e^{-c_1/32}$ by~\eqref{eq:c0defn}, and the last inequality since $c_1<1$, and so $3 g(r)^{-1/5}/2>1+c_1/{32}$.

Now take $r'\in \N$. By~\eqref{eq:lemmovefarst1} in Lemma~\ref{lem:movefar} with $t=c_0 2^{1-k}$, $n=\lfloor 2^{k+1}K_1 N \rfloor$ and $m = c_1 g(r'+1)^{1/3} 2^{k}K_1 N$, we have
\begin{align} \label{eq:uselemfar}
q_{r'}
&\le c_1^{-1} g(r'+1)^{-1/3} (2^{k}K_1 N)^{-1} 2^{k+1}K_1 N e^{c_0} e^{-\frac {r'} {R_1} \log \left( \frac{r'}{R_1 \gamma c_0 e}\right)} \notag \\
&= 2 c_1^{-1}  e^{c_0} g(r'+1)^{-1/3} e^{-\frac {r'} {R_1} \log \left( \frac{r'}{R_1 \gamma c_0 e}\right)}.
\end{align}
We now establish a second upper bound on $q_{r'}$.
Recall from~\eqref{eq:gdefn} that $g(r)=1\wedge r^{-6d-2}$, and recall from~\eqref{eq:c0defn2} that
$$
e^{c_0} e^{-\frac {r'} {R_1} \log \left( \frac{r'}{R_1 \gamma c_0 e}\right)} \le \tfrac 14 c_1 g(r'+1)^{1/3} \quad \forall r'\in \N.
$$
Hence for $r'\in \N$, by~\eqref{eq:lemmovefarst2} in Lemma~\ref{lem:movefar} with $t=c_0 2^{1-k}$, $n=\lfloor 2^{k+1}K_1 N \rfloor$ and $m= c_1 g(r'+1)^{1/3} 2^k K_1 N$, 
\begin{align*}
q_{r'}
&\le 8 (c_1g(r'+1)^{1/3} 2^{k}K_1 N)^{-2} 2^{k+1}K_1 N e^{2c_0}
= 16 c_1^{-2} e^{2c_0}  g(r'+1)^{-2/3} (2^{k}K_1 N)^{-1}. 
\end{align*}
Therefore, by~\eqref{eq:lemRnum*},~\eqref{eq:lemRnum*new} and~\eqref{eq:uselemfar},
and since $c_1<1$ and $g(s)\le 1$ $\forall s\ge 0$,
\begin{align} \label{eq:Rnumsum}
\p{R^{\mathrm{num}}_{x,k,r}} 
&\le e^{- g(r)^{-1/5} 2^{k}K_1 N} \notag \\
&\quad +\sum_{r'=\lfloor r/2 \rfloor \vee 1}^\infty 16 c_1^{-2} e^{2c_0}  g(r'+1)^{-2/3}\min \Big(e^{-\frac {r'} {R_1} \log \left( \frac{r'}{R_1 \gamma c_0 e}\right)}, (2^{k}K_1 N)^{-1} \Big) .
\end{align}
By splitting the sum according to whether $r'\le \lfloor (\log N)^2 \rfloor$, and since $r'\mapsto g(r')$ is non-increasing and $g(r)\le 1$, we can write
\begin{align*}
\p{R^{\mathrm{num}}_{x,k,r}} 
&\le e^{- 2^{k}K_1 N}
+\1_{\lfloor (\log N)^2 \rfloor>\lfloor r/2 \rfloor \vee 1}
\sum_{r'=\lfloor r/2 \rfloor \vee 1}^{ \lfloor (\log N)^2 \rfloor}
16 c_1^{-2} e^{2c_0}  g((\log N)^2+1)^{-2/3} (2^{k}K_1 N)^{-1} \notag \\
&\quad +\sum_{r'=(\lfloor r/2 \rfloor \vee 1)\vee \lfloor (\log N)^2 \rfloor}^\infty 
16 c_1^{-2} e^{2c_0} g(r'+1)^{-2/3} e^{-\frac {r'} {R_1} \log \left( \frac{r'}{R_1 \gamma c_0 e}\right)}.
\end{align*}
Hence, for $N$ sufficiently large, for any $r\in \N$, $k\in \N$ and $x\in \Z^d$ we have
\begin{align*}
\p{R^{\mathrm{num}}_{x,k,r}} 
&\le N^{-1/2}.
\end{align*}
Moreover, for $N$ sufficiently large, if $r\in \N$ is sufficiently large that $f(r)^2 \le N^{-1/2}$, by~\eqref{eq:Rnumsum} we also have
\begin{align*}
\p{R^{\mathrm{num}}_{x,k,r}} 
&\le e^{- g(r)^{-1/5} 2^{k}K_1 N} +\sum_{r'=\lfloor r/2 \rfloor \vee 1 }^\infty 16 c_1^{-2}  e^{2c_0} (r'+1)^{2(6d+2)/3} e^{-\frac {r'} {R_1} \log \left( \frac{r'}{R_1 \gamma c_0 e}\right)}\\
&\le f(r)^2,
\end{align*}
since $f(r)^2=e^{-\frac{2r}{9R_1}\log(r+1)}$ by~\eqref{eq:fdefn}
and $g(r)^{-1/5}=r^{(6d+2)/5}$ by~\eqref{eq:gdefn}.
This completes the proof.
\end{proof}

\begin{lemma} \label{lem:Rxykr}
For $N$ sufficiently large, for $r,k\in \N$ with $k\le k_r$ and $x,y\in \Z^d$,
$$
\p{R_{x,y,k,r}}\le 2^k \min( f(r)^{2}, N^{-1/2}).
$$
\end{lemma}

\begin{proof}
Recall the definition of $k_r$ in~\eqref{eq:krdefn}.
Take $x,y\in \Z^d$.
For $r,k\in \N$, by a union bound, and since $r'\mapsto g(r')$ is non-increasing and $r'\mapsto k_{r'}$ is non-decreasing,
\begin{align} \label{eq:lemRxyinit}
&\p{R_{x,y,k,r}} \notag \\ &\le \p{\#\{(j,u): j\in \llbracket g(r)^{-1/5} 2^{k+3}K_1 N\rrbracket,\,u\in \mathcal N^{\mathrm{red},x,y,k,j}_{c_0}\}\ge \tfrac 12 k_r^{-1} |\ball_{r/2}(0)|^{-1} g(r)^{-1}K_1^2 N}
\notag \\
&\quad + \sum_{r'=\lfloor r/2\rfloor \vee 1}^\infty \tilde q_{r'},
\end{align}
where
\begin{align*}
    &\tilde q_{r'}:=\mathbb P \big(\#\{(j,u): j\in \llbracket  g(r)^{-1/5} 2^{k+3}K_1 N\rrbracket,\,u\in \mathcal N^{\mathrm{red},x,y,k,j}_{c_0}, \notag\\
&\hspace{4cm}\sup_{t\in [0,c_0]}\|X^{\mathrm{red},x,y,k,j}_t(u)\|\ge r' \} > k^{-1}_{2(r'+1)}|\ball_{r'+1}(0)|^{-1}g(2(r'+1))K_1^2 N \big).
\end{align*}
For $r'\in \N$, by~\eqref{eq:lemmovefarst1} in Lemma~\ref{lem:movefar} with $t=c_0$, $n=\lfloor g(r)^{-1/5} 2^{k+3}K_1 N\rfloor$ and $m=k^{-1}_{2(r'+1)}|\ball_{r'+1}(0)|^{-1}$ $g(2(r'+1))K_1^2 N$,
\begin{align} \label{eq:lemRxyfar}
\tilde q_{r'}
 &\le
(k^{-1}_{2(r'+1)}|\ball_{r'+1}(0)|^{-1}g(2(r'+1))K_1^2 N)^{-1}g(r)^{-1/5} 2^{k+3}K_1 N e^{c_0} e^{-\frac {r'} {R_1} \log \left( \frac {r'} {R_1 \gamma c_0 e} \right) }\notag \\
&=k_{2(r'+1)}|\ball_{r'+1}(0)|g(2(r'+1))^{-1}K_1^{-1}g(r)^{-1/5}2^{k+3}e^{c_0}
e^{-\frac {r'} {R_1} \log \left( \frac {r'} {R_1 \gamma c_0 e} \right) }.
\end{align}
We now bound the first term on the right hand side of~\eqref{eq:lemRxyinit}.
Recall from our assumption before~\eqref{eq:c0defn} that $c_0 e<1$.
Therefore
by \eqref{eq:geom1a} from Lemma~\ref{lem:geomsum} with $t=c_0$, $n=\lfloor g(r)^{-1/5} 2^{k+3}K_1 N \rfloor$ and $m=\tfrac 12 k_r^{-1}|\ball_{r/2}(0)|^{-1} g(r)^{-1} K_1^2 N$, if $k\le k_r$,
\begin{align} \label{eq:geomsum*R}
&\p{\#\{(j,u): j\in \llbracket g(r)^{-1/5} 2^{k+3}K_1 N\rrbracket,\,u\in \mathcal N^{\mathrm{red},x,y,k,j}_{c_0}\}\ge \tfrac 12 k_r^{-1}|\ball_{r/2}(0)|^{-1} g(r)^{-1} K_1^2 N} \notag \\
&\le e^{-\frac 12 k_r^{-1}|\ball_{r/2}(0)|^{-1} g(r)^{-1} K_1^2 N}\left(\frac {e} {1-ec_0 }\right)^{g(r)^{-1/5} 2^{k_r+3}K_1 N} \notag \\
&\le e^{(-\frac 12 k_r^{-1}|\ball_{r/2}(0)|^{-1} g(r)^{-4/5} K_1+(1+\frac 1 {32}c_1)2^{k_r+3})g(r)^{-1/5} K_1 N} \notag \\
&\le e^{-\frac 14 k_r^{-1}|\ball_{r/2}(0)|^{-1} g(r)^{-1} K_1^2N},
\end{align}
where the second inequality follows by~\eqref{eq:c0defn} and the last inequality by~\eqref{eq:K0defn2} and since $c_1<1$.
It follows from~\eqref{eq:lemRxyinit} and~\eqref{eq:lemRxyfar} that for $N$ sufficiently large, for $r\in \N$ sufficiently large that $f(r)^{2}\le N^{-1/2}$ and $k\in \N$ with $k\le k_r$,
\begin{align} \label{eq:Rxykrbd1}
\p{R_{x,y,k,r}}&\le 2^{k+3}e^{c_0}K_1^{-1}g(r)^{-1/5}\sum_{r'=\lfloor r/2\rfloor \vee 1}^\infty k_{2(r'+1)}|\ball_{r'+1}(0)| g(2(r'+1))^{-1}e^{-\frac {r'} {R_1} \log \left( \frac {r'} {R_1 \gamma c_0 e} \right) }
\notag \\
&\qquad +e^{-\frac 14 k_r^{-1}|\ball_{r/2}(0)|^{-1} r^{6d+2} K_1^2N}
\notag \\
&\le 2^k f(r)^{2}.
\end{align}
For $r\in \N$ and $r'\ge \lfloor r/2\rfloor \vee 1$, we have $r\le 2(r'+1)$, and so
$g(r)^{1/5}\ge g(2(r'+1))^{1/5}$ and $k_r\le k_{2(r'+1)}$. Therefore
by~\eqref{eq:K0defn3}, for $k\in \N$ with $k\le k_r$,
$$
g(r)^{-1/5}2^{k+3}e^{c_0}e^{-\frac {r'} {R_1} \log \left( \frac {r'} {R_1 \gamma c_0 e} \right) }\le \tfrac 12 k_{2(r'+1)}^{-1} |\ball_{r'+1}(0)|^{-1} g(2(r'+1))K_1.
$$
Hence by~\eqref{eq:lemmovefarst2} in Lemma~\ref{lem:movefar} with $n=\lfloor g(r)^{-1/5}2^{k+3}K_1 N\rfloor$, $m=k_{2(r'+1)}^{-1}|\ball_{r'+1}(0)|^{-1}g(2(r'+1))K_1^2 N$ and $t=c_0$, we have that for $r,k \in \N$ with $k \le k_r$ and $r'\ge \lfloor r/2\rfloor \vee 1$,
\begin{align*}
\tilde q_{r'}
 &\le 
8e^{2c_0}k_{2(r'+1)}^2 |\ball_{r'+1}(0)|^2 g(2(r'+1))^{-2}K_1^{-4}N^{-2}\cdot g(r)^{-1/5}2^{k+3}K_1 N
\\
&=8e^{2c_0}2^{k+3} K_1^{-3}k_{2(r'+1)}^2 |\ball_{r'+1}(0)|^2 g(2(r'+1))^{-2}g(r)^{-1/5}N^{-1}.
\end{align*}
By~\eqref{eq:lemRxyinit},~\eqref{eq:lemRxyfar} and~\eqref{eq:geomsum*R},
it follows that for $N$ sufficiently large, for $r,k\in \N$ with $f(r)^{2}\ge N^{-1/2}$ and $k\le k_r$,
we have $r\le (\log N)^2$ and so
\begin{align*}
&\p{R_{x,y,k,r}} \\
&\le e^{-\frac 14 k_r^{-1} |\ball_{r/2}(0)|^{-1}g(r)^{-1}K_1^2 N}\\
&\quad +\sum_{r'=\lfloor r/2\rfloor \vee 1}^{ \lfloor (\log N)^2 \rfloor}8e^{2c_0}2^{k+3} K_1^{-3}k_{2(r'+1)}^2 |\ball_{r'+1}(0)|^2 g(2(r'+1))^{-2}g(r)^{-1/5}N^{-1}\\
&\quad +
2^{k+3}e^{c_0}K_1^{-1}g(r)^{-1/5}\sum_{r'= \lfloor (\log N)^2 \rfloor}^\infty  k_{2(r'+1)}|\ball_{r'+1}(0)| g(2(r'+1))^{-1}e^{-\frac {r'} {R_1} \log \left( \frac {r'} {R_1 \gamma c_0 e} \right) }
\\
&\le
2^{k} N^{-1/2},
\end{align*}
which, together with~\eqref{eq:Rxykrbd1}, completes the proof.
\end{proof}
We can now combine Lemmas~\ref{lem:Rno},~\ref{lem:Rnum} and~\ref{lem:Rxykr} to prove Lemma~\ref{lem:bad_red}.
\begin{proof}[Proof of Lemma~\ref{lem:bad_red}]
Fix $\ep>0$, and then take
$x\in \Z^d$.
Recall from~\eqref{eq:krdefn} that $k_r = \lfloor (2d+3)\log_2 (r+1) \rfloor $ for $r\ge 1$; let $k_0=0$.
By a union bound, for $r\in \N$,
\begin{align} \label{eq:Pxrsum}
\p{R_{x,r} }
&\le  \sum_{k=k_r+1}^{\infty} \p{R^{\mathrm{cr}}_{x,k}}+ \sum_{k=1}^{k_r} \p{R^{\mathrm{num}}_{x,k,r} }+ \sum_{y\in \ball_{r/2}(x)}\sum_{k=1}^{k_r-1} \p{R_{x,y,k,r}},
\end{align}
and $\p{R_{x,0} }\le  \sum_{k=k_0+1}^{\infty} \p{R^{\mathrm{cr}}_{x,k}}$.
For the first sum on the right hand side of~\eqref{eq:Pxrsum}, for $N$ sufficiently large, by Lemma~\ref{lem:Rno} we have
\begin{align*}
\sum_{k=k_r+1}^{\infty} \p{R^{\mathrm{cr}}_{x,k}}
\le \sum_{k=k_r+1}^{\infty} e^{-(2^k N)^{1/(2d+2)}}
&\le \tfrac 13 \min\big(\ep, f(r+1+\ep ^{-1}) \big)
\end{align*}
for $N$ sufficiently large, by~\eqref{eq:fdefn} and since $2^{(k_r+1)/(2d+2)}\ge (r+1)^{(2d+3)/(2d+2)}$.
For the second sum on the right hand side of~\eqref{eq:Pxrsum}, for $N$ sufficiently large, by Lemma~\ref{lem:Rnum}, for $r\in \N$,
\begin{align*}
\sum_{k=1}^{k_r} \p{R^{\mathrm{num}}_{x,k,r} }
&\le k_r \min(f(r)^2,N^{-1/2})\le \tfrac 13 \min\big(\ep, f(r+1+\ep ^{-1}) \big)
\end{align*}
for $N$ sufficiently large, since $k_r f(r)^2 \le \tfrac 13 \min\big(\ep, f(r+1+\ep ^{-1}) \big)$ for $r$ sufficiently large.
For the third sum on the right hand side of~\eqref{eq:Pxrsum}, for $N$ sufficiently large, by Lemma~\ref{lem:Rxykr}, for $r\in \N$,
\begin{align*}
\sum_{y\in \ball_{r/2}(x)}\sum_{k=1}^{k_r-1} \p{R_{x,y,k,r}}
&\le |\ball_{r/2}(0)|k_r 2^{k_r} \min(f(r)^{2},N^{-1/2})\\
&\le \tfrac 13 \min\big(\ep, f(r+1+\ep ^{-1}) \big)
\end{align*}
for $N$ sufficiently large, since $|\ball_{r/2}(0)|k_r 2^{k_r} f(r)^{2}\le \tfrac 13 \min\big(\ep, f(r+1+\ep ^{-1}) \big)$ for $r$ sufficiently large.
Collecting these inequalities and plugging them into \eqref{eq:Pxrsum} completes the proof.
\end{proof}

\subsection{Proof of Lemmas \ref{lem:blue1} and \ref{lem:bad_blue}} \label{subsec:lembadblue}

We first prove Lemma~\ref{lem:blue1}. We will see that it is basically a consequence of the following large deviation result, which is a direct consequence of Theorem~1.2 in \cite{deacosta} (see also Theorem~13 part 1 in~\cite{borovkov}).

Let $(X_t)_{t\ge 0}$ be a continuous-time random walk starting at 0 with jump rate $\gamma$ and jump kernel $p$. Recall from our assumptions \new{in~\eqref{eq:p_assumptions}} that $p$ has finite range, and so in particular $\E{e^{\theta \|X_1\|}}<\infty$ $\forall \theta>0$. 
Also, recall from~\eqref{eq:Idefn} that for $v\in \R^d$, we let $I (v)=\sup_{u \in \R^d} \left(\langle v,u \rangle - \log \E{e^{\langle u, X_1 \rangle}}\right).$
\begin{lemma} \label{lem:largedev}
For $v\in \R^d$ and $\epsilon>0$,
$$
\liminf_{t\to \infty} \frac 1 t \log \p{\left\| \frac{X_{ts}}t -sv \right\|<\epsilon \; \forall s\in [0,1]}\ge -I (v).
$$
\end{lemma}

\begin{proof}[Proof of Lemma~\ref{lem:blue1}]
Let $v_0\in\R^d\backslash\{0\}$ and let $a\in [0,a_{v_0})$. Recall from Lemma~\ref{lem:basicLambda} and the remark after \eqref{eq:a0defn} that $I(\mu) = 0$ and $I(\mu+a_{v_0} v_0) \le 1$. By convexity of $I$, we then have $I(\mu+av_0) \le a/a_{v_0} < 1$. Furthermore, $\mu+av_0\in\dom(I)^\circ$ because it is a non-trivial convex combination of $\mu\in\dom(I)^\circ$ (by Lemma~\ref{lem:basicLambda}) and $\mu+a_{v_0} v_0\in \dom(I)$, see Theorem~6.1 in \cite{Rockafellar1970}. Hence, using the continuity of $I$ on $\dom(I)^\circ$ provided by Lemma~\ref{lem:basicLambda}, there exists $\delta_0>0$ such that $I(v) < 1-4\delta_0$ for all $v \in V$, where
\[
V = \{v\in \R^d:\|v-(\mu+av_0)\|\le \delta_0\}.
\]

Let $\epsilon>0$ and take a finite $\epsilon/2$-mesh $V_\epsilon$ of $V$, i.e.~$V_\epsilon$ is a finite subset of $V$ such that for all $v\in V$ there exists $v^*\in V_\epsilon $ with $\|v-v^*\|<\epsilon/2$.
By Lemma~\ref{lem:largedev}, we can take $t_0>1$ sufficiently large that for $t\ge t_0$, for each $v\in V_\epsilon$,
\begin{equation} \label{eq:lemblue1*}
\p{\left\| \frac{X_{ts}}t -sv \right\|<\tfrac 12 \epsilon \; \forall s\in [0,1]}\ge e^{-(1-3\delta_0)t}.
\end{equation}
Take $t\ge t_0$ and $x,y\in \R^d$ with $\|y-x-(\mu+av_0)t\|\le \delta_0 t$. Then there exists $v\in V_\epsilon$ with $\|v-(y-x)/t\|<\epsilon/2$ and so
\begin{align*}
\p{X_s \in \ball_{\epsilon t}((y-x)s/t ) \, \, \forall s\in [0,t]}
&\ge \p{\left\| \frac{X_{ts}}t -s v \right\|<\tfrac 12 \epsilon \; \forall s\in [0,1]}\\
&\ge e^{-(1-3\delta_0)t},
\end{align*}
by~\eqref{eq:lemblue1*}. This finishes the proof.
\end{proof}

We now go on to prove Lemma~\ref{lem:bad_blue}. We will prove each item seperately.
As in Section~\ref{subsec:lembadred},
let $\mathcal T^1$, $\mathcal T^2,\ldots$ be i.i.d.~copies of the BRW tree with resiliences $\mathcal T$, with $\mathcal T^i = ((\mathcal N^i_t)_{t\ge0},(X^i_t(u),t\ge0,u\in \mathcal N^i_t),(\alpha^i(u),u\in \mathcal U),(\beta^i(u),u\in \mathcal U),(\rho^i(u),u\in \mathcal U))$.
Take $v_0\in \R^d \backslash \{0\}$ and $0\le \aminus <\aplus <a_{v_0}$ as in Section~\ref{sec:propclosedinit}.
Let $\delta_0 = \min(\delta_0(\new{\gamma,p,}v_0,\aminus),\delta_0(\new{\gamma,p,}v_0,\aplus))>0$ and for $\epsilon>0$, let $t_0(\epsilon)=\max(t_0(\epsilon,\new{\gamma,p,}v_0,\aminus),t_0(\epsilon,\new{\gamma,p,}v_0,\aplus))$, as defined in Lemma~\ref{lem:blue1}.
Item~1 of Lemma~\ref{lem:bad_blue} will follow easily from the following result.
\begin{lemma} \label{lem:blue2}
For $\epsilon >0$ and $J_1,J_2\in \N$, for $t\ge t_0(\epsilon) \vee (\delta_0^{-1} \log (2J_2)+1)$ and $x,y \in \Z^d$ with $\|y-x-(\mu +av_0)t\|<\delta_0 t$ for some $a\in \{\aminus,\aplus\}$,
\begin{align*}
&\p{ \#\Big\{ (j,u): j\in\llbracket J_1\rrbracket,\, u\in \mathcal N_t^{j}, 
 \rho^{j}(v) \ge \delta_0 \,\forall v \prec u,
X^{j}_s(u) \in \ball_{\epsilon t}((y-x)s/t)\,\forall s\in[0,t]\Big\} < J_1J_2}\\
&\qquad \le 2e^{2t}J_1^{-1}.
\end{align*}
\end{lemma}
\begin{proof}
Take $\epsilon>0$, $J_1,J_2\in \N$, $t\ge t_0(\epsilon) \vee (\delta_0^{-1} \log (2J_2)+1)$, $a\in \{\aminus,\aplus\}$ and $x,y \in \Z^d$ with $\|y-x-(\mu +av_0)t\|<\delta_0 t$.
For $j\in \llbracket J_1\rrbracket$, let 
$$ A^j_t = \#\Big\{ u\in \mathcal N_t^{j}: 
 \rho^{j}(v) \ge \delta_0 \,\forall v \prec u,
X^{j}_s(u) \in \ball_{\epsilon  t}((y-x)s/t)\; \forall s\in[0,t]\Big\},
$$
and for $s\ge 0$, let
$$ \bar{A}^j_s = \#\Big\{ u\in \mathcal N_s^{j}: 
 \rho^{j}(v) \ge \delta_0 \,\forall v \prec u\Big\}.
$$
Then $(\bar{A}^j_s)_{s\ge 0}$ is the number of particles in a continuous-time branching process in which each individual branches into two offspring at rate $e^{-2\delta_0}$ and dies at rate $(1-e^{-\delta_0})^2$, with $\E{\bar{A}^j_0}=e^{-\delta_0}$.
It follows that for $s\ge 0$,
$$
\E{\bar{A}^j_s}=\E{\bar{A}^j_0} e^{(e^{-2\delta_0}-(1-e^{-\delta_0})^2)s}=e^{-\delta_0} e^{(2e^{-\delta_0}-1)s}.
$$
By the many-to-one lemma (Lemma~\ref{lem:many-to-one}), and then by Lemma~\ref{lem:blue1},
for each $j\in \llbracket J_1\rrbracket$, it follows that
$$
\E{A^j_t}
= e^{-\delta_0} e^{(2e^{-\delta_0}-1)t} \p{X_s \in \ball_{\epsilon  t}((y-x)s/t)\;  \forall s\in [0,t]}
\ge e^{-\delta_0} e^{(2e^{-\delta_0}-1)t}e^{-(1-3\delta_0)t}\ge e^{ \delta_0 (t-1)},
$$
where the last inequality follows since $e^{-\delta_0}\ge 1-\delta_0 $.
Moreover, since $\mathcal N^j_t\sim \text{Geom}(e^{-t})$ by Lemma~\ref{lem:yule},
$$
\text{Var}(A^j_t)\le \E{(A^j_t)^2} \le \E{|\mathcal N^j_t |^2 }\le 2 e^{2t}.
$$
Since we chose $t\ge \delta_0^{-1} \log (2J_2)+1$, we have $e^{ \delta_0 (t-1)}\ge 2J_2$, and so
by Chebychev's inequality,
\begin{align*}
\p{\sum_{j\in \llbracket J_1\rrbracket}A^j_t <J_1J_2}
\le \p{\left| \sum_{j\in \llbracket J_1\rrbracket} \left(A^j_t -\E{A^j_t}\right)
\right| \ge J_1}
\le J_1^{-2} \sum_{j\in \llbracket J_1\rrbracket}\text{Var}(A^j_t)
\le 2e^{2t}J_1^{-1},
\end{align*}
which completes the proof.
\end{proof}
\begin{proof}[Proof of Lemma~\ref{lem:bad_blue} item 1]
\new{Recall the definition of $B_{x,y}$ in~\eqref{eq:Bxydefn}.
Recall~\eqref{eq:Ybin} and note that for any $x\in \Z^d$,
\begin{align} \label{eq:bluenomove}
    &\p{\#\{ j\in\llbracket J\rrbracket : \beta^{\mathrm{blue},x,j}(\emptyset) > T,
 \rho^{\mathrm{blue},x,j}(\emptyset) \ge K_2 T, X^{\mathrm{blue},x,j}_t(\emptyset) =0\,\forall t\in[0,T]\} =0}\notag \\
 &\quad \le \p{Y(J,e^{-(1+K_2+\gamma)T})=0}=(1-e^{-(1+K_2+\gamma)T})^{J}.
\end{align}}
Take $a\in \{\aminus,\aplus\}$.
Recall from~\eqref{eq:R'Ldefn} that $c_2 T=\delta_0$ and $L=\epsilon_0 T$.
Recall from~\eqref{eq:Tdefn} that
$T\ge t_0\vee (\delta_0^{-1} \log(2|\ball_L(0)|)+1).$
Applying Lemma~\ref{lem:blue2} with $\epsilon=\epsilon_0$, $J_1=J$, $J_2 = |\ball_L(0)|$ and $t=T$, \new{and using~\eqref{eq:bluenomove} and a union bound}, for $x,y \in \Z^d$ with $\|y-x-(\mu +av_0)T\|<\delta_0 T$, 
$$
\p{B_{x,y}}\le 2 e^{2T}J^{-1}\new{+(1-e^{-(1+K_2+\gamma)T})^{J}} < \ep,
$$
for $N$ sufficiently large, recalling that $J=\lfloor N^{1/3}\rfloor$.
\end{proof}
The following lemma will easily imply item~2 of Lemma~\ref{lem:bad_blue}.
\begin{lemma} \label{lem:yellow}
For $\epsilon>0$, for $t\ge t_0(\epsilon)>1$ and $a\in \{\aminus,\aplus\}$, for $N$ sufficiently large, for $x,y \in \Z^d$ with $\|y-x-r(\mu +av_0)t\|<\delta_0 r t$ for some $r\in [\epsilon_1,1]$ or $x=y$,
\begin{align*}
&\mathbb P \Big( \#\Big\{ j\in\llbracket N^{1/2} \rrbracket : 
 \beta^{j}(\emptyset) > t,  \rho^{j}(\emptyset) \ge K_2 t, 
 X^{j}_s(\emptyset) \in \ball_{\epsilon t}((y-x)(\tfrac s{\epsilon_1 t}\wedge 1))\,\forall s\in[0,t]
\Big\} < N^{2/5}\Big)\\
&\qquad \le e^{-N^{2/5}/8}.
\end{align*}
\end{lemma}
\begin{proof}
Take $t\ge t_0(\epsilon)$, $r\in [\epsilon_1,1]$ and $x,y\in \Z^d$ with $\|y-x-r(\mu +av_0)t\|<\delta_0 r t$ or $x=y$.
For $j\in \N$, let
$$
A^j_t = \1_{\beta^{j}(\emptyset) > t,  \rho^{j}(\emptyset) \ge K_2 t, 
X^{j}_s(\emptyset) \in \ball_{\epsilon t}((y-x)(\frac s{\epsilon_1 t}\wedge 1))\,\forall s\in[0,t]}.
$$
Then
\begin{equation} \label{eq:EAjyellow}
\E{A^j_t} \ge e^{-t} e^{-K_2 t} e^{-\gamma(t-\epsilon_1 t)}\p{X_s \in \ball_{\epsilon t}((y-x)s/(\epsilon_1 t))\; \forall s\in [0,\epsilon_1 t]}.
\end{equation}
The probability on the right hand side of \eqref{eq:EAjyellow} can be bounded from below by a negative exponential in $t$. We make the constants explicit. Impatient readers can jump directly to \eqref{eq:direct_jump}.
Let $(N_s)_{s\ge 0}$ be a Poisson process with rate $\gamma$. Let $\zeta_1,\zeta_2,\ldots $ be i.i.d.~with distribution given by the jump kernel $p$, and take an independent sequence $\theta_1, \theta_2, \ldots $ i.i.d.~with $\theta_1$ a Bernoulli random variable with mean $\epsilon_1 /r$.
For $s\ge 0$, let $Y_s=\sum_{i=1}^{N_{sr/\epsilon_1}}\zeta_i$ and $Y'_s=\sum_{i=1}^{N_{sr/\epsilon_1}}\zeta_i \theta_i$.
Then by the thinning property of Poisson processes, $(Y'_s)_{s\ge 0}\stackrel{d}{=}(X_s)_{s\ge 0}$.
Therefore, for any $\eta>0$,
\begin{align*}
&\p{X_s \in \ball_{\epsilon t}((y-x)s/(\epsilon_1 t))\; \forall s\in [0,\epsilon_1 t]}\\
&\ge \p{Y_s \in \ball_{\epsilon t}((y-x)s/(\epsilon_1 t))\; \forall s\in [0,\epsilon_1 t], Y_s =Y'_s \; \forall s\in [0,\epsilon_1 t]}\\
&\ge \p{Y_s \in \ball_{\epsilon t}((y-x)s/(\epsilon_1 t))\; \forall s\in [0,\epsilon_1 t], \theta_i =1 \; \forall i\in \llbracket \eta t \rrbracket, N_{rt}\le \eta t}.
\end{align*}
Note that since $(Y_s)_{s\ge 0}\stackrel{d}{=}(X_{sr/\epsilon_1})_{s\ge 0}$,
\begin{align} \label{eq:timechange}
\p{Y_s \in \ball_{\epsilon t}((y-x)s/(\epsilon_1 t))\; \forall s\in [0,\epsilon_1 t]}
&=\p{X_s \in \ball_{\epsilon t}(r^{-1}(y-x)s/t)\; \forall s\in [0,r t]} \notag \\
&\ge \p{X_s \in \ball_{\epsilon t}(r^{-1}(y-x)s/t)\; \forall s\in [0,t]} \notag \\
&\ge e^{-(\gamma \vee 1)t}
\end{align}
by Lemma~\ref{lem:blue1} and since $t\ge t_0(\epsilon)$ and either $\|r^{-1}(y-x)-(\mu +av_0)t\|<\delta_0  t$ or $x=y$.
Let
$\eta=\gamma(e-1)+1+\gamma \vee 1$; then by Markov's inequality and since $N_{rt} \sim \text{Poisson}(\gamma r t)$,
\begin{align} \label{eq:Ntbig}
\p{N_{rt} >\eta t} \le e^{-\eta t} e^{\gamma r t(e-1)}\le  e^{-(1+\gamma \vee 1)t}
\end{align}
by our choice of $\eta$.
It follows that
\begin{align*}
&\p{X_s \in \ball_{\epsilon t}((y-x)s/(\epsilon_1 t))\; \forall s\in [0,\epsilon_1 t]}\\
&\ge \p{Y_s \in \ball_{\epsilon t}((y-x)s/(\epsilon_1 t))\; \forall s\in [0,\epsilon_1 t], \theta_i =1 \; \forall i\in \llbracket \eta t \rrbracket}\\
&\qquad - \p{\theta_i =1 \; \forall i\in \llbracket \eta t \rrbracket, N_{rt} > \eta t}\\
&=(\epsilon_1/r)^{\lfloor \eta t \rfloor} \left( \p{Y_s \in \ball_{\epsilon t}((y-x)s/(\epsilon_1 t))\; \forall s\in [0,\epsilon_1 t]}-\p{N_{rt}>\eta t}
\right)\\
&\ge (\epsilon_1/r)^{\eta t}(e^{-(\gamma \vee 1)t}-e^{-(1+\gamma \vee 1)t}),
\end{align*}
where the last line follows by~\eqref{eq:timechange} for the first term, and by~\eqref{eq:Ntbig} for the second term.
Hence, since $t\ge t_0(\epsilon)>1$, by~\eqref{eq:EAjyellow} we have
\begin{equation}
\label{eq:direct_jump}
\E{A^j_t} \ge (1-e^{-1}) e^{-(1+K_2 +\gamma+\gamma\vee 1 +\eta \log(1/\epsilon_1)) t} .
\end{equation}
Take $N$ sufficiently large that $N^{1/10} e^{-(1+K_2 +\gamma+\gamma\vee 1 +\eta \log(1/\epsilon_1)) t} \ge 4$. Then 
$\E{A^j_t}\ge 2 N^{-1/10}$ for each $j\in \N$, and so, recalling the notation from~\eqref{eq:Ybin},
\begin{align*}
\p{\sum_{j\in \llbracket N^{1/2} \rrbracket} A^j_t <N^{2/5} }
\le \p{Y(\lfloor N^{1/2} \rfloor, 2  N^{-1/10})<N^{2/5}}
\le e^{-N^{2/5}/8}
\end{align*}
for $N$ sufficiently large, by \eqref{eq:concen_lower}.
\end{proof}
\begin{proof}[Proof of Lemma~\ref{lem:bad_blue} item 2]
\new{Recall the definition of the event $Y_{x,y}$ in~\eqref{eq:Yxydefn}.
First note that for any $x\in \Z^d$,
\begin{align*}
    &\p{\#\{j\in \llbracket N^{1/2}\rrbracket :\beta^{\mathrm{yellow},x,j}(\emptyset)>T, \rho^{\mathrm{yellow},x,j}(\emptyset)\ge K_2 T, X_t^{\mathrm{yellow},x,j}(\emptyset)=0\,\forall t\in [0,T]\}=0}\\
    &\le \p{Y(\lfloor N^{1/2}\rfloor , e^{-(1+K_2+\gamma)T})=0}
    =(1-e^{-(1+K_2+\gamma)T})^{\lfloor N^{1/2} \rfloor}.
\end{align*}}
Recall from~\eqref{eq:Tdefn} that $T\ge t_0=t_0(\epsilon_0)$.
Take $N$ sufficiently large that Lemma~\ref{lem:yellow} holds with $\epsilon=\epsilon_0$ and $t=T$, and that
$\lfloor N^{1/3} \rfloor |\ball_L(0)|\le N^{2/5}$ and $e^{-N^{2/5}/8}\new{+(1-e^{-(1+K_2+\gamma)T})^{\lfloor N^{1/2} \rfloor}} < \ep$. Recall that $J=\lfloor N^{1/3}\rfloor$.
Then for $a\in \{\aminus, \aplus\}$ and $x,y\in \Z^d$ with $\|y-x-r(\mu +av_0)T\|<\delta_0 r T$ for some $r\in [\epsilon_1,1]$ or $x=y$, since $L=\epsilon_0 T$ we have 
$$\p{Y_{x,y}}\le e^{-N^{2/5}/8}\new{+(1-e^{-(1+K_2+\gamma)T})^{\lfloor N^{1/2} \rfloor}}<\ep $$
by Lemma~\ref{lem:yellow} and our choice of $N$.
\end{proof}

It remains to prove item~3 of Lemma~\ref{lem:bad_blue}, which will follow easily from the following result.
\begin{lemma} \label{lem:BYspread}
For $N$ sufficiently large, for $r\in \N_0$ and $x\in \Z^d$, 
$$
\p{R^{\mathrm{spread}}_{x,r}}\vee \p{B^{\mathrm{spread}}_{x,r}}\le  \min( f(r)^2 , N^{-1/2})
\quad \text{ and } \quad 
\p{Y^{\mathrm{spread}}_{x,r}}\le  \min( f(r)^2 , N^{-1/4}).
$$
\end{lemma}
\begin{proof}
Take $r\in \N_0$ and $x\in \Z^d$.
Note first that since $|\ball_{(r/4)\vee 1}(x)|\ge 1$ and $g(r/4)^{-1}\ge 1$, we have $\p{B^{\mathrm{spread}}_{x,r}} \le \p{R^{\mathrm{spread}}_{x,r}}$. Furthermore, by a union bound,
\begin{align} \label{eq:Bspreadinit}
&\p{R^{\mathrm{spread}}_{x,r}}\notag \\
&\le 
\mathbb P\Big( \#\{(y,j,u):y\in \ball_{(r/4)\vee 1}(x), j\in \llbracket g(r/4)^{-1} K_0 N \rrbracket,u\in \mathcal N^{\mathrm{red},x,y,j}_T, \notag\\ &\hspace{8cm} \sup_{t\in [0,T]} \|X^{\mathrm{red},x,y,j}_t (u)\|\geq r/4 \}\geq e^{3T} g(r)K_0 N
\Big) \notag \\
&\quad + \p{\#\{(y,j,u):y\in \ball_{(r/4)\vee 1}(x), j\in \llbracket g(r/4)^{-1} K_0 N \rrbracket,u\in \mathcal N^{\mathrm{red},x,y,j}_T\}\ge e^{3T}g(r)^{-2}K_0 N}.
\end{align}
We begin by bounding the first term on the right hand side of~\eqref{eq:Bspreadinit}.
For $r\in \N$, by~\eqref{eq:lemmovefarst1} in Lemma~\ref{lem:movefar} with $n=|\ball_{(r/4)\vee 1}(0)|\lfloor g(r/4)^{-1} K_0 N \rfloor$, $t=T$ and $m=e^{3T} g(r) K_0 N$,
\begin{align} \label{eq:Bspreadinit2}
&\mathbb P\Big( \#\{(y,j,u):y\in \ball_{(r/4)\vee 1}(x), j\in \llbracket g(r/4)^{-1} K_0 N \rrbracket,u\in \mathcal N^{\mathrm{red},x,y,j}_T, \notag\\ &\hspace{8cm} \sup_{t\in [0,T]} \|X^{\mathrm{red},x,y,j}_t (u)\|\geq r/4 \}\geq e^{3T} g(r)K_0 N
\Big) \notag \\
&\quad \le (e^{3T} g(r) K_0 N)^{-1} |\ball_{(r/4)\vee 1}(0)|g(r/4)^{-1} K_0 N e^T e^{-\frac r {4R_1} \log \left( \frac r {4R_1 \gamma T e}\right)} \notag \\
&\quad =  |\ball_{(r/4)\vee 1}(0)|g(r/4)^{-1}e^{-2T} g(r)^{-1} e^{-\frac r {4R_1} \log \left( \frac r {4R_1 \gamma T e}\right)}.
\end{align}
We now bound the second term on the right hand side of~\eqref{eq:Bspreadinit}.
Let $s= e^{-T}/2$.
Recall from before~\eqref{eq:C1defn2} that we chose $T$ sufficiently large that $(1-e^{-T})e^s <1$.
Hence for any $r\in \N_0$, by~\eqref{eq:geom1} from Lemma~\ref{lem:geomsum} with $t=T$, $n =|\ball_{(r/4)\vee 1}(0)| \lfloor g(r/4)^{-1}K_0 N \rfloor$ and $m=e^{3T}g(r)^{-2} K_0 N$,
\begin{align} \label{eq:Bspreadgeom}
 &\p{\#\{(y,j,u):y\in \ball_{(r/4)\vee 1}(x), j\in \llbracket g(r/4)^{-1} K_0 N \rrbracket,u\in \mathcal N^{\mathrm{red},x,y,j}_T\}\ge e^{3T}g(r)^{-2}K_0 N} \notag \\
&\quad \le e^{- s e^{3T} g(r)^{-2}K_0 N } \left(\frac {e^s} {1-(1-e^{-T})e^s}\right)^{|\ball_{(r/4)\vee 1}(0)|g(r/4)^{-1} K_0 N} \notag \\
&\quad \le e^{-\frac 14 e^{2T} g(r)^{-2}K_0 N},
\end{align}
where the last line follows since
$\tfrac 14 e^{2T}g(r)^{-2}\ge |\ball_{(r/4)\vee 1}(0)|g(r/4)^{-1} \log (\tfrac{e^s}{1-(1-e^{-T})e^s})$ by~\eqref{eq:C1defn2} and $se^{3T}=\frac 12 e^{2T}$.
Hence for $r\in \N$ sufficiently large, by~\eqref{eq:Bspreadinit} and~\eqref{eq:Bspreadinit2},
and since $f(r)^2 =e^{-\frac{2r}{9R_1}\log (r+1)}$, we have
$$
\p{R^{\mathrm{spread}}_{x,r}}\vee \p{B^{\mathrm{spread}}_{x,r}} \le f(r)^2.
$$
We now establish another bound on the first term on the right hand side of~\eqref{eq:Bspreadinit}, which will be stronger than~\eqref{eq:Bspreadinit2} for small values of $r$.
By~\eqref{eq:C1defn}, we have
$$
|\ball_{r/4}(0)| g(r/4)^{-1} e^T \min\left(1, e^{-\frac r {4R_1} \log \left( \frac r {4R_1 \gamma T e}\right)}\right)\le \tfrac 12 e^{3T} g(r) \quad \forall r\ge 1 \quad \text{ and }\quad  e^T \le \tfrac 12 e^{3T}.
$$
Hence for $r\in \N_0$, by~\eqref{eq:lemmovefarst2} in Lemma~\ref{lem:movefar} with $n=|\ball_{(r/4)\vee 1}(0)|\lfloor g(r/4)^{-1} K_0 N \rfloor$, $t=T$ and $m=e^{3T}g(r)K_0 N$,
\begin{align*}
&\mathbb P\Big( \#\{(y,j,u):y\in \ball_{(r/4)\vee 1}(x), j\in \llbracket g(r/4)^{-1} K_0 N \rrbracket,u\in \mathcal N^{\mathrm{red},x,y,j}_T, \notag\\ &\hspace{8cm} \sup_{t\in [0,T]} \|X^{\mathrm{red},x,y,j}_t (u)\|\geq r/4 \}\geq e^{3T} g(r)K_0 N
\Big) \notag \\
&\quad \le 8 (e^{3T} g(r) K_0 N)^{-2}  |\ball_{(r/4)\vee 1}(0)| g(r/4)^{-1} K_0 N e^{2T} \\
&\quad = 8 e^{-4T} g(r)^{-2}|\ball_{(r/4)\vee 1}(0)| g(r/4)^{-1} (K_0 N)^{-1}.
\end{align*}
For $N$ sufficiently large, for $r\in \N_0$, if $f(r)\ge N^{-1/2}$ then $r\le \log N$ and so by~\eqref{eq:Bspreadinit} and~\eqref{eq:Bspreadgeom},
\begin{align*}
&\p{R^{\mathrm{spread}}_{x,r}}\vee \p{B^{\mathrm{spread}}_{x,r}}\\
&\le e^{-\frac 14  e^{2T}g(r)^{-2}K_0 N}+8 e^{-4T} K_0^{-1} g(\log N)^{-2} |\ball_{(\log N)/4}(0)| g((\log N)/4)^{-1} N^{-1} \le N^{-1/2}
\end{align*}
for $N$ sufficiently large.
The bound on $\p{Y^{\mathrm{spread}}_{x,r}}$ follows by the same argument.
\end{proof}

\begin{proof}[Proof of Lemma~\ref{lem:bad_blue} item 3]
Fix $\ep>0$; then for $x\in \Z^d$ and $r\in \N_0$, by a union bound,
\begin{equation} \label{eq:Pxrepsum}
\p{P_{x,r,\ep}}\le \p{Z^{(x)}_\ep >r/4}+\p{ R^{\mathrm{spread}}_{x,r} }+\p{ B^{\mathrm{spread}}_{x,r} }+ \p{Y^{\mathrm{spread}}_{x,r} }.
\end{equation}
By~\eqref{eq:Zepsdefn} we have that if $\ep$ is sufficiently small, for all $r\ge 0$,
$$
\p{Z^{(x)}_\ep >r/4}=\p{Z^{(x)}_\ep \ge \lfloor r/4 \rfloor +1}=\min(f(\lfloor r/4 \rfloor +1+\ep^{-1}),\ep)
\le \tfrac 12 \min\Big(2\ep, \frac{f(r+\ep ^{-1})^{1/5}}{|\ball_{r+3D/2}(0)|} \Big).
$$
Moreover, for $N$ sufficiently large, by Lemma~\ref{lem:BYspread},
\begin{align*}
\p{ R^{\mathrm{spread}}_{x,r} }+\p{ B^{\mathrm{spread}}_{x,r} }+ \p{Y^{\mathrm{spread}}_{x,r} }\le 3\min(f(r)^2 ,N^{-1/4})\le \tfrac 12 \min\Big(2\ep, \frac{f(r+\ep ^{-1})^{1/5}}{|\ball_{r+3D/2}(0)|} \Big)
\end{align*}
for $N$ sufficiently large.
By~\eqref{eq:Pxrepsum}, this completes the proof.
\end{proof}

\bibliographystyle{alpha}
\bibliography{2021_brw_nonlocal_competition}

\end{document}